\documentclass[12pt]{amsart}
\usepackage{rotating}
\usepackage{psfrag,epsfig,url}
\numberwithin{figure}{section}
\numberwithin{table}{section}
\newtheorem{theorem}{Theorem}[section]
\newtheorem{lemma}[theorem]{Lemma}
\newtheorem{prop}[theorem]{Proposition}

\theoremstyle{definition}
\newtheorem{definition}[theorem]{Definition}
\newtheorem{example}[theorem]{Example}
\newtheorem{cor}[theorem]{Corollary}

\theoremstyle{remark}
\newtheorem{remark}[theorem]{Remark}
\numberwithin{equation}{section}

\def \h{{\mathfrak h}}

\def \Q{{\mathbb Q}}

\def \A{{\mathfrak A}}
\def \[{[ }
\def \]{] }

\def\t{\widetilde}
\def\mr{\mathrm}
\def\h{\widehat}

\def \l{\langle }
\def \r{\rangle }
\def \B{{\mathcal B}}

\begin{document}

\title[Cluster algebras of finite mutation type via unfoldings]{Cluster algebras of finite mutation type\\ via unfoldings}
\author{Anna Felikson}
\address{Independent University of Moscow, B. Vlassievskii 11, 119002 Moscow, Russia}
\curraddr{School of Engineering and Science, Jacobs University Bremen, Campus Ring 1, D-28759, Germany}
\email{felikson@mccme.ru}
\thanks{Research of Michael Shapiro was supported by grants DNS 0800671 and PHY 0555346}

\author{Michael Shapiro}
\address{Department of Mathematics, Michigan State University, East Lansing, MI 48824, USA}
\email{mshapiro@math.msu.edu}

\author{Pavel Tumarkin}
\address{School of Engineering and Science, Jacobs University Bremen, Campus Ring 1, D-28759, Germany}
\email{p.tumarkin@jacobs-university.de}

\begin{abstract}
We complete classification of mutation-finite cluster algebras by extending the technique derived by Fomin, Shapiro, and Thurston to skew-symmetrizable case. We show that for every mutation-finite skew-symmetrizable matrix a diagram characterizing the matrix admits an {\it unfolding} which embeds its mutation class to the mutation class of some mutation-finite skew-symmetric matrix. In particular, this establishes a correspondence between a large class of skew-symmetrizable mutation-finite cluster algebras and triangulated marked bordered surfaces.
\end{abstract}
\maketitle

\setcounter{tocdepth}{1}

\tableofcontents

\section{Introduction}
\label{intro}
In the present paper, we continue investigation of cluster algebras of finite mutation type started in~\cite{FST1}.

Cluster algebras were introduced by Fomin and Zelevinsky in the series of papers~\cite{FZ1},~\cite{FZ2},~\cite{BFZ3},
~\cite{FZ4}. Up to isomorphism, each cluster algebra is defined by a \emph{skew-symmetrizable} $n\times n$ integer matrix called {\it exchange matrix}, where integer matrix $B$ is skew-symmetrizable if there exists an integer diagonal $n\times n$ matrix $D$ such that $BD$ is skew-symmetric.
Exchange matrices admit {\it mutations} (see~\ref{eq:MatrixMutation}).
Collection of all exchange matrices of a cluster algebra form a {\it mutation class} of exchange matrices.

In~\cite{FST1}, we classified all the \emph{skew-symmetric} exchange matrices with finite mutation class. In this paper, we complete classification of finite mutation classes of exchange matrices by presenting an answer in full generality.

The method we use is based on the following two main tools. The first main tool is the technique of {\it block decompositions} introduced by Fomin, Shapiro, and Thurston in~\cite{FST}. The results of~\cite{FST1} are primary based on application of this technique. We combine this technique with studying of {\it diagrams} associated to skew-symmetrizable matrices defined by Fomin and Zelevinsky in~\cite{FZ2} by introducing {\it s-decomposable diagrams}. The second main tool is a counterpart of the {\it unfolding procedure} introduced by Lusztig in~\cite{L} for generalized Cartan matrices. Using the unfolding procedure, we assign to each diagram of a mutation-finite skew-symmetrizable matrix a mutation-finite quiver. Due to results of~\cite{FST} and~\cite{FST1}, this allows us to relate a large class of skew-symmetrizable mutation-finite matrices with $2$-dimensional bordered marked surfaces.

We prove the following theorem (the precise definitions will be given in Sections~\ref{cluster} and~\ref{blockdecomp}).
 
\setcounter{section}{5}
\setcounter{theorem}{12}
%\begin{thrm}[Theorem~\ref{all-s}]
\begin{theorem}
\setcounter{section}{1}
\setcounter{theorem}{0}

A skew-symmetrizable $n\times n$ matrix, $n\ge 3$, that is not skew-symmetric, has finite mutation class if and only if its diagram is either s-decomposable or mutation-equivalent to one of the seven types $\t G_2$, $F_4$, $\t F_4$, $G_2^{(*,+)}$, $G_2^{(*,*)}$, $F_4^{(*,+)}$, $F_4^{(*,*)}$ shown on Fig.~\ref{allfign}.

\begin{figure}[!h]
\begin{center}
\psfrag{2}{\scriptsize $2$}
\psfrag{2-}{\scriptsize $2$}
\psfrag{3}{\scriptsize $3$}
\psfrag{4}{\scriptsize $4$}
\psfrag{G}{$\t G_2$}
\psfrag{F}{$F_4$}
\psfrag{Ft}{$\t F_4$}
\psfrag{W}{$G_2^{(*,*)}$}
\psfrag{V}{$G_2^{(*,+)}$}
%\psfrag{Y5}{$Y_5$}
\psfrag{Y6}{$F_4^{(*,+)}$}
\psfrag{Z}{$F_4^{(*,*)}$}
\epsfig{file=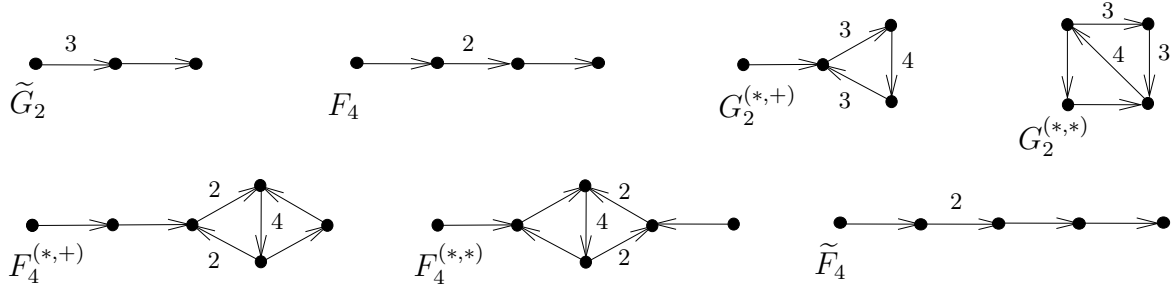,width=1.0\linewidth}
\caption{Non-decomposable mutation-finite non-skew-symmetric diagrams of order at least $3$}
\label{allfign}
\end{center}
\end{figure}

%\end{thrm}
\end{theorem}

\noindent
{\em Remark.\ }\  
The diagrams $G_2^{(*,+)}$, $G_2^{(*,*)}$, $F_4^{(*,+)}$, and $F_4^{(*,*)}$ are, actually, diagrams of extended affine root systems (see~\cite{Sa}). Each of them corresponds to two extended affine root systems: $G_2^{(*,+)}$ corresponds to root systems $G_2^{(1,3)}$ and $G_2^{(3,1)}$ (whose matrices are mutation-equivalent), $F_4^{(*,+)}$ corresponds to root systems $F_4^{(1,2)}$ and $F_4^{(2,1)}$ (whose matrices are also mutation-equivalent up to change of all signs), $G_2^{(*,*)}$ corresponds to root systems $G_2^{(1,1)}$ and $G_2^{(3,3)}$, and $F_4^{(*,*)}$ corresponds to root systems $F_4^{(1,1)}$ and $F_4^{(2,2)}$ (see Table~\ref{non} and~\cite[Table~1]{Sa}).

We recall that mutation class of any $2\times 2$ skew-symmetrizable matrix is finite.

Combined with results of~\cite{FST1}, Theorem~\ref{all-s} completes the classification of muta\-tion-finite skew-symmetrizable matrices.

Using Theorem~\ref{all-s}, we prove the following theorem.

\setcounter{section}{6}
\begin{theorem}
Any s-decomposable diagram admits an unfolding to a diagram arising from ideal tagged triangulation of a marked bordered surface.
Any mutation-finite matrix with non-decomposable diagram admits an unfolding to a mutation-finite skew-symmetric matrix. 
\end{theorem}
\setcounter{theorem}{0}
\setcounter{section}{1}

Tagged triangulations corresponding to unfoldings of skew-symmetrizable matrices with s-decomposable diagrams (constructed in Section~\ref{unfolding-st}) have special symmetry property: each of them contains a pair of edges representing the same isotopy class (one tagged plain and the other tagged notched, we call them \emph{conjugate pair} of edges). In particular, we obtain a correspondence between s-decomposable diagrams and marked tagged triangulations:

\setcounter{theorem}{1}
\setcounter{section}{7}
\begin{theorem}
There is a one-to-one correspondence between s-decomposable skew-symmetrizable diagrams with fixed block decomposition and ideal tagged triangulations of marked bordered surfaces with fixed tuple of conjugate pairs of edges.
\end{theorem}
\setcounter{theorem}{0}
\setcounter{section}{1}
In the correspondence above, one direction is provided by \emph{local unfoldings} (see Section~\ref{unfolding-st}). The other direction is provided by \emph{folding} (see Section~\ref{triangle}) of some of conjugate pairs of edges: due to the existence of unfolding, this operation occurs to be well-defined. Under this correspondence, block-decomposable diagrams correspond to triangulations with no conjugate pairs chosen. 

Note also that the correspondence above is invariant under mutations (resp., composite flips): if the triangulation $T(S)$ corresponds to a diagram $S$, then the triangulation $T(\mu_x(S))$ for a mutation $\mu_x(S)$ of a diagram $S$ in the vertex $x$ can be obtained by performing flips in all the edges of $T(S)$  corresponding to images of $x$ under local unfolding.

\smallskip

As in the skew-symmetric case (cf.~\cite[Theorem~7.5]{FST1}), consideration of minimal mutation-infinite diagrams gives rise to a polynomial-time algorithm to determine whether a large skew-symmetrizable matrix is mutation-finite:

\setcounter{theorem}{4}
\setcounter{section}{8}
\begin{theorem}
A skew-symmetrizable $n\times n$ matrix $B$, $n\ge 10$, has finite mutation class if and only if a mutation class of every principal $10\times 10$ submatrix of $B$ is finite.

\end{theorem}
\setcounter{theorem}{0}
\setcounter{section}{1}

\medskip

The paper is organized as follows. In Section~\ref{cluster}, we recall necessary definitions and basic facts on cluster algebras, exchange matrices, and their diagrams.

%In Section~\ref{main}, we present the sketch of the proof of the Main Theorem. We list all the key steps, and discuss the main combinatorial and computational ideas we use. Sections~\ref{blockdecomp}--\ref{class} contain the detailed proofs.

Section~\ref{blockdecomp} is devoted to the technique of s-decomposable diagrams. We recall the basic facts from~\cite{FST}, and reformulate the results of~\cite{FST} in the language of diagrams. Further, we introduce new blocks and prove several properties of block decompositions of diagrams. In particular, we show that s-decomposable diagrams are mutation-finite.

In Section~\ref{unfolding-s} we give a definition of unfolding of skew-symmetrizable matrices introduced by A.~Zelevinsky (personal communication), and extend it to a notion of unfolding of a diagram. This is the core construction of the paper. In general, an unfolding may not be unique. We construct a uniquely defined {\it local unfolding} for any s-decomposable diagram. Making use of this construction, we show that s-decomposable diagrams carry the same properties as block-decomposable quivers do.

Section~\ref{minimal} contains the proof of Theorem~\ref{all-s}. In Section~\ref{unfoldings}, we present a construction of unfolding for non-decomposable mutation-finite skew-symmetrizable matrices.

Section~\ref{triangle} is devoted to applications of the results of Section~\ref{unfoldings} to construction of relations between s-decomposable diagrams and triangulations of bordered surfaces.  

Finally, in Section~\ref{inf} we provide a polynomial-time algorithm which determines whether a skew-symmetrizable matrix has finite mutation class.

\medskip

We would like to thank B.~Keller who attracted our attention to foldings, and V.~Fock, A.~Goncharov, and S.~Fomin for fruitful discussions and advices. We are especially grateful to A.~Zelevinsky for introduction to unfoldings and numerous stimulating discussions leading to appearing of the present paper.
The first author thanks the Max Planck Institute for Mathematics in Bonn for hospitality.

\section{Cluster algebras, mutations, and diagrams}
\label{cluster}

\noindent
We briefly remind the definition of coefficient-free cluster algebra.

An integer $n\times n$ matrix $B$ is called \emph{skew-symmetrizable} if there exists an
integer diagonal $n\times n$ matrix $D=diag(d_1,\dots,d_n)$,
%\begin{pmatrix}
%                                         d_1 &  & 0 \\
%                                          & \ddots &  \\
%                                         0 &  & d_n \\
%                                       \end{pmatrix}$,
such that the product $BD$ is a skew-symmetric matrix, i.e.,
                                       $b_{ij}d_j=-b_{ji}d_i$.

\emph{A seed} is a pair $(f,B)$, where $f=\{f_1,\dots,f_n\}$ form a collection of algebraically independent rational functions of $n$ variables
$x_1,\dots,x_n$, and $B$ is a skew-symmetrizable matrix.

The part $f$ of seed $(f,B)$ is called \emph{cluster}, elements $f_i$ are called \emph{cluster variables},
and $B$ is called \emph{exchange matrix}.

\begin{definition}
For any $k$, $1\le k\le n$ we define \emph{the mutation} of seed $(f,B)$ in direction $k$
as a new seed $(f',B')$ in the following way:
\begin{equation}\label{eq:MatrixMutation}
b'_{ij}=\left\{
           \begin{array}{ll}
             -b_{ij}, & \hbox{ if } i=k \hbox{ or } j=k; \\
             b_{ij}+\frac{|b_{ik}|b_{kj}+b_{ik}|b_{kj}|}{2}, & \hbox{ otherwise.}
           \end{array}
         \right.
\end{equation}

\begin{equation}\label{eq:ClusterMutation}
f'_i=\left\{
           \begin{array}{ll}
             f_i, & \hbox{ if } i\ne k; \\
             \frac{\prod_{b_{ji}>0} f_j^{b_{ji}}+\prod_{b_{ji}<0} f_j^{-b_{ji}}}{f_i}, & \hbox{ otherwise.}
           \end{array}
         \right.
\end{equation}
\end{definition}

\noindent
We write $(f',B')=\mu_k\left((f,B)\right)$.
Notice that $\mu_k(\mu_k((f,B)))=(f,B)$.
We say that two seeds are \emph{mutation-equivalent}
if one is obtained from the other by a sequence of seed mutations.
Similarly we say that two clusters or two exchange matrices are \emph{mutation-equivalent}.

Notice that exchange matrix mutation~(\ref{eq:MatrixMutation}) depends only on the exchange matrix itself.
The collection of all matrices mutation-equivalent to a given matrix $B$ is called the \emph{mutation class} of $B$.

For any skew-symmetrizable matrix $B$ we define \emph{initial seed} $(x,\!B)$ as
$(\!\{x_1,\dots,x_n\}\!,\!B)$, $B$ is the \emph{initial exchange matrix}, $x=\{x_1,\dots,x_n\}$ is the \emph{initial cluster}.

{\it Cluster algebra} $\A(B)$ associated with the skew-sym\-met\-ri\-zab\-le $n\times n$ matrix $B$ is a subalgebra of $\Q(x_1,\dots,x_n)$ generated by all cluster variables of the clusters mutation-equivalent
to the initial seed $(x,B)$.

Cluster algebra $\A(B)$ is called \emph{of finite type} if it contains only finitely many
cluster variables. In other words, all clusters mutation-equivalent to initial cluster contain
totally only finitely many distinct cluster variables.

In~\cite{FZ2}, Fomin and Zelevinsky proved a remarkable theorem that cluster algebras of finite type
can be completely classified. More excitingly, this classification is parallel to the famous Cartan-Killing classification
of simple Lie algebras.

%Namely, the following theorem holds.

Let $B$ be an integer $n\times n$ matrix. Its \emph{Cartan companion} $C(B)$ is the integer $n\times n$ matrix defined as follows:

\begin{equation*}
  C(B)_{ij}=\left\{
              \begin{array}{ll}
                2, & \hbox{ if } i=j; \\
                 -|b_{ij}|, & \hbox{ otherwise.}
              \end{array}
            \right.
\end{equation*}

\begin{theorem}[\cite{FZ2}]
\label{thm:FinTypeClass}
There is a canonical bijection between the Cartan matrices of
finite type and cluster algebras of finite
type. Under this bijection, a Cartan matrix $A$ of finite type corresponds to the cluster algebra
$\A(B)$, where $B$ is an arbitrary skew-symmetrizable matrix with $C(B) = A$.
\end{theorem}

The results by Fomin and Zelevinsky were further developed in~\cite{S} and~\cite{BGZ}, where the effective criteria for cluster algebras of finite type were given.

A cluster algebra of finite type has only finitely many distinct seeds.
Therefore, any cluster algebra that has only finitely many cluster variables contains only finitely many
distinct exchange matrices. Quite the contrary, the cluster algebra with finitely many exchange matrices
is not necessarily of finite type.

\begin{definition}\label{def:FinMutType} A cluster algebra with only finitely many exchange matrices is called \emph{of finite mutation type}.
\end{definition}

\begin{example} The easiest example of infinite cluster algebra of finite mutation type is the algebra whose
exchange matrix is
\begin{center}
$\begin{pmatrix}
   0 & 2  \\
   -2 & 0
 \end{pmatrix}$
 \end{center}
 This cluster algebra is not of finite type, however, mutation in any direction leads simply to sign change of exchange matrix.
 Therefore, the algebra is clearly of finite mutation type.
\end{example}

%\begin{example} One example of infinite cluster algebra of finite mutation type is the Markov cluster algebra whose
%exchange matrix is
%\begin{center}
%$\begin{pmatrix}
%   0 & 2 & -2 \\
%   -2 & 0 & 2 \\
%   2 & -2 & 0 \\
% \end{pmatrix}$
% \end{center}
% It was described in details in~\cite{FZ1}. Markov cluster algebra is not of finite type, moreover, it is even not finitely generated.
% Notice, however, that mutation in any direction leads simply to sign change of exchange matrix.
% Therefore, the Markov cluster algebra is clearly of finite mutation type.
%\end{example}

\begin{remark} Since the orbit of an exchange matrix depends on the exchange matrix only, we may speak about skew-symmetrizable matrices of finite mutation type.
\end{remark}

Therefore, Theorem~\ref{all-s} describes  \emph{all skew-symmetrizable integer matrices whose mutation class is finite}.

% Clearly that any finite type cluster algebra is of finite mutation type. However, there are examples of infinite type cluster algebras that are of finite mutation type.

%{\bf [SKEW-SYMMETRIC MATRICES -> ORIENTED GRAPHS]}

Following~\cite{FZ2}, we encode an $n\times n$ skew-symmetrizable integer matrix $B$ by a finite simplicial $1$-complex $S$ with oriented weighted edges called {\it diagram}. The weights of a diagram are positive integers.

Vertices of $S$ are labeled by $[1,\dots,n]$. If $b_{ij}>0$, we join vertices $i$ and $j$ by an  edge directed from $i$ to $j$ and assign to this edge weight $-b_{ij}b_{ji}$. Not every diagram corresponds to a skew-symmetrizable integer matrix: given a diagram $S$ of a skew-symmetrizable integer matrix $B$, a product of weights along any chordless cycle of $S$ is a perfect square (cf.~\cite[Exercise~2.1]{Kac}).

Distinct matrices may have the same diagram. At the same time, it is easy to see that only finitely many matrices may correspond to the same diagram.
All weights of a diagram of a skew-symmetric matrix are perfect squares. Conversely, if all weights of a diagram $S$ are perfect squares, then there exists a skew-symmetric matrix $B$ with diagram $S$.

As it is shown in~\cite{FZ2}, mutations of exchange matrices induce {\it mutations of diagrams}. If $S$ is the diagram corresponding to matrix $B$, and $B'$ is a mutation of $B$ in direction $k$, then we call the diagram $S'$ associated to $B'$ a {\it mutation of $S$ in direction $k$} and denote it by $\mu_k(S)$. A mutation in direction $k$ changes weights of diagram in the way described in Figure~\ref{quivermut} (see~\cite{FZ2}).

\begin{figure}[!h]
\begin{center}
\psfrag{a}{\small $a$}
\psfrag{b}{\small $b$}
\psfrag{c}{\small $c$}
\psfrag{d}{\small $d$}
\psfrag{k}{\small $k$}
\psfrag{mu}{\small $\mu_k$}
\epsfig{file=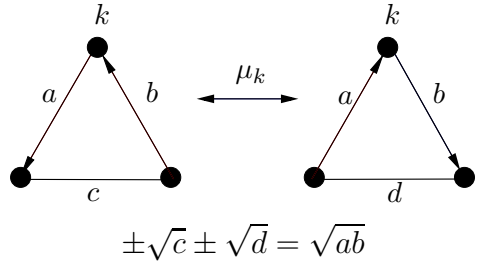,width=0.4\linewidth}\\
\medskip
$\pm\sqrt{c}\pm\sqrt{d}=\sqrt{ab}$
\caption{Mutations of diagrams. The sign before $\sqrt{c}$ (resp., $\sqrt{d}$) is positive if the three vertices form an oriented cycle, and negative otherwise. Either $c$ or $d$ may vanish. If $ab$ is equal to zero then neither value of $c$ nor orientation of the corresponding edge does change.}
\label{quivermut}

\end{center}
\end{figure}

For given diagram, the notion of {\it mutation class} is well-defined. We call a diagram (resp., matrix) {\it mutation-finite} if its mutation class is finite.
\begin{remark}
\label{d-m}
Note that the order of mutation class of a matrix may differ from the order of mutation class of corresponding diagram (see Example~\ref{diff} below). However, mutation class of a matrix is finite if and only if a mutation class of the corresponding diagram is finite.
\end{remark}

\begin{example}
\label{diff}
The mutation class of the following matrix
\begin{center}
$\begin{pmatrix}
   0 & 2 & -4 \\
   -1 & 0 & 2 \\
   1 & -1 & 0 \\
 \end{pmatrix}$
  \end{center}
consists of $6$ matrices (up to simultaneous permutations of rows and columns). At the same time, the mutation class of the corresponding diagram  contains $4$ diagrams only.

\end{example}

Due to Remark~\ref{d-m}, we can reduce the problem of classification of exchange matrices of finite mutation type to the following: \emph{find all mutation-finite diagrams}.

The following criterion for a diagram to be mutation-finite is well-known. We present a short proof for the convenience of the reader.

\begin{theorem}
\label{less3}
A connected diagram $S$ of order at least $3$ is mutation-finite if and only if any diagram in the mutation class of $S$ contains no edges of weight greater than $4$.
\end{theorem}

\begin{proof}
The sufficiency is evident. To prove the necessity, it is sufficient to show that any connected diagram of order $3$ containing an edge of weight at least $5$ is mutation-infinite. For that we show that, in the assumptions above, there always exists a sequence of at most two mutations increasing the sum of the three weights (we call this sum {\it total weight}) and preserving the maximal weight.

Let $S$ be a diagram of order $3$ with weights $(a,b,c)$, $a\ge b\ge c$, $a\ge 5$. If $S$ is cyclically oriented (i.e., $S$ is an oriented cycle), then mutating in the common vertex of edges with weights $a$ and $b$ we get a triple $(a,b,(\sqrt{ab}-\sqrt{c})^2)$, which has larger total weight since $a\ge b\ge c$ and $a\ge 5$ imply $(\sqrt{ab}-\sqrt{c})^2>c$.

Now let $S$ be not cyclically oriented. Applying one mutation (without changing weights) if needed, we may assume that the edges with weights $a$ and $b$ are oriented in the same way. Mutating in their common vertex, we get a triple $(a,b,(\sqrt{ab}+\sqrt{c})^2)$ which clearly has larger total weight than the initial triple did.

\end{proof}

\begin{remark}
The case of {\it mutation-acyclic} diagrams was treated by Seven in~\cite{S1}: it is proved there that mutation class of a mutation-finite diagram $S$ contains a diagram without oriented cycles if and only if $S$ is mutation equivalent to orientation of Dynkin (or extended Dynkin) diagram.    

\end{remark}

\medskip

%\noindent
From now on, we use language of diagrams. The following notation will be used throughout the paper.

Let $S$ be a diagram. A {\it subdiagram} $S_1\subset S$ is a subcomplex of $S$. The {\it order}  $|S|$ is the number of vertices of diagram $S$.
If $S_1$ and $S_2$ are subdiagrams of diagram $S$, we denote by $\l S_1,S_2\r$ the subdiagram of $S$ spanned by all the vertices of $S_1$ and $S_2$.

An edge is called {\it simple} if its weight is equal to one, and {\it multiple} otherwise.

\section{Block decompositions of diagrams}
\label{blockdecomp}

First, we rephrase the definition~3.1 from~\cite{FST} in terms of diagrams.

In~\cite{FST}, a {\it block} is a diagram isomorphic to one of the diagrams with black/white colored vertices shown on Fig.~\ref{bloki}, or to a single vertex. Vertices marked in white are called {\it outlets}, we call the remaining ones {\it dead ends}. A connected diagram $S$ is called {\it block-decomposable} if it can be obtained from a collection of blocks by identifying outlets of different blocks along some partial matching (matching of outlets of the same block is not allowed), where two simple edges with same endpoints and opposite directions cancel out, and two simple edges with same endpoints and same directions form an edge of weight $4$. A non-connected diagram $S$ is called  block-decomposable either if $S$ satisfies the definition above, or if $S$ is a disjoint union of several mutually orthogonal diagrams satisfying the definition above.
If $S$ is not block-decomposable then we call $S$ {\it non-decomposable}. Depending on a block, we call it {\it a block of type} $\rm{I}$, $\rm{II}$, $\rm{III}$, $\rm{IV}$, $\rm{V}$, or simply {\it a block of $n$-th type}.

\begin{figure}[!h]
\begin{center}
\psfrag{1}{${\rm{I}}$}
\psfrag{2}{${\rm{II}}$}
\psfrag{3a}{${\rm{IIIa}}$}
\psfrag{3b}{${\rm{IIIb}}$}
\psfrag{4}{${\rm{IV}}$}
\psfrag{5}{${\rm{V}}$}
\epsfig{file=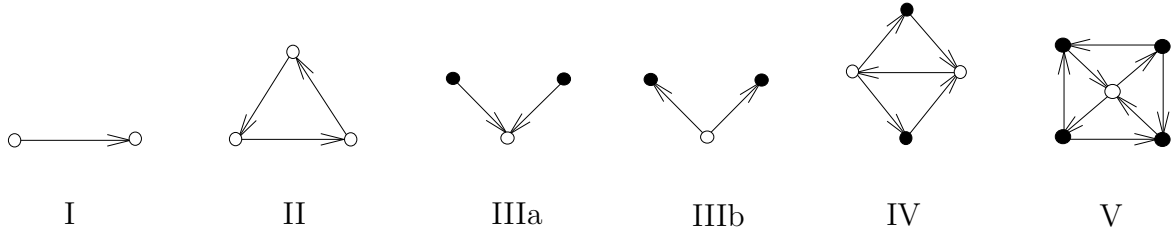,width=0.999\linewidth}
\caption{Blocks. Outlets are colored in white, dead ends are black.}
\label{bloki}
\end{center}
\end{figure}

%\begin{remark}
Block-decomposable diagrams are in one-to-one correspondence with adjacency matrices of arcs of ideal (tagged) triangulations of bordered two-dimensional surfaces with marked points (see~\cite[Section~13]{FST} for the detailed explanations). Mutations of block-decomposable diagrams correspond to flips of triangulations. In particular, this implies that mutation class of any block-decomposable diagram is finite, and any subdiagram of a block-decomposable one is block-decomposable too.

\medskip

Clearly, adjacency matrices of arcs of ideal triangulations are skew-symmetric. To adopt the technique of blocks to general (skew-symmetrizable) case, we introduce new blocks of types $\mr{\t{III}a}$, $\mr{\t{III}b}$, $\t{\mr{IV}}$, $\t{\mr V}_1$, $\t{\mr V}_2$, $\t{\mr V}_{12}$, and $\t{\mr{VI}}$ shown in Table~\ref{newblocks}.

\begin{table}[!h]
\caption{New blocks and their local unfoldings (see Sections~\ref{unfolding-s},~\ref{unfolding-st}). Vertex $x_i$ and the set $E_i$ are marked in the same way.}
\label{newblocks}
\begin{tabular}{|l|c|c|c|c|c|c|c|}
\hline
&&&&&&&\\
\begin{tabular}{c}
\raisebox{-0.4cm}{New blocks}
\end{tabular}
&
%\begin{figure}
\psfrag{2-}{\tiny $2$}
\psfrag{3at}{\scriptsize ${\mr{\t{III}a}}$}
\parbox[c]{1.4cm}{\epsfig{file=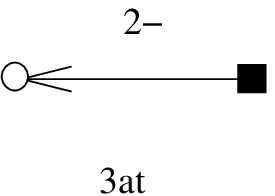,width=0.99\linewidth}}
%\end{figure}
&
%\begin{figure}
\psfrag{2-}{\tiny $2$}
\psfrag{3bt}{\scriptsize ${\mr{\t{III}b}}$}
\parbox[c]{1.4cm}{\epsfig{file=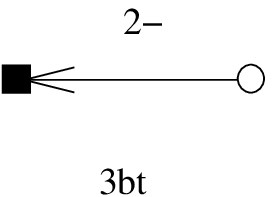,width=0.99\linewidth}}
%\end{figure}
&
%\begin{figure}[!h]
\psfrag{2-}{\tiny $2$}
\psfrag{4t}{\scriptsize $\t{\mr{IV}}$}
\parbox[c]{1.4cm}{\epsfig{file=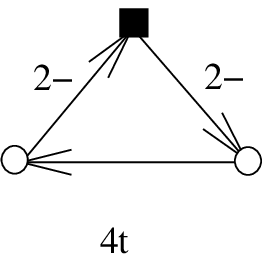,width=0.99\linewidth}}
%\end{figure}
&
%\begin{figure}[!h]
\psfrag{2-}{\tiny $2$}
\psfrag{51t}{\scriptsize $\t{\mr{V}}_1$}
\parbox[c]{1.4cm}{\epsfig{file=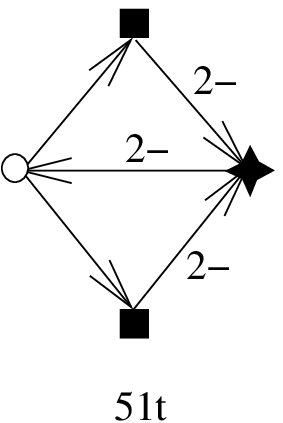,width=0.99\linewidth}}
%\end{figure}
&
%\begin{figure}[!h]
\psfrag{2-}{\tiny $2$}
\psfrag{52t}{\scriptsize $\t{\mr{V}}_2$}
\parbox[c]{1.4cm}{\epsfig{file=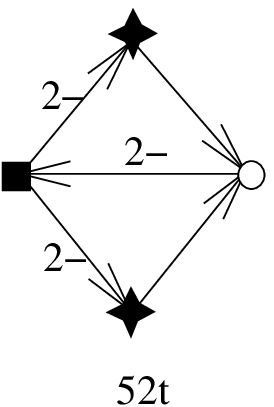,width=0.99\linewidth}}
%\end{figure}
&
%\begin{figure}[!h]
\psfrag{2-}{\tiny $2$}
\psfrag{4}{\tiny $4$}
\psfrag{512t}{\scriptsize $\t{\mr{V}}_{12}$}
\parbox[c]{1.4cm}{\epsfig{file=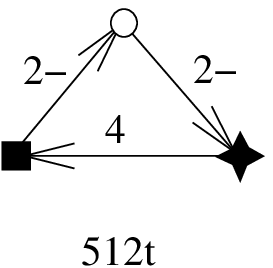,width=0.99\linewidth}}
%\end{figure}
&
%\begin{figure}[!h]
\psfrag{2-}{\tiny $2$}
\psfrag{6t}{\scriptsize $\t{\mr{VI}}$}
\parbox[c]{1.4cm}{\epsfig{file=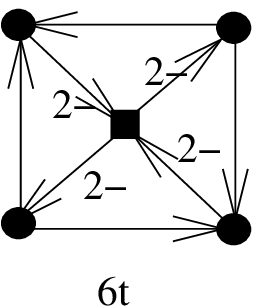,width=0.99\linewidth}}
%\end{figure}
\\
&&&&&&&
\\
\hline
&&&&\multicolumn{3}{c|}{}&\multicolumn{1}{c|}{}\\
\begin{tabular}{c}
Unfoldings
\end{tabular}
&
%\begin{figure}[!h]
%\psfrag{3a}{\scriptsize $\widetilde{V}_1$}
\parbox[c]{1.4cm}{\epsfig{file=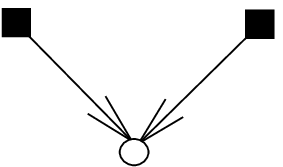,width=0.99\linewidth}}
%\end{figure}
&
%\begin{figure}[!h]
%\psfrag{3a}{\scriptsize $\widetilde{V}_1$}
\parbox[c]{1.4cm}{\epsfig{file=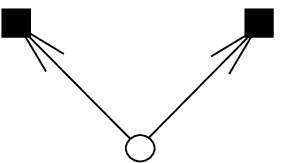,width=0.99\linewidth}}
%\end{figure}
&
%\begin{figure}[!h]
%\psfrag{3a}{\scriptsize $\widetilde{V}_1$}
\parbox[c]{1.4cm}{\epsfig{file=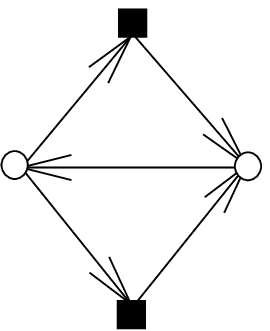,width=0.99\linewidth}}
%\end{figure}
&
\multicolumn{3}{c|}{
%\begin{figure}[!h]
%\psfrag{3a}{\scriptsize $\widetilde{V}_1$}
\parbox[c]{1.4cm}{\epsfig{file=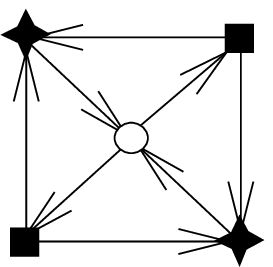,width=0.99\linewidth}}
%\end{figure}
}
&
\parbox[c]{1.4cm}{\epsfig{file=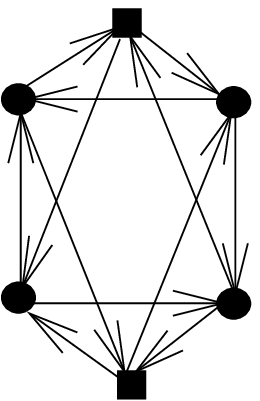,width=0.99\linewidth}}
%\begin{figure}[!h]
%\psfrag{3a}{\scriptsize $\widetilde{V}_1$}
%\epsfig{file=diagrams_pic/block6.eps,width=0.08\linewidth}
%\end{figure}
\\
&&&&\multicolumn{3}{c|}{}&\multicolumn{1}{c|}{}\\
\hline

\end{tabular}
\end{table}

Again, outlets are marked white. We keep the way of gluing (this remains well-defined since any edge with two outlets as ends is simple). More precisely, gluing of two edges of weight one will result in either empty edge (in case of distinct orientations) or an edge with weight $4$. 
\begin{definition}
A diagram is {\it s-decomposable} if it can be glued from blocks (both old and new). 
\end{definition}
We keep the term ``block-decomposable'' for s-decomposable diagrams corresponding to skew-symmetric matrices.

Our aim is to prove that s-decomposable diagrams satisfy the same properties as block-decomposable ones do. In particular, in Theorem~\ref{invariant} we show that the set of s-decomposable diagrams is invariant under mutations (which implies that they are mutation-finite). In the next section we prove that any subdiagram of s-decomposable diagrams is s-decomposable (see Corollary~\ref{sub}).

%While drawing diagrams, we keep the following notation:
%
%\begin{itemize}
%\item
%a non-oriented edge is used when orientation does not play any role in the proof;
%
%\item
%
%\psfrag{u}{\tiny $u$}
%\psfrag{v}{\tiny $v$}
%an edge \ \ \epsfig{file=diagrams_pic/dot_uv.eps,width=0.15\linewidth} \ \ is an edge of a block containing $u$ and $v$, where $u$ and $v$ are not joined in the diagram; the figure assumes a fixed block decomposition.
%\
%\item
%\psfrag{x}{\tiny $x$}
%\psfrag{a}{\tiny $a$}
%an edge \ \ \epsfig{file=diagrams_pic/dotdif_xa.eps,width=0.15\linewidth} \ \ means that $x$ is joined with $a$ by some edge.

%\end{itemize}

%%%%%%%%%%%%%%%%%%%

Let $S$ be an s-decomposable diagram with fixed decomposition (we denote this by $S_{\mr{dec}}$). We say that $x\in S_{\mr{dec}}$ is an {\it outlet} if $x$ is contained in exactly one block, and $x$ is an outlet in that block. Further, suppose that for some $y\in S_{\mr{dec}}$ the diagram  $\mu_y(S)$ is s-decomposable. Then a block decomposition $\mu_y(S)_{\mr{dec}}$ of
$\mu_y(S)$ is {\it $y$-good} if all outlets of $S_{\mr{dec}}$ (probably, except $y$ itself) are outlets of $\mu_y(S)_{\mr{dec}}$.

If $S$ is s-decomposable and a decomposition is fixed, we define $N_x(S_{\mr{dec}})$ to be the union of all blocks containing $x$.  Note that $N_x(S_{\mr{dec}})$ may not be a subdiagram of $S$.

\begin{lemma}
\label{good}

Let $S_{\mr{dec}}$ coincide with $N_x(S_{\mr{dec}})$ (i.e. $S_{\mr{dec}}$ is composed of blocks ${\mathsf{B}}_1$ and ${\mathsf{B}}_2$, ${\mathsf{B}}_2$ may be empty), $x\in S$, where $x\in {\mathsf{B}}_1\cap {\mathsf{B}}_2$ if ${\mathsf{B}}_2\ne\emptyset$.
%Suppose also that $S_{\mr{dec}}$ is different from $\Sigma_{\mr{dec}}$ shown on Fig.~\ref{exclusion}.
Then there exists an $x$-good block decomposition of $\mu_x(S)$.
\end{lemma}

Proof is straightforward: we need to examine $49$ diagrams of gluings of two blocks.

\begin{example}
\label{ex-good}
We illustrate the proof of lemma~\ref{good} on one example shown on Fig.~\ref{fig-good}, left. Here ${\mathsf{B}}_1$ is of type ${\mr{II}}$, and ${\mathsf{B}}_2$ is of type $\t{\mr{IV}}$. Outlets of $S_{\mr{dec}}$ are $y_1$, $y_2$, and $y_3$.
\begin{figure}[!h]
\begin{center}
\psfrag{2-}{\tiny $2$}
\psfrag{x}{\scriptsize $x$}
\psfrag{1}{\scriptsize $y_1$}
\psfrag{2}{\scriptsize $y_2$}
\psfrag{3}{\scriptsize $y_3$}
$S_{\mr{dec}}\quad$\raisebox{5mm}{\parbox[c]{4.0cm}{\epsfig{file=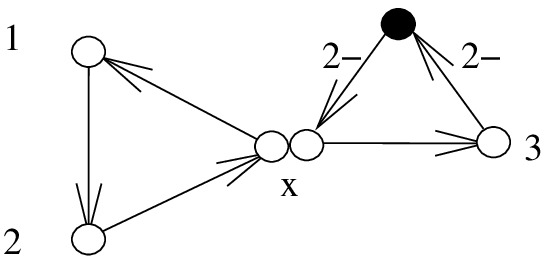,width=1.0\linewidth}}}\qquad\qquad\qquad
${\mu_x(S)}_{\mr{dec}}\quad$\raisebox{5mm}{\parbox[c]{3.6cm}{\epsfig{file=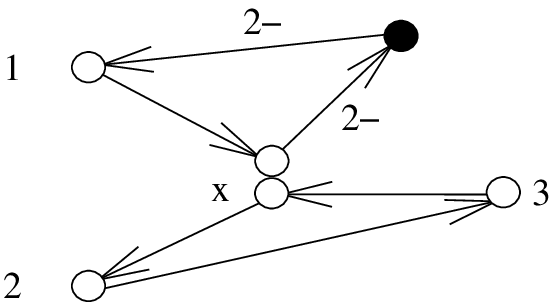,width=1.0\linewidth}}}
\caption{}
\label{fig-good}
\end{center}
\end{figure}

Then $\mu_x(S)$ has a block decomposition $\mu_x(S)_{\mr{dec}}$ shown on Fig.~\ref{fig-good}, right. Clearly, the vertices $y_1$, $y_2$, and $y_3$ are outlets of $\mu_x(S)_{\mr{dec}}$, so the decomposition is $x$-good.

\end{example}

%\begin{figure}[!h]
%\begin{center}
%\psfrag{2-}{\tiny $2$}
%$\Sigma_{\mr{dec}}=$ \parbox[c]{1.8cm}{\epsfig{file=diagrams_pic/excl1.eps,width=1.0\linewidth}}
%\caption{}
%\label{exclusion}
%\end{center}
%\end{figure}

\begin{lemma}
\label{glue_in}
Suppose $N_x(S_{\mr{dec}})=\l {\mathsf{B}}_1,{\mathsf{B}}_2\r$, ${\mathsf{B}}_2$ may be empty. Let $x_1,x_2$ be outlets of $N_x(S_{\mr{dec}})$ ($x_1,x_2\ne x$). Suppose also that $S_{\mr{dec}}$ consists of $N_x(S_{\mr{dec}})$ and a block ${\mathsf{B}}$, where $x_1$ and $x_2$ are outlets of ${\mathsf{B}}$. Then $\mu_x(S)$ is s-decomposable with block ${\mathsf{B}}$, i.e.  $$\l \mu_x(N_x(S_{\mr{dec}})), {\mathsf{B}}\r=\mu_x(\l N_x(S_{\mr{dec}})), {\mathsf{B}}\r)$$

\end{lemma}

The s-decomposability immediately follows from Lemma~\ref{good}. The equality follows from the definition of mutation, see Fig.~\ref{quivermut}.

As a corollary, we get the following theorem.

\begin{theorem}
\label{invariant}
Let $S$ be s-decomposable.
%and pair $(x, S_{\mr{dec}})$ does not coincide with $(x,\Sigma_{\mr{dec}})$ for $x$ belonging to intersection of the two blocks.
Then any mutation of $S$ is s-decomposable.
\end{theorem}

Proof follows from Lemma~\ref{glue_in}. Indeed, given decomposition of $S$ and $x\in S$, $\mu_x$ affects only $N_x(S_{\mr{dec}})$ and blocks with at least two points in common with $N_x(S_{\mr{dec}})$. According to Lemma~\ref{good}, $\mu_x(N_x(S_{\mr{dec}}))$ admits $x$-good decomposition. By Lemma~\ref{glue_in}, we can construct a decomposition of $\mu_x(S)$ by attaching to $x$-good decomposition of $\mu_x(N_x(S_{\mr{dec}}))$ the same blocks as in $S_{\mr{dec}}$ in the same way.

\begin{cor}
\label{mut-fin}
All s-decomposable diagrams are mutation-finite.
\end{cor}

%For all diagrams which are not mutation-equivalent to $\Sigma$ this follows from from Theorem~\ref{invariant}. Mutation class of $\Sigma$ contains $4$ diagrams only.

%\begin{remark}
%\label{block6}
%Mutation class of the diagram $\Sigma$ shown on Fig.~\ref{exclusion} consists of $3$ s-decomposable diagrams and one mentioned in Table~\ref{newblocks} as block~${\t{\mr{VI}}}$. This one is not s-decomposable, however, it has no outlets. If we add it to the list of blocks, then Lemma~\ref{glue_in} and Theorem~\ref{invariant} (as well as Lemma~\ref{good}) will hold without any exclusions. The proof is not changed at all.
%\end{remark}

\begin{remark}
\label{block6}
As one can notice, the block~${\t{\mr{VI}}}$ has no outlets. However, it is essential: its mutation class consists of $4$ diagrams, $3$ of them are s-decomposable (without making use of block~${\t{\mr{VI}}}$), and the fourth one is block~${\t{\mr{VI}}}$ itself (which cannot be decomposed in any other way).
\end{remark}

\section{Unfoldings of matrices and diagrams}
\label{unfolding-s}

Let $B$ be an indecomposable $n\times n$ skew-symmetrizable integer matrix, and let $BD$ be a skew-symmetric matrix, where $D=(d_{i})$ is diagonal integer matrix with positive diagonal entries. Notice that for any matrix $\mu_i(B)$ the matrix $\mu_i(B)D$ will be skew-symmetric.

We use the following definition of unfolding of a skew-symmetrizable matrix (communicated to us by A.~Zelevinsky).

Suppose that we have chosen disjoint index sets $E_1,\dots, E_n$ with $|E_i| =d_i$. Denote $m=\sum\limits_{i=1}^n d_i$.
Suppose also that we choose a skew-symmetric integer matrix $C$ of size $m\times m$ with rows and columns indexed by the union of all $E_i$, such that

(1) the sum of entries in each column of each $E_i \times E_j$ block of $C$ equals $b_{ij}$;

(2) if $b_{ij} \geq 0$ then the $E_i \times E_j$ block of $C$ has all entries non-negative.

Define a {\it composite mutation} $\h\mu_i = \prod_{\hat\imath \in E_i} \mu_{\hat\imath}$ on $C$. This mutation is well-defined, since all the mutations  $\mu_{\hat\imath}$, $\hat\imath\in E_i$, for given $i$ commute.

We say that $C$ is an {\it unfolding} for $B$ if $C$ satisfies assertions $(1)$ and $(2)$ above, and for any sequence of iterated mutations $\mu_{k_1}\dots\mu_{k_m}(B)$ the matrix $C'=\h\mu_{k_1}\dots\h\mu_{k_m}(C)$ satisfies assertions $(1)$ and $(2)$ with respect to $B'=\mu_{k_1}\dots\mu_{k_m}(B)$.

\begin{example}
The matrix $C$ below is an unfolding for the matrix $B$. Here $d_1=1$, $d_2=2$, $E_1=\{1\}$, $E_2=\{2,3\}$.
$$
B=\begin{pmatrix}
0&-1\\
2&0
\end{pmatrix}
\qquad\qquad
C=\begin{pmatrix}
0&-1&-1\\
1&0&0\\
1&0&0
\end{pmatrix}
$$
\end{example}

\begin{example}
\label{good-unf}
The matrices $B$ and $C$ below satisfy the assertions (1) and (2) of the definition of the unfolding. Here $d_1=2$, $d_2=1$, $d_3=2$, $E_1=\{1,2\}$, $E_2=\{3\}$, $E_3=\{4,5\}$.
$$
B=\begin{pmatrix}
0&2&-2\\
-1&0&1\\
2&-2&0
\end{pmatrix}
\qquad\qquad
C=\begin{pmatrix}
0&0&1&-2&0\\
0&0&1&0&-2\\
-1&-1&0&1&1\\
2&0&-1&0&0\\
0&2&-1&0&0
\end{pmatrix}
$$
However, the matrix $C$ is not an unfolding for the matrix $B$. Indeed, after mutation $\mu_2$ of $B$ (resp, $\mu_3$ of $C$), the assertion (2) does not hold for block $E_1\times E_3$ of $\mu_3(C)$.
\end{example}

\smallskip

If $C$ is an unfolding of a skew-symmetrizable integer matrix $B$, it is natural to define an {\it unfolding of a diagram} of $B$ as a diagram of $C$. In general, we say that a diagram $\h S$ is an unfolding of a diagram $S$ if there exist matrices $B$ and $C$ with diagrams $S$ and $\h S$ respectively, and $C$ is an unfolding of $B$. This definition is equivalent to the following one.

\begin{definition}
Let $S$ be a diagram with vertices $x_1,\dots,x_n$, and let $d_1,\dots,d_n$ be positive integers. Let $\h S$ be a connected skew-symmetric diagram with vertices $x_{\hat\imath}$ indexed by sets $E_i$ of order $d_i$, such that for each $i,j\in\[1\dots n\]$ the following holds:

(A) there are no edges joining vertices inside $E_i$ and $E_j$;

(B) for all $\hat\imath\in E_i$ the sum of weights of all edges joining $x_{\hat\imath}$ with $E_j$ is the same, and all the arrows are oriented simultaneously either from $E_i$ to $E_j$ or from $E_j$ to $E_i$;

(C) the product of total weight of edges joining $x_{\hat\imath}$ with $E_j$ and total weight of edges joining $x_{\hat\jmath}$ with $E_i$ equals the weight of $x_ix_j$.

Define a {\it composite mutation} $\h\mu_i = \prod_{\hat\imath \in E_i} \mu_{\hat\imath}$ on $\h S$. As in the case of matrices, the mutation is well-defined. We say that $\h S$ is an {\it unfolding} of $S$ if for any sequence of iterated mutations $\mu_{i_1}\dots\mu_{i_k}$ a pair of diagrams $(\mu_{i_1}\dots\mu_{i_k}\! S,\,\h\mu_{i_1}\dots\h\mu_{i_k}\!\h S\,)$ satisfies the same conditions as the pair $(S,\h S)$ does, i.e. for each $i,j\le n$ the assumptions (A), (B) and (C) hold.

\end{definition}

The following example shows that an unfolding of a diagram may not be unique.

\begin{example}

Diagram \quad
\psfrag{2-}{\small $2$}
\psfrag{1}{\tiny $1$}
\psfrag{2}{\tiny $2$}
\psfrag{3}{\tiny $3$}
\raisebox{-0.37cm}[5mm][5mm]{\epsfig{file=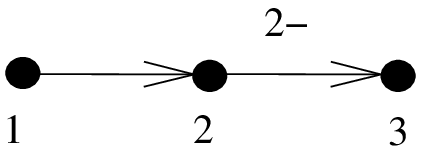,width=0.17\linewidth}}
\quad corresponds to two matrices
$$
\begin{pmatrix}
0&1&0\\
-1&0&1\\
0&-2&0
\end{pmatrix}
\qquad{\mr{and}}\qquad
\begin{pmatrix}
0&1&0\\
-1&0&2\\
0&-1&0
\end{pmatrix}
$$
with unfoldings, respectively,
$$
\begin{pmatrix}
0&1&0&0\\
-1&0&1&1\\
0&-1&0&0\\
0&-1&0&0
\end{pmatrix}
\qquad{{\mr{and}}}\qquad
\begin{pmatrix}
0&0&1&0&0\\
0&0&0&1&0\\
-1&0&0&0&1\\
0&-1&0&0&1\\
0&0&-1&-1&0
\end{pmatrix}
$$

It is easy to see that these two unfoldings correspond, respectively, to diagrams
$$\psfrag{2-}{\small $2$}
\psfrag{1}{\tiny $1$}
\psfrag{2}{\tiny $2$}
\psfrag{3}{\tiny $3$}
\psfrag{4}{\tiny $4$}
\psfrag{5}{\tiny $5$}
{\epsfig{file=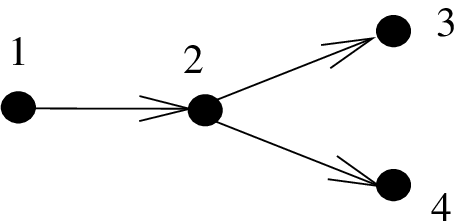,width=0.2\linewidth}}\qquad\quad\raisebox{0.52cm}{{\rm{and}}}\quad\qquad \raisebox{0.52cm}{\epsfig{file=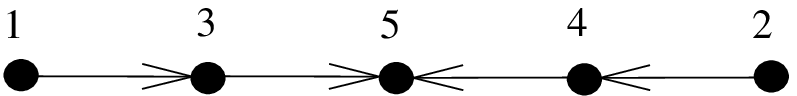,width=0.3\linewidth}}$$

\end{example}

\begin{lemma}
\label{blocks-un}
The diagrams in the second row of Table~\ref{newblocks} are unfoldings of the corresponding blocks shown in the first row of the table.

\end{lemma}

The proof consists of an elementary straightforward verification. We call the unfoldings of blocks shown in the second row of Table~\ref{newblocks} {\it local unfoldings}. They can be characterized as follows:
\begin{definition}
\label{def-simpl}
An unfolding is {\it local} if for any outlet $x_i$ of the initial skew-sym\-met\-rizable diagram, the corresponding integer $d_i$ is equal to one.
\end{definition}
This allows us to define for each s-decomposable diagram $S$ with fixed decomposition $S_{\mr{dec}}$ a skew-symmetric diagram (denote it by $\tau(S_{\mr{dec}})$) by gluing of unfoldings of corresponding blocks. Since all the local unfoldings of blocks are skew-symmetric blocks, $\tau(S_{\mr{dec}})$ is block-decomposable diagram. In other words, we may understand $\tau$ as a map from block decompositions of s-decomposable diagrams to block decompositions of block-decomposable ones. Our current goal is to prove Theorem~\ref{unfolding} which states that $\tau(S_{\mr{dec}})$ is an unfolding for $S$.

\begin{lemma}
\label{comm}
Let $S_{\mr{dec}}$ coincide with $N_x(S_{\mr{dec}})$ (i.e. $S_{\mr{dec}}$ is composed of blocks ${\mathsf{B}}_1$ and ${\mathsf{B}}_2$, ${\mathsf{B}}_2$ may be empty), $x\in S$, where $x\in {\mathsf{B}}_1\cap {\mathsf{B}}_2$ if ${\mathsf{B}}_2\ne\emptyset$.
Suppose also that $S_{\mr{dec}}$ is different from ones shown on Fig.~\ref{exclusion2}. Then there exists an $x$-good decomposition of $\mu_x(S)$, such that $$\tau({\mu_x(S)}_{\mr{dec}})=\h\mu_{x}(\tau(S_{\mr{dec}}))$$
\end{lemma}

\begin{figure}[!h]
\begin{center}
%\psfrag{2-}{\tiny $2$}
%\parbox[c]{1.8cm}{\epsfig{file=diagrams_pic/excl1.eps,width=1.0\linewidth}}
%\qquad
\psfrag{2-}{\tiny $2$}
\parbox[c]{1.8cm}{\epsfig{file=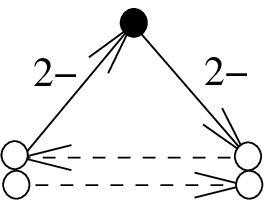,width=1.0\linewidth}}
\qquad\qquad
\psfrag{2-}{\tiny $2$}
\parbox[c]{1.8cm}{\epsfig{file=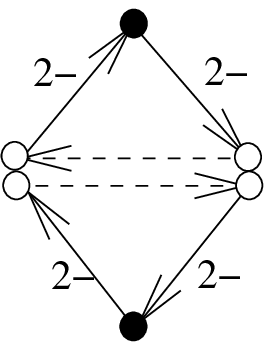,width=1.0\linewidth}}
\caption{Exceptional s-decomposable diagrams. Dotted edges are the edges of blocks disappearing in the diagram.}
\label{exclusion2}
\end{center}
\end{figure}

The proof considers the same cases as in the proof of Lemma~\ref{good} (in fact, this consideration includes proof of Lemma~\ref{good} as a partial case).

Combining Lemmas~\ref{glue_in} and~\ref{comm}, we get the following lemma.

\begin{lemma}
\label{unfolding-l}
Let $S$ be s-decomposable, and $x\in S_{\mr{dec}}$. If $S_{\mr{dec}}$ is different from ones shown on Fig.~\ref{exclusion2}, then there exists a decomposition of $\mu_x(S)$, such that $$\tau({\mu_x(S)}_{\mr{dec}})=\h\mu_{x}(\tau(S_{\mr{dec}}))$$

\end{lemma}

As a corollary, we obtain the unfolding theorem for diagrams.

\begin{theorem}
\label{unfolding}
Every s-decomposable diagram has a block-decomposable unfolding.

\end{theorem}

\begin{proof}
For diagrams that are not mutation-equivalent to ones shown on Fig.~\ref{exclusion2} the statement follows from Lemma~\ref{unfolding-l} (note that these two diagrams have no outlets, so they do not affect other mutation classes). Now consider the two mutation classes represented by the diagrams shown on Fig.~\ref{exclusion2}.

The left diagram has another block decomposition: it can be glued from two blocks of type $\mr{\t{III}}$. Starting from this decomposition, we get an unfolding according to Lemma~\ref{unfolding-l}.

Mutation class of the right diagram from Fig.~\ref{exclusion2} consists of three diagrams. Unfoldings are shown in Table~\ref{excl2}. All of them are block-decomposable: they can be glued either from two blocks of type $\mr{{IV}}$ (diagrams on the left and on the right), or from four blocks of type $\mr{{II}}$ (the one  in the middle).

\end{proof}

\begin{table}
\caption{Exceptional s-decomposable diagrams and their unfoldings}
\label{excl2}
\begin{tabular}{|l|c|c|c|}
\hline
&&&\\
\begin{tabular}{c}
{Diagrams}
\end{tabular}
&
%\begin{figure}
\psfrag{2-}{\tiny $2$}
\psfrag{4-}{\tiny $4$}
\parbox[c]{1.5cm}{\epsfig{file=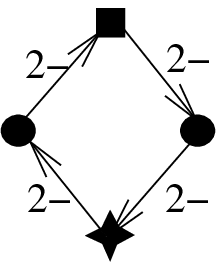,width=0.99\linewidth}}
%\end{figure}
&
%\begin{figure}
\psfrag{2-}{\tiny $2$}
\psfrag{4-}{\tiny $4$}
\parbox[c]{1.5cm}{\epsfig{file=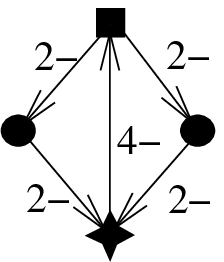,width=0.99\linewidth}}
%\end{figure}
&
%\begin{figure}[!h]
\psfrag{2-}{\tiny $2$}
\psfrag{4-}{\tiny $4$}
\parbox[c]{1.5cm}{\epsfig{file=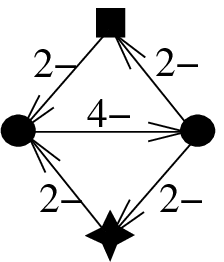,width=0.99\linewidth}}
%\end{figure}
\\
&&&
\\
\hline
&&&
\\
\begin{tabular}{c}
Unfoldings
\end{tabular}
&
%\begin{figure}[!h]
%\psfrag{3a}{\scriptsize $\widetilde{V}_1$}
\parbox[c]{1.7cm}{\epsfig{file=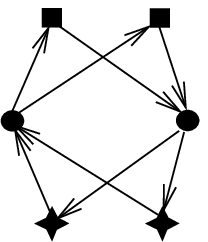,width=0.99\linewidth}}
%\end{figure}
&
%\begin{figure}[!h]
%\psfrag{3a}{\scriptsize $\widetilde{V}_1$}
\parbox[c]{2.6cm}{\epsfig{file=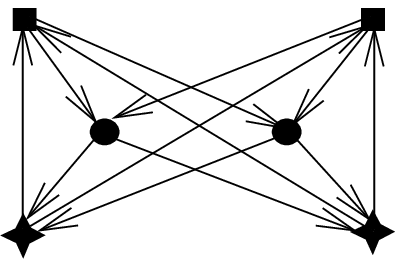,width=0.99\linewidth}}
%\end{figure}
&
%\begin{figure}[!h]
%\psfrag{3a}{\scriptsize $\widetilde{V}_1$}
\parbox[c]{1.7cm}{\epsfig{file=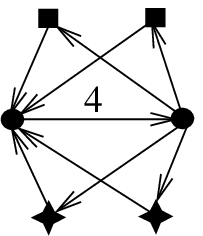,width=0.99\linewidth}}
%\end{figure}
\\
&&&
\\
\hline

\end{tabular}

\end{table}

%\begin{remark}
%Diagrams shown on Fig.~\ref{exclusion2} have no outlets, so they affect on a small number of diagrams only (more precisely, mutation classes of these diagrams contain $8$ diagrams in total). For entries of these three mutation classes unfoldings can be built explicitely; they are listed in Table~\ref{unf_excl2}.
%\end{remark}

\begin{lemma}
\label{sub}
Subdiagram of s-decomposable diagram is s-decomposable.
\end{lemma}

To prove the lemma, it is sufficient to show a way to substitute any block ${\mathsf{B}}$ with a vertex $x$ removed by some s-decomposable diagram such that all outlets remain outlets. The choice of substitutions is shown in Table~\ref{subst}.
%For seven diagrams mutation-equivalent to ones shown on Fig.~\ref{exclusion2} we can easily check the statement directly.

\begin{table}[!h]
\caption{Block decompositions of blocks with one vertex removed.}
\label{subst}
\begin{tabular}{|c|c|c|c|c|c|}
\hline
\vphantom{$\int\limits^-$}Block ${\mathsf{B}}$
&${\t{\mr{IV}}}$%{\epsfig{file=diagrams_pic/b4t1.eps,width=0.09\linewidth}}
&${\mr{\t{V}_1}}$%{\epsfig{file=diagrams_pic/b4t2.eps,width=0.09\linewidth}}
&${\mr{\t{V}_2}}$%{\epsfig{file=diagrams_pic/b5t1.eps,width=0.09\linewidth}}
&${\mr{\t{V}_{12}}}$%{\epsfig{file=diagrams_pic/b5t2.eps,width=0.09\linewidth}}
&${\mr{\t{VI}}}$%{\epsfig{file=diagrams_pic/b5t3.eps,width=0.09\linewidth}}
\\
\hline
\begin{tabular}{c}
Decomposition\\
of ${\mathsf{B}}\setminus x$
\end{tabular}
&${\mr{\t{III}}}$ or ${\mr{{I}}}$
&
\begin{tabular}{c}
${\mr{\t{IV}}}$, ${\mr{{III}}}$ or\\
\psfrag{2-}{\tiny $2$}\raisebox{-0.32cm}{\epsfig{file=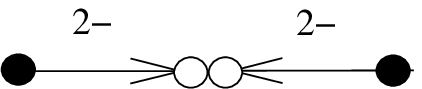,width=0.14\linewidth}}
\end{tabular}
&\begin{tabular}{c}
${\mr{\t{IV}}}$, ${\mr{{III}}}$ or\\
\psfrag{2-}{\tiny $2$}\raisebox{-0.32cm}{\epsfig{file=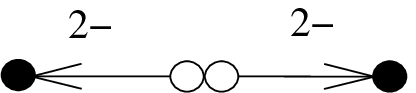,width=0.14\linewidth}}
\end{tabular}
&\begin{tabular}{c}
${\mr{{III}}}$ or\\
\psfrag{2-}{\tiny $2$}
\psfrag{1}{}
\psfrag{2}{}
\psfrag{3}{}
\raisebox{-0.42cm}{\epsfig{file=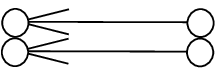,width=0.08\linewidth}}
\end{tabular}
&\begin{tabular}{c}
${\mr{\t{V}_1}}$, ${\mr{\t{V}_2}}$ or\vphantom{$\int\limits^-$}\\
{\epsfig{file=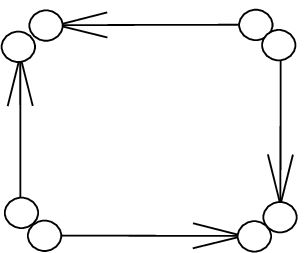,width=0.09\linewidth}}
\end{tabular}
\\
\hline
\end{tabular}
\end{table}

\begin{remark}
Lemma~\ref{sub} can be considered as a corollary of Theorem~\ref{unfolding}. More precisely, Theorem~\ref{unfolding} gives a geometric interpretation of Table~\ref{subst}. It is known that any subdiagram  $S\setminus x$ of block-decomposable diagram $S$ is block-decomposable: to obtain the corresponding triangulation of a bordered surface we need to cut the triangulation for $S$ along the edge corresponding to $x$. It is easy to check that if a block $\h {\mathsf{B}}$ is an unfolding of a block ${\mathsf{B}}$, and $x\in {\mathsf{B}}$, then removing all the vertices of type $\h x$ from $\h {\mathsf{B}}$ we are always left with a union of several blocks, such that initial symmetries of the block $\h {\mathsf{B}}$ are preserved. In other words, unfolding $\h {\mathsf{B}}$ of block ${\mathsf{B}}$ with $\h x$ removed can be ``folded back''.

\end{remark}

\section{Classification of mutation-finite diagrams and matrices}
\label{minimal}

Our proof of Theorem~\ref{all-s} follows the proof of Theorem~6.1 from~\cite{FST1}.

First, we define {\it minimal non-decomposable diagram} as a diagram which is not s-decomposable, but any its subdiagram is s-decomposable. According to Corollary~\ref{sub}, a non-decomposable diagram of order $n$ is minimal if and only if any its subdiagram of order $n-1$ is s-decomposable.

Then we prove the following generalization of~\cite[Theorem 5.2]{FST1}.

\begin{theorem}
\label{g8}
Any minimal non-decomposable diagram contains at most $7$ vertices.

\end{theorem}

The proof follows the proof of~\cite[Theorem 5.2]{FST1}. The only difference is now we need to consider more types of blocks. All essential tools remain the same. The complete list of refinements is contained in the Appendix~A.

The further program is the same as in skew-symmetric case (see~\cite{FST1}).
\begin{theorem}
\label{min}
The only minimal non-decomposable mutation-finite diagrams with at least three vertices are ones mutation-equivalent to one of the four diagrams $E_6$, $X_6$, $\t{G}_2$ and $F_4$ shown on Figure~\ref{minfig}.

\begin{figure}[!h]
\begin{center}
\psfrag{E}{${E_6}$}
\psfrag{X}{$X_6$}
\psfrag{G}{$\t{G}_2$}
\psfrag{F}{$F_4$}
\psfrag{4}{\tiny $4$}
\psfrag{3}{\tiny $3$}
\psfrag{2}{\tiny $2$}
\epsfig{file=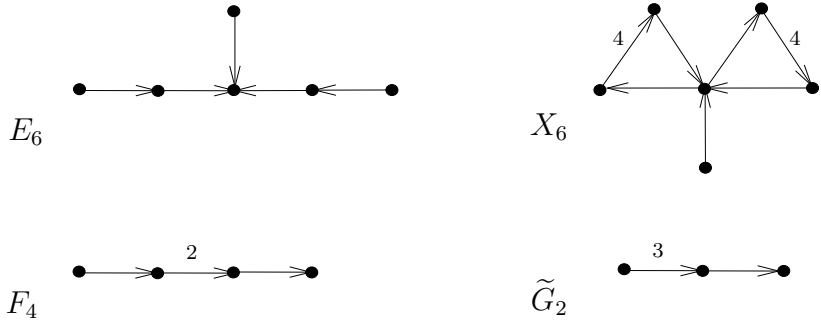,width=0.7\linewidth}\qquad\qquad
\caption{Minimal non-decomposable mutation-finite diagrams of order at least three}
\label{minfig}
\end{center}
\end{figure}

\end{theorem}

\begin{remark}
\label{g2}
Amongst diagrams of order two, there is exactly one non-decomposable diagram (called $G_2$) admitting an unfolding to a block-decomposable diagram (this diagram and corresponding unfolding $D_4$ are shown on Figure~\ref{g2u}). Moreover, $G_2$ is a unique non-decomposable diagram of order $2$ that can be a subdiagram of a mutation-finite diagram. Due to this fact, we may think $G_2$ to be minimal non-decomposable instead of $\t{G}_2$ (every mutation of which contains $G_2$).

\end{remark}

\begin{figure}[!ht]
\begin{center}
\psfrag{3}{\scriptsize $3$}
\psfrag{v2}{\tiny $v_2$}
\psfrag{w}{\tiny $w$}
\psfrag{w2}{\tiny $w_2$}
\psfrag{x}{\tiny $x$}
\psfrag{or}{\small or}
$G_2$\;\raisebox{0.45cm}{\epsfig{file=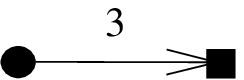,width=0.1\linewidth}}\qquad\qquad $D_4$\;\epsfig{file=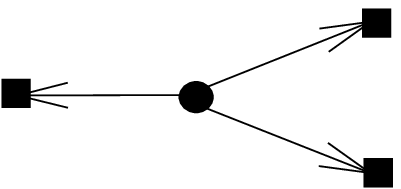,width=0.16\linewidth}
\caption{$G_2$ is a unique non-decomposable diagram of order two admitting an unfolding to a block-decomposable diagram (which is $D_4$).}
\label{g2u}
\end{center}
\end{figure}

\begin{proof}[Proof of Theorem~\ref{min}]
It is easy to see that the four diagrams shown on Figure~\ref{minfig} are mutation-finite and non-decomposable ($E_6$ and $X_6$ are discussed in~\cite{FST1}). To prove the theorem, it is sufficient to show that all other mutation-finite diagrams on at most $7$ vertices either are s-decomposable, or contain subdiagrams which are mutation-equivalent to one of $\t{G}_2$, $F_4$, $E_6$ or $X_6$. Due to Remark~\ref{g2}, instead of looking for subdiagrams mutation-equivalent to $\t{G}_2$ it is enough to find an edge of weight $3$.

Let $S$ be a minimal non-decomposable mutation-finite diagram. By Theorem~\ref{g8}, $|S|\le 7$. Since the mutation class of $S$ is finite, weights of edges of $S$ do not exceed $4$. The number of diagrams on at most $7$ vertices with bounded multiplicities of edges is finite. We use a computer~\cite{progr} to list all diagrams, choose mutation-finite ones, and check which of them are s-decomposable. The check is organized as in the proof of Theorem~5.11 from~\cite{FST1}.

As a result, besides skew-symmetric diagrams, we get $7$ mutation classes of non-decomposable mutation-finite diagrams of order at least two: $1$ of order three, $3$ of order four, $1$ of order five, and $2$ of order six. All these diagrams are shown on Figure~\ref{allfign}. Furthermore, a short straightforward check (using {\rm {\tt Java}} applet~\cite{K}) shows that any diagram which is mutation-equivalent to any of these $7$ ones contains either an edge of weight $3$ (and a subdiagram mutation-equivalent to $\t{G}_2$) or a subdiagram mutation-equivalent to $F_4$. The minimality is evident.

\end{proof}

\begin{cor}
\label{contmin}
Every non-decomposable mutation-finite diagram contains an edge of weight $3$ or subdiagram mutation-equivalent to one of $F_4$, $E_6$ and $X_6$.

\end{cor}

\begin{remark}
\label{no7}
As it follows from computations made in the proof of Theorem~\ref{g8}, any non-decomposable mutation-finite diagram of order $7$ is skew-symmetric. In other words, for any non-decomposable diagram $S$ of order $7$ containing an edge of weight $2$ or $3$, and any diagram $S'$ containing $S$ as a subdiagram, $S'$ is mutation-infinite. We will use this to show that there are no other non-decomposable diagrams except ones listed above.

The same computations show that any mutation-finite diagram containing an edge of weight $3$ is of order at most $4$. Clearly, all such diagrams are non-decomposable (since no block contains an edge of weight $3$).

\end{remark}

\begin{theorem}
\label{all}
A connected non-decomposable mutation-finite diagram of order greater than $2$ is mutation-equivalent
to one of the eleven diagrams $E_6$, $E_7$, $E_8$, $\widetilde E_6$, $\widetilde E_7$,
$\widetilde E_8$, $X_6$, $X_7$, $E_6^{(1,1)}$, $E_7^{(1,1)}$, $E_8^{(1,1)}$ shown on Figure~\ref{allfig}, or to one of the seven diagrams $\t G_2$, $F_4$, $\t F_4$, $G_2^{(*,+)}$, $G_2^{(*,*)}$, $F_4^{(*,+)}$, $F_4^{(*,*)}$ shown on Figure~\ref{allfign}.

\begin{figure}[!h]
\begin{center}
\psfrag{4}{\scriptsize $4$}
\psfrag{1}{$E_6$}
\psfrag{2}{$E_7$}
\psfrag{3}{$E_8$}
\psfrag{1_}{$\widetilde E_6$}
\psfrag{2_}{$\widetilde E_7$}
\psfrag{3_}{$\widetilde E_8$}
\psfrag{1__}{$E_6^{(1,1)}$}
\psfrag{2__}{$E_7^{(1,1)}$}
\psfrag{3__}{$E_8^{(1,1)}$}
\psfrag{4-}{$X_6$}
\psfrag{5-}{$X_7$}
\epsfig{file=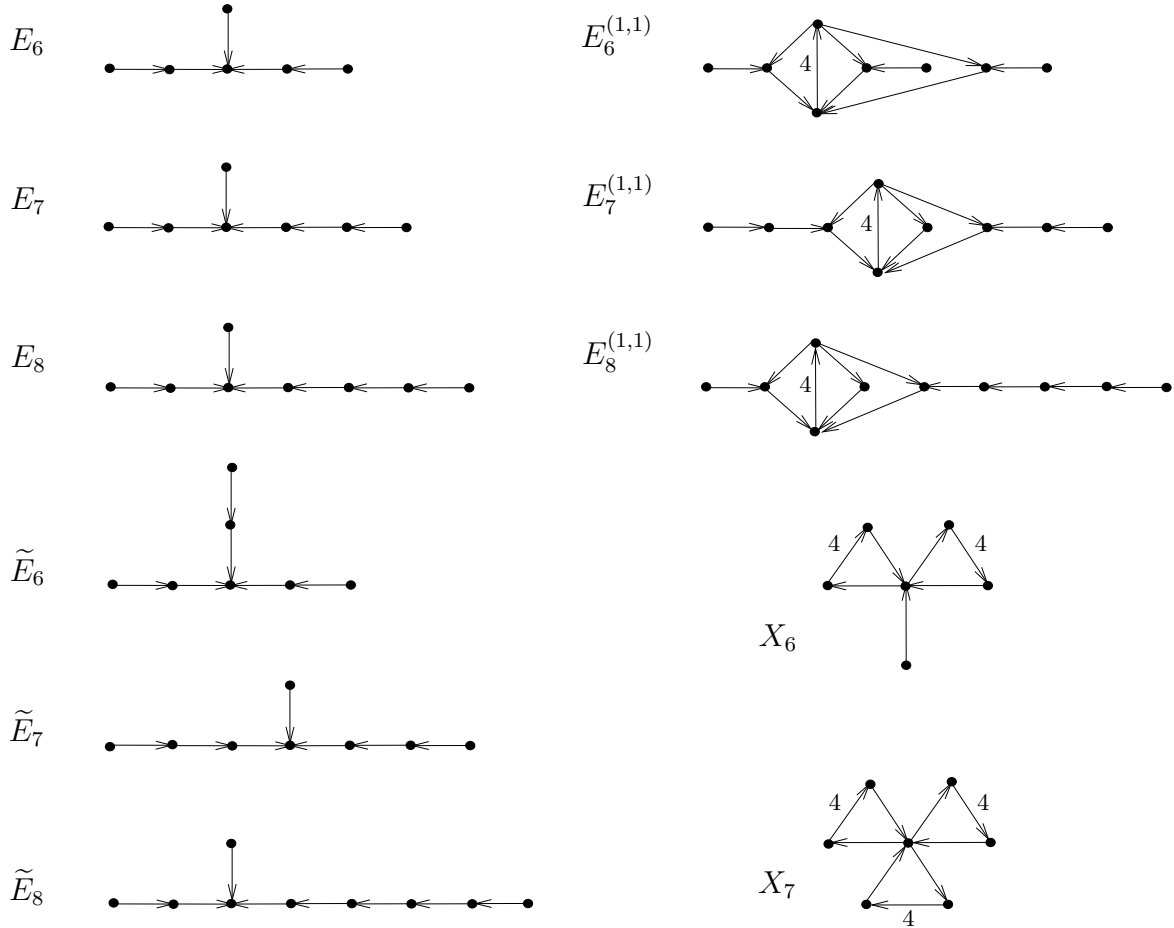,width=1.0\linewidth}
\caption{Non-decomposable mutation-finite skew-symmetric diagrams of order at least $3$}
\label{allfig}
\end{center}
\end{figure}

\end{theorem}

%The list of non-decomposable diagrams, as it appears in the Theorem,  has been conjectured to be complete in~\cite{DO}.

As we have already shown (see the proof of Theorem~\ref{min}), all these diagrams have finite mutation class and are non-decomposable (for skew-symmetric ones see~\cite{FST1}). We need to prove completeness of the list.

The following two lemmas are evident.

\begin{lemma}[\cite{FST1}, Lemma~6.4]
\label{submt}
Let  $S_1$ be a proper subdiagram of $S$, let $S_0$ be a diagram mutation-equivalent to $S_1$. Then there exists a diagram $S'$ which is mutation-equivalent to $S$ and contains $S_0$.

\end{lemma}

\begin{lemma}[\cite{FST1}, Lemma~6.2]
\label{m1}
Let $S$ be a non-decomposable diagram of order $d\ge 7$ with finite mutation class. Then $S$ contains a non-decomposable mutation-finite subdiagram $S_1$ of order $d-1$.

\end{lemma}

\begin{cor}
\label{p1}
Suppose that for some $d\ge 7$ there are no non-de\-com\-po\-sable  mutation-finite diagrams of order $d$. Then order of any non-decomposable  mutation-finite diagram does not exceed $d-1$.
\end{cor}

\begin{proof}[Proof of Theorem~\ref{all}]
In the proof of Theorem~\ref{g8} we listed all non-decomposable mutation-finite diagrams of order at most $7$. Now we want to show that all non-decomposable mutation-finite diagrams of order at least $8$ (in fact, at least $7$, see Remark~\ref{no7}) are skew-symmetric.

Suppose that $S$ is a non-decomposable mutation-finite diagram of order at least $8$, and $S$ is not skew-symmetric. Then $S$ contains a minimal non-decomposable mutation-finite subdiagram $S_1$ which is mutation-equivalent to a diagram of one of the four types shown on Fig.~\ref{minfig} (Theorem~\ref{g8}). If $S_1$ is mutation-equivalent to $\t{G}_2$ or $F_4$ then, taking any connected subdiagram $S'\subset S$ of order $7$ we see that $S'$ is mutation-infinite, which implies that $S$ is mutation-infinite, too. Therefore, $S_1$ is mutation-equivalent to $E_6$ or $X_6$.

Notice that any connected subdiagram $S'\subset S$ of order $7$ containing $S_1$ is skew-symmetric (otherwise $S'$ is mutation-infinite due to Remark~\ref{no7}), so it is mutation-equivalent to one of $E_7$, $X_7$, and $\t{E}_6$. According to Lemma~\ref{submt}, we may assume that $S'$ coincides with $E_7$, $X_7$, or $\t{E}_6$.

Suppose that $|S|=8$, and consider the unique vertex $x\in S\setminus S'$. If $x$ is joined with some vertex of $S_1$, then $S_2=\l S_1,x\r$ is of order $7$, so $S_2$ is skew-symmetric. This implies that the only edge which breaks skew-symmetry of $S$ is one joining $x$ with $S'\setminus S_1$. Therefore, this edge cannot be contained in any cycle: otherwise $S$ is not skew-symmetrizable. In particular, $x$ is not joined with any vertex of $S_1$.

In $X_7$ and $\t{E}_6$ every vertex is contained in some $X_6$ or $E_6$ respectively, so there is no way to add a vertex to $X_7$ or $\t{E}_6$ to get a mutation-finite diagram that is not skew-symmetric. In $E_7$ there is a unique vertex not contained in $E_6$. Attaching to that vertex an edge of weight $2$ or $4$ we get mutation-infinite diagrams~\cite{K} (weight $3$ is prohibited by Remark~\ref{no7}).
%According to Lemma~\ref{submt}, all the diagrams obtained by the same procedure from diagrams, mutation-equivalent to $E_7$, are also mutation-infinite.
Thus, all non-decomposable mutation-finite diagrams of order $8$ are skew-symmetric.

Now we proceed in the same way for diagrams of order $9$. Any such non-decomposable mutation-finite diagram is mutation-equivalent to one (denote it by $S$) containing $E_6$ or $X_6$. As it was proved, any connected subdiagram of $S$ of order $8$ containing $E_6$ or $X_6$ is skew-symmetric, so, performing some mutations, we can assume that $S$ contains $S'$ equal to one of $E_6^{(1,1)}$, $\t{E}_7$, $E_8$, and the remaining vertex of $S$ is not joined with any of $E_7$ and $\t{E}_6$ contained in $S'$. Again, any vertex of $E_6^{(1,1)}$ and $\t{E}_7$ belongs to some $\t{E}_6$ or $E_7$, and there is a unique vertex of $E_8$ not contained in $E_7$. Attaching to that vertex an edge of weight $2$ or $4$ we get mutation-infinite diagrams, so all non-decomposable diagrams of order $9$ are skew-symmetric.

We repeat the same procedure for diagrams of order $10$ without any new results (here we attach a node to $\t{E}_8$, while any vertex of $E_7^{(1,1)}$ belongs to some $\t{E}_7$), and then for diagrams of order $11$ (here any vertex of $E_8^{(1,1)}$ belongs to some $\t{E}_8$). Finally, we see that there are no non-decomposable diagrams of order $11$. In view of Corollary~\ref{p1}, this completes the proof.

\end{proof}

Now we will reformulate the result of this section in terms of matrices. We recall two evident statements about exchange matrices and their diagrams.

\begin{lemma}
\label{dmatr}
Diagram of mutation-finite matrix is mutation-finite.

\end{lemma}

\begin{lemma}
\label{find}
Any diagram is represented only by a finite number of skew-symmetriz\-able matrices.

\end{lemma}

Combining Lemmas~\ref{dmatr} and~\ref{find}, we get the following lemma.

\begin{lemma}
\label{iff}
A skew-symmetrizable matrix is mutation-finite if and only if its diagram is mutation-finite.

\end{lemma}

As an immediate corollary of Lemma~\ref{iff} and Theorem~\ref{all}, we obtain the following theorem.

\begin{theorem}
\label{all-s}
A skew-symmetrizable $n\times n$ matrix, $n\ge 3$, that is not skew-symmetric, has finite mutation class if and only if its diagram is either s-decomposable or mutation-equivalent to one of the seven types $\t G_2$, $F_4$, $\t F_4$, $G_2^{(*,+)}$, $G_2^{(*,*)}$, $F_4^{(*,+)}$, $F_4^{(*,*)}$ shown on Fig.~\ref{allfign}.
\end{theorem}

\section{Unfoldings of mutation-finite matrices and diagrams}
\label{unfoldings}

In this section we complete the construction of unfoldings for all mutation-finite diagrams, and specify the corresponding matrices. We also construct unfoldings for all mutation-finite matrices with non-decomposable diagrams.
 
First, we consider mutation-finite matrices admitting local unfoldings. As it is shown in Section~\ref{unfolding-s}, this leads to a block-decomposable unfolding for every s-decompos\-able diagram. All these unfoldings appear to be block-decomposable. Next, we show examples of non-local unfoldings for matrices with s-decompos\-able diagrams. Finally, we present unfoldings for all mutation-finite matrices with non-decomposable diagrams. These unfoldings are also mutation-finite but have (usually) non-decomposable diagrams. In particular, we obtain the following generalization of the results of Section~~\ref{unfolding-s}.

\begin{theorem}
\label{unf}
Any s-decomposable diagram admits an unfolding to a diagram arising from ideal tagged triangulation of a marked bordered surface.
Any mutation-finite matrix with non-decomposable diagram admits an unfolding to a mutation-finite skew-symmetric matrix. 
\end{theorem}

\subsection{Local unfoldings}
\label{unfolding-st}
In Section~\ref{unfolding-s} we constructed a local unfolding for every s-decomposable diagram. Let us describe the choice of matrices $B$ and $C$ corresponding to a diagram $S$ and its local unfolding $\h S$ respectively.

These matrices can be easily reconstructed by looking at the local unfoldings of blocks, see Table~\ref{reg}. To each edge of weight $4$ we assign a skew-symmetric submatrix. To each new block we assign a submatrix in such a (unique) way that for each outlet $x_i$ the number $d_i$ is a unit. In terms of matrix elements, this means that for any outlet $x_i$ and entry $b_{ij}\ne -b_{ji}$ the inequality $|b_{ij}|<|b_{ji}|$ holds if and only if $i<j$. The local unfoldings of blocks are diagrams of unfoldings of these matrices with coprime numbers $d_i$.

Now we take any block decomposition of a diagram $S$, assign to each block $S_j$ a matrix $B_j$ defined above (for skew-symmetric blocks the matrix is uniquely defined), and then glue all them in a natural way to obtain matrix $B$ with diagram $S$. In terms of matrices ``gluing'' is equivalent to summation of matrices, composed of $B_j$ at corresponding place and zeros outside. Since $d_i=1$ for any outlet $x_i$, after gluing we still have  $|b_{ij}|<|b_{ji}|$ if and only if $i<j$ and $b_{ij}\ne -b_{ji}$.

To obtain an unfolding $C$ of $B$ we take unfoldings $C_j$ of all matrices $B_j$ and glue them along outlets. Again, this procedure is well-defined since for every outlet $x_i$ the number $d_i$ is equal to one.

\begin{table}
\caption{Local unfoldings of blocks}
\label{reg}
\begin{tabular}{|c|c|c|c|c|}
\hline
Block
&Diagram&Matrix&Unfolding
&\begin{tabular}{c}
Diagram\\
unfolding
\end{tabular}\\
\hline
%&&&&\\
$\mr{\t{III}a}$&
\psfrag{2-}{\tiny $2$}\psfrag{3at}{}\raisebox{-5.5mm}{\epsfig{file=diagrams_pic/block3at.eps,width=0.09\linewidth}}&
{\small $\left(\begin{smallmatrix}
0&-1\\
2&0\\
\end{smallmatrix}\right)$}&
{\small $\left(\begin{smallmatrix}
0&-1&-1\\
1&0&0\\
1&0&0\\
\end{smallmatrix}\right)$}&\psfrag{3a}{}\raisebox{-3.5mm}[7mm][4mm]{\epsfig{file=diagrams_pic/block3a.eps,width=0.09\linewidth}}\\
%&&&&\\
\hline
%&&&&\\
$\mr{\t{III}b}$&
\psfrag{2-}{\tiny $2$}\psfrag{3bt}{}\raisebox{-5.5mm}{\epsfig{file=diagrams_pic/block3bt.eps,width=0.09\linewidth}}&
{\small $\left(\begin{smallmatrix}
0&-2\\
1&0\\
\end{smallmatrix}\right)$}&
{\small $\left(\begin{smallmatrix}
0&0&-1\\
0&0&-1\\
1&1&0\\
\end{smallmatrix}\right)$}&\psfrag{3b}{}\raisebox{-3.5mm}[7mm][4mm]{\epsfig{file=diagrams_pic/block3b.eps,width=0.09\linewidth}}\\
%&&&&\\
\hline
%&&&&\\
$\mr{\t{IV}}$&
\psfrag{2-}{\tiny $2$}\psfrag{4t}{}\raisebox{-7.5mm}{\epsfig{file=diagrams_pic/block4t.eps,width=0.09\linewidth}}&
{\small $\left(\begin{smallmatrix}
0&1&-1\\
-1&0&1\\
2&-2&0\\
\end{smallmatrix}\right)$}&
{\small $\left(\begin{smallmatrix}
0&1&-1&-1\\
-1&0&1&1\\
1&-1&0&0\\
1&-1&0&0\\
\end{smallmatrix}\right)$}&\psfrag{4}{}\raisebox{-8mm}[10mm][9.6mm]{\epsfig{file=diagrams_pic/block4.eps,width=0.09\linewidth}}\\
%&&&&\\
\hline
%&&&&\\
$\mr{\t{V}_1}$&
\psfrag{2-}{\tiny $2$}\psfrag{51t}{}\raisebox{-11.5mm}[10mm][9mm]{\epsfig{file=diagrams_pic/block51t.eps,width=0.09\linewidth}}&
{\small $\left(\begin{smallmatrix}
0&1&-1&1\\
-1&0&1&0\\
2&-2&0&-2\\
-1&0&1&0\\
\end{smallmatrix}\right)$}&
{\small $\left(\begin{smallmatrix}
0&1&-1&-1&1\\
-1&0&1&1&0\\
1&-1&0&0&-1\\
1&-1&0&0&-1\\
-1&0&1&1&0
\end{smallmatrix}\right)$}&\psfrag{5}{}\raisebox{-6.2mm}{\epsfig{file=diagrams_pic/block5.eps,width=0.10\linewidth}}\\
%&&&&\\
\hline
%&&&&\\
$\mr{\t{V}_2}$&
\psfrag{2-}{\tiny $2$}\psfrag{52t}{}\raisebox{-11.5mm}[10mm][9mm]{\epsfig{file=diagrams_pic/block52t.eps,width=0.09\linewidth}}&
{\small $\left(\begin{smallmatrix}
0&2&-2&2\\
-1&0&1&0\\
1&-1&0&-1\\
-1&0&1&0\\
\end{smallmatrix}\right)$}&
{\small $\left(\begin{smallmatrix}
0&0&1&-1&1\\
0&0&1&-1&1\\
1&1&0&1&0\\
1&1&-1&0&-1\\
1&1&0&1&0\\
\end{smallmatrix}\right)$}&\psfrag{5}{}\raisebox{-6.2mm}{\epsfig{file=diagrams_pic/block5.eps,width=0.10\linewidth}}\\
%&&&&\\
\hline
%&&&&\\
$\mr{\t{V}_{12}}$&
\psfrag{2-}{\tiny $2$}\psfrag{4}{\tiny $4$}\psfrag{512t}{}\raisebox{-7.5mm}{\epsfig{file=diagrams_pic/block512t.eps,width=0.09\linewidth}}&
{\small $\left(\begin{smallmatrix}
0&2&-2\\
-1&0&1\\
2&-2&0\\
\end{smallmatrix}\right)$}&
{\small $\left(\begin{smallmatrix}
0&0&1&-1&-1\\
0&0&1&-1&-1\\
-1&-1&0&1&1\\
1&1&-1&0&0\\
1&1&-1&0&0\\
\end{smallmatrix}\right)$}&\psfrag{5}{}\raisebox{-6.6mm}[9mm][7mm]{\epsfig{file=diagrams_pic/block5.eps,width=0.10\linewidth}}\\
%&&&&\\
\hline
%&&&&\\
$\mr{\t{VI}}$&
\psfrag{2-}{\tiny $2$}\psfrag{6t}{}\raisebox{-9.9mm}{\epsfig{file=diagrams_pic/block6t.eps,width=0.10\linewidth}}&
{\small $\left(\begin{smallmatrix}
0&1&0&1&-1\\
-1&0&-1&0&1\\
0&1&0&1&-1\\
-1&0&-1&0&1\\
2&-2&2&-2&0
\end{smallmatrix}\right)$}&
{\small $\left(\begin{smallmatrix}
0&1&0&1&-1&-1\\
-1&0&-1&0&1&1\\
0&1&0&1&-1&-1\\
-1&0&-1&0&1&1\\
1&-1&1&-1&0&0\\
1&-1&1&-1&0&0\\
\end{smallmatrix}\right)$}&
\psfrag{6t}{}\raisebox{-11.5mm}[13.5mm][12.5mm]{\epsfig{file=diagrams_pic/unf6_0.eps,width=0.10\linewidth}}
\\
%&&&&\\
\hline
\end{tabular}
\end{table}

\begin{example}\label{ex:non-simple}
Consider a diagram $S$ shown on Fig.~\ref{ex}, left. It has a block decomposition shown in the middle of the figure.

\begin{figure}[!hb]
\begin{tabular}{cp{1.5cm}cp{1.5cm}c}
\psfrag{2-}{\tiny $2$}
\raisebox{1mm}{\epsfig{file=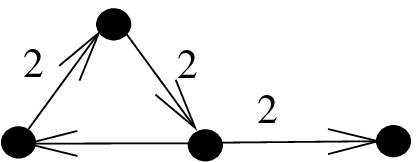,width=0.2\linewidth}}
&&
\psfrag{2-}{\tiny $2$}
\raisebox{1mm}{\epsfig{file=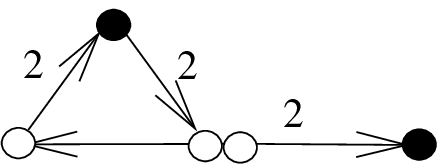,width=0.2\linewidth}}
&&
\psfrag{2-}{\tiny $2$}
\raisebox{-1mm}{\epsfig{file=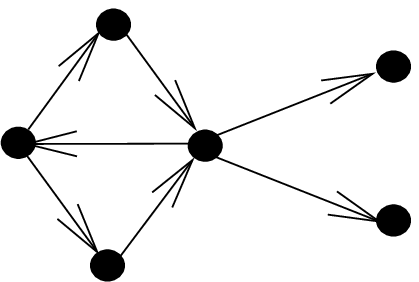,width=0.15\linewidth}}\\
$S$&&$S_{\mr{dec}}$&&$\h S$
\end{tabular}
\caption{Diagram $S$ with block decomposition $S_{\mr{dec}}$ and unfolding $\h S$}
\label{ex}
\end{figure}

Let $S_1$ and $S_2$ be blocks of type $\t{\mr{IV}}$ and ${\mr{\t{III}b}}$ respectively. Then the corresponding matrices are
$$
B_1=\begin{pmatrix}
0&1&-1\\
-2&0&2\\
1&-1&0
\end{pmatrix}
\qquad{{\mathrm{and}}}\qquad
B_2=\begin{pmatrix}
0&1\\
-2&0
\end{pmatrix},
$$
so we can write down the matrix
$$
B=\begin{pmatrix}
0&1&-1&0\\
-2&0&2&0\\
1&-1&0&1\\
0&0&-2&0
\end{pmatrix}
$$
corresponding to diagram $S$. Unfoldings of $B_1$ and $B_2$ are
$$
C_1=\begin{pmatrix}
0&1&1&-1\\
-1&0&0&1\\
-1&0&0&1\\
1&-1&-1&0
\end{pmatrix}
\qquad{{\mathrm{and}}}\qquad
C_2=\begin{pmatrix}
0&1&1\\
-1&0&0\\
-1&0&0
\end{pmatrix}
$$
Gluing them together, we obtain an unfolding $C$ of $B$,
$$
C=\begin{pmatrix}
0&1&1&-1&0&0\\
-1&0&0&1&0&0\\
-1&0&0&1&0&0\\
1&-1&-1&0&1&1\\
0&0&0&-1&0&0\\
0&0&0&-1&0&0
\end{pmatrix}
$$
The diagram $\h S$ of $C$ is shown on Fig.~\ref{ex} on the right.
\end{example}

\begin{remark}
\label{proof2}
By construction, diagrams of all the unfoldings described in the section are block-decomposable. This proves the first statement of Theorem~\ref{unf}.

\end{remark}

\subsection{Matrices with s-decomposable diagrams}
\label{sec-irr-unf}
Now consider arbitrary skew-sym\-metrizable matrix $B$ with s-decomposable diagram $S$. Let $x_1,\dots,x_n$ be vertices of $S$. We can assume numbers $d_1,\dots,d_n$ to be coprime (otherwise, divide all of them by the common divisor). Take any block decomposition of $S$.

\begin{lemma}
\label{equal}
For any two blocks $S_1$ and $S_2$ and any outlets $x_i\in S_1$ and $x_j\in S_2$ the numbers $d_i$ and $d_j$ are equal.

\end{lemma}

\begin{proof}
Looking at the list of blocks, it is easy to see that for any block $S'$ and any matrix $B'$ representing this block all outlets in $S'$ have the same numbers $d_i'$, where $d_i'$ are entries of diagonal matrix $D'$ skew-symmetrizing $B'$. Further, for any $x_i$ the number $d_i$ is a product of $d_i'$ and some number $d(S')$ which is the same for all vertices of $S'$. Thus, any two outlets in one block of $S_{\mr{dec}}$ have the same $d_i$. Now we are left to observe that for any outlets $x_i,x_j\in S_{\mr{dec}}$ there exists a sequence of outlets $x_{i_1}=x_i,x_{i_2},\dots,x_{i_k}=x_j$, such that any two consecutive entries belong to one block.

\end{proof}

Given $S$, $B$, and block decomposition of $S$, Lemma~ref{equal} allows us to define the {\it weight} of $S_{\mr{dec}}$ as the number $w=d_i$ for any outlet $x_i$ of any block. We call by a {\it regular part} of $S_{\mr{dec}}$ a union of blocks represented either by skew-symmetric matrices, or by matrices admitting a local unfolding. Regular part may not be connected, and every connected component of regular part always admits a local unfolding. Blocks of regular part are called {\it regular blocks}. The union of blocks admitting no local unfolding is called {\it irregular part} of $S_{\mr{dec}}$. Blocks of this part are {\it irregular blocks}.

\begin{lemma}
\label{weight}
Either $w=1$ and $B$ admits a local unfolding, or $w=2$.

\end{lemma}

\begin{proof}
If $w=1$ then we are in assumptions of previous section, so $B$ admits a local unfolding. Now suppose that $w>1$. Looking at the list of blocks (see Table~\ref{newblocks}), we see that $w$ is at most two times larger than the minimal value of $d_i$. Moreover, all $d_i$ are powers of two. In view of GCD equal to one, this implies that the minimal value is also one, so $w=2$.

\end{proof}

Now we construct unfoldings for all matrices representing irregular blocks. The proof of the following lemma is straightforward.

\begin{lemma}
\label{u-ir}
The third column of Table~\ref{irr} contains all possible matrices representing irregular blocks. Matrices in the fourth column are unfoldings of ones on the left.

\end{lemma}

\begin{table}
\caption{Unfoldings of irregular blocks}
\label{irr}
\begin{tabular}{|c|c|c|c|c|}
\hline
\begin{tabular}{c}
Block\\
number
\end{tabular}&Diagram&Matrix&Unfolding&\begin{tabular}{c}
Diagram\\
unfolding
\end{tabular}\\
\hline
%&&&&\\
$\mr{\t{III}a}$&
\psfrag{2-}{\tiny $2$}\psfrag{3at}{}\raisebox{-5mm}{\epsfig{file=diagrams_pic/block3at.eps,width=0.09\linewidth}}&
{\small $\left(\begin{smallmatrix}
0&-2\\
1&0\\
\end{smallmatrix}\right)$}&
{\small $\left(\begin{smallmatrix}
0&0&-1\\
0&0&-1\\
1&1&0\\
\end{smallmatrix}\right)$}&
\psfrag{2-}{\tiny $2$}\raisebox{-3mm}[7mm][4mm]{\epsfig{file=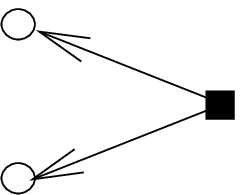,width=0.075\linewidth}}\\
%&&&&\\
\hline
%&&&&\\
$\mr{\t{III}b}$&
\psfrag{2-}{\tiny $2$}\psfrag{3bt}{}\raisebox{-5mm}{\epsfig{file=diagrams_pic/block3bt.eps,width=0.09\linewidth}}&
{\small $\left(\begin{smallmatrix}
0&-1\\
2&0\\
\end{smallmatrix}\right)$}&
{\small $\left(\begin{smallmatrix}
0&-1&-1\\
1&0&0\\
1&0&0\\
\end{smallmatrix}\right)$}&
\psfrag{2-}{\tiny $2$}\raisebox{-3mm}[7.5mm][4mm]{\epsfig{file=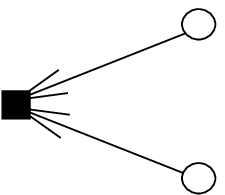,width=0.075\linewidth}}
\\
%&&&&\\
\hline
%&&&&\\
$\mr{\t{IV}}$&
\psfrag{2-}{\tiny $2$}\psfrag{4t}{}\raisebox{-7.5mm}{\epsfig{file=diagrams_pic/block4t.eps,width=0.09\linewidth}}&
{\small $\left(\begin{smallmatrix}
0&1&-2\\
-1&0&2\\
1&-1&0\\
\end{smallmatrix}\right)$}&
{\small $\left(\begin{smallmatrix}
0&0&1&0&-1\\
0&0&0&1&-1\\
-1&0&0&0&1\\
0&-1&0&0&1\\
1&1&-1&-1&0\\
\end{smallmatrix}\right)$}&
\psfrag{2-}{\tiny $2$}\raisebox{-7mm}[9.5mm][8.2mm]{\epsfig{file=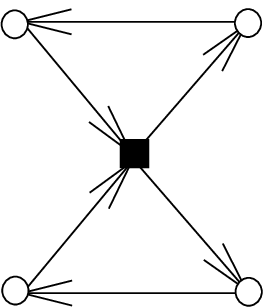,width=0.09\linewidth}}
\\
%&&&&\\
\hline
%&&&&\\
$\mr{\t{V}_1}$&
\psfrag{2-}{\tiny $2$}\psfrag{51t}{}\raisebox{-11mm}{\epsfig{file=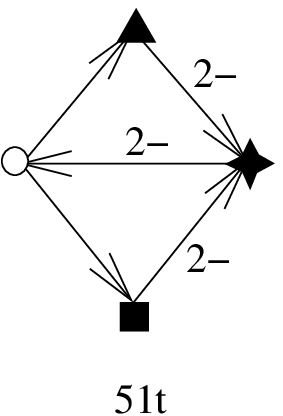,width=0.09\linewidth}}&
{\small $\left(\begin{smallmatrix}
0&1&-2&1\\
-1&0&2&0\\
1&-1&0&-1\\
-1&0&2&0\\
\end{smallmatrix}\right)$}&
{\small $\left(\begin{smallmatrix}
0&0&1&0&-1&1&0\\
0&0&0&1&-1&0&1\\
-1&0&0&0&1&0&0\\
0&-1&0&0&1&0&0\\
1&1&-1&-1&0&-1&-1\\
-1&0&0&0&1&0&0\\
0&-1&0&0&1&0&0
\end{smallmatrix}\right)$}&
\psfrag{2-}{\tiny $2$}\raisebox{-6.9mm}[11.2mm][7mm]{\epsfig{file=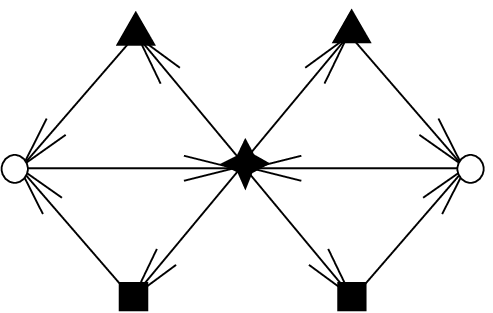,width=0.17\linewidth}}
\\
%&&&&\\
\hline
%&&&&\\
$\mr{\t{V}_2}$&
\psfrag{2-}{\tiny $2$}\psfrag{52t}{}\raisebox{-11mm}{\epsfig{file=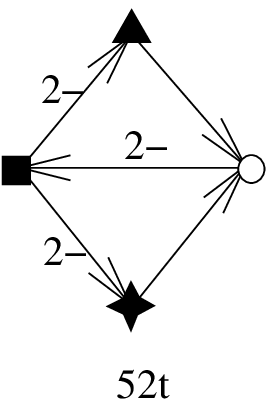,width=0.09\linewidth}}&
{\small $\left(\begin{smallmatrix}
0&1&-1&1\\
-2&0&1&0\\
2&-1&0&-1\\
-2&0&1&0\\
\end{smallmatrix}\right)$}&
{\small $\left(\begin{smallmatrix}
0&1&1&-1&-1&1&1\\
-1&0&0&1&0&0&0\\
-1&0&0&0&1&0&0\\
1&-1&0&0&0&-1&0\\
1&0&-1&0&0&0&-1\\
-1&0&0&1&0&0&0\\
-1&0&0&0&1&0&0\\
\end{smallmatrix}\right)$}&
\psfrag{2-}{\tiny $2$}\raisebox{-7.5mm}[11.2mm][7mm]{\epsfig{file=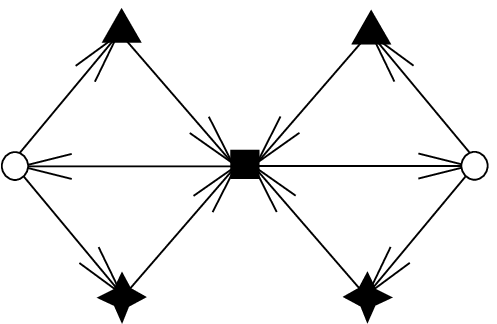,width=0.17\linewidth}}
\\
%&&&&\\
\hline
%&&&&\\
$\mr{\t{V}_{12}}$&
\psfrag{2-}{\tiny $2$}\psfrag{4}{\tiny $4$}\psfrag{512t}{}\raisebox{-8.5mm}{\epsfig{file=diagrams_pic/block512t.eps,width=0.09\linewidth}}&
{\small $\left(\begin{smallmatrix}
0&1&-1\\
-2&0&1\\
4&-2&0\\
\end{smallmatrix}\right)$}&
{\small $\left(\begin{smallmatrix}
0&1&1&-1&-1&-1&-1\\
-1&0&0&1&0&1&0\\
-1&0&0&0&1&0&1\\
1&-1&0&0&0&0&0\\
1&0&-1&0&0&0&0\\
1&-1&0&0&0&0&0\\
1&0&-1&0&0&0&0\\
\end{smallmatrix}\right)$}&
\psfrag{2-}{\tiny $2$}\raisebox{-7.2mm}[11.2mm][9.5mm]{\epsfig{file=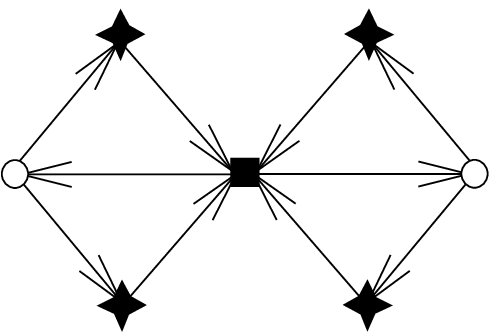,width=0.17\linewidth}}\\
%&&&&\\
\hline
%&&&&\\
$\mr{\t{V}_{12}}$&
\psfrag{2-}{\tiny $2$}\psfrag{4}{\tiny $4$}\psfrag{512t}{}\raisebox{-7.5mm}{\epsfig{file=diagrams_pic/block512t.eps,width=0.09\linewidth}}&
{\small $\left(\begin{smallmatrix}
0&2&-4\\
-1&0&2\\
1&-1&0\\
\end{smallmatrix}\right)$}&
{\small $\left(\begin{smallmatrix}
0&0&0&0&1&0&-1\\
0&0&0&0&0&1&-1\\
0&0&0&0&1&0&-1\\
0&0&0&0&0&1&-1\\
-1&0&-1&0&0&0&1\\
0&-1&0&-1&0&0&1\\
1&1&1&1&-1&-1&0\\
\end{smallmatrix}\right)$}&
\psfrag{2-}{\tiny $2$}\raisebox{-6.9mm}[11.2mm][9.5mm]{\epsfig{file=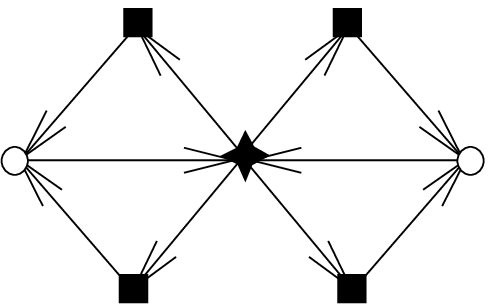,width=0.17\linewidth}}\\
%&&&&\\
\hline
%&&&&\\
$\mr{\t{V}_{12}}$&
\psfrag{2-}{\tiny $2$}\psfrag{4}{\tiny $4$}\psfrag{512t}{}\raisebox{-7.5mm}{\epsfig{file=diagrams_pic/block512t.eps,width=0.09\linewidth}}&
{\small $\left(\begin{smallmatrix}
0&1&-2\\
-2&0&2\\
2&-1&0\\
\end{smallmatrix}\right)$}&
{\small $\left(\begin{smallmatrix}
0&1&1&-2\\
-1&0&0&1\\
-1&0&0&1\\
2&-1&-1&0\\
\end{smallmatrix}\right)$}&
\psfrag{4}{\tiny $4$}\raisebox{-6.9mm}[11.2mm][9.5mm]{\epsfig{file=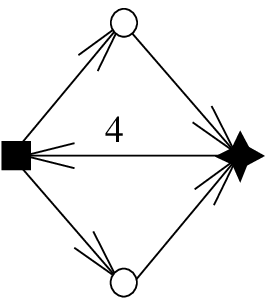,width=0.1\linewidth}}\\
%&&&&\\
\hline
\end{tabular}
\end{table}

\smallskip

From now on we can assume $w=2$. We will use matrices from Table~\ref{irr} together with local unfoldings (see Table~\ref{reg}) as a construction set for the following procedure. In the case the matrix has s-decomposable diagram containing only regular blocks and irregular blocks of types $\mr{\t{V}_{12}}$ (listed in the last row of Table~\ref{irr}) and $\mr{\t{III}}$, the procedure gives rise to an unfolding. We will generalize this construction and prove the existence of unfoldings in~\cite{orbifolds} using a geometric description in terms of triangulations of underlying orbifolds. 

We describe the procedure in terms of diagrams, then it can be easily translated to the language of matrices. 

First, for each connected component $S'$ of regular part we take its local unfolding $\h{S'}$. Then we take two copies of $\h{S'}$ and paint one of them in black, and the other in red. Now, looking at the list of unfoldings of irregular blocks (Table~\ref{irr}) one can note the following two properties: in all but one block there is exactly one vertex $x_i$ with $d_i=1$ (the exception is the last one, where unfolding contains two such vertices $x_i$ and $y_i$), and the unfolding consists of two similar blocks (of type ${\mr{I}}$, ${\mr{II}}$, or ${\mr{IV}}$) glued along $x_i$ (or $x_i$ and $y_i$). In other words, blocks contained in the unfolding of irregular part form pairs.

Therefore, we can do the following. For each irregular block $S''$ we take the corresponding unfolding $\h{S}''$ from Table~\ref{irr}, and paint one half of it (which is a skew-symmetric block) in black, and the other in red (we are interested in the color of outlets only, so the vertices $x_i$ and $y_i$ may remain uncolored). Now for every irregular block $S''$ and every outlet $x\in S''$, glue the unfolding $\h{S}''$ to red copy of the regular part of $S_{\mr{dec}}$ along red copy of $\h x$, and to black copy of the regular part of $S_{\mr{dec}}$ along black copy of $\h x$. In this way we get a diagram $\h S$. Performing the same operations with corresponding matrices, we obtain a matrix $C$.

\begin{example}
We show an example of a non-local unfolding provided by the construction above.
Consider a diagram $S$ shown on Fig.~\ref{ex2}, left, with  block decomposition shown at the center of the figure.

\begin{figure}[!hb]
\begin{tabular}{cp{0.8cm}cp{0.8cm}c}
\psfrag{2-}{\tiny $2$}
\raisebox{1.2mm}{\epsfig{file=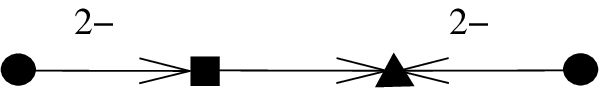,width=0.27\linewidth}}
&&
\psfrag{2-}{\tiny $2$}
\raisebox{1.2mm}{\epsfig{file=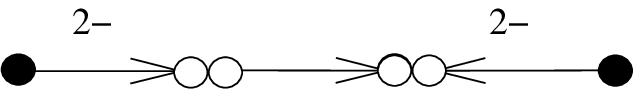,width=0.275\linewidth}}
&&
\psfrag{2-}{\tiny $2$}
\raisebox{-3.4mm}{\epsfig{file=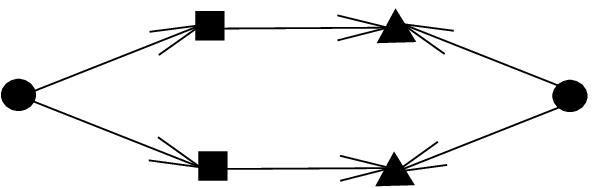,width=0.24\linewidth}}\\
$S$\phantom{SSSSSSSSS}&&$S_{\mr{dec}}$\phantom{SSSSSSSSS}&&$\h S$\phantom{SSSSSSSSSSSSSS}
\end{tabular}
\caption{Diagram $S$ with block decomposition $S_{\mr{dec}}$ and non-local unfolding $\h S$}
\label{ex2}
\end{figure}

Let both blocks $S_1$ and $S_3$ of type ${\rm{\t{III}}}$ be irregular. Then the corresponding matrices are
$$
B_1=\begin{pmatrix}
0&1\\
-2&0
\end{pmatrix}
\qquad{{\mathrm{and}}}\qquad
B_3=\begin{pmatrix}
0&-2\\
1&0
\end{pmatrix}
$$
The regular part $B_2$ with diagram $S_2$ consists of skew-symmetric matrix
$$
B_2=\begin{pmatrix}
0&1\\
-1&0
\end{pmatrix}
$$
The matrix $B$ representing $S$ will look like
$$
B=\begin{pmatrix}
0&1&0&0\\
-2&0&1&0\\
0&-1&0&-2\\
0&0&1&0
\end{pmatrix}
$$
Unfoldings of $B_1$ and $B_3$ are
$$
C_1=\begin{pmatrix}
0&1&1\\
-1&0&0\\
-1&0&0
\end{pmatrix}
\qquad{{\mathrm{and}}}\qquad
C_3=\begin{pmatrix}
0&0&-1\\
0&0&-1\\
1&1&0
\end{pmatrix}
$$
Gluing two copies of regular part with $C_1$ and $C_3$, we obtain the matrix
$$
C=\begin{pmatrix}
0&1&1&0&0&0\\
-1&0&0&1&0&0\\
-1&0&0&0&1&0\\
0&-1&0&0&0&-1\\
0&0&-1&0&0&-1\\
0&0&0&1&1&0\\
\end{pmatrix}
$$
The diagram $\h S$ of $C$ is shown on Fig.~\ref{ex2} on the right. A direct verification by checking all mutations in the complete mutation class shows that $C$ is an unfolding of $B$. 
\end{example}

\subsection{Matrices with non-decomposable diagrams}

According to Theorem~\ref{all-s}, the number of mutation-finite matrices with non-decomposable diagrams is finite, and the number of mutation classes is small. In Table~\ref{non} we present unfoldings for all matrices with non-decomposable mutation-finite diagrams. The straightforward proof makes use of Keller's {\rm {\tt Java}} applet~\cite{K} and elementary {\tt C++} code~\cite{progr}.

\begin{table}
\vbox to\textheight{\vss
\hbox to\textwidth{\hss
\begin{turn}{90}
\begin{minipage}{\textheight}\centering
\caption{Unfoldings of matrices with non-decomposable mutation-finite diagrams}
\label{non}
\begin{tabular}{|c|c|c|c|c|c|}
\hline
\multicolumn{2}{|c|}{\small Diagram}&{\small Matrix}&{\small Unfolding}&{\small Diagram unfolding}&
\begin{tabular}{c}
{\small Mutation class}\\ {\small of the unfolding}
\end{tabular}
\\
\hline
%&&&&&\\
$\t G_2$&\psfrag{3}{\tiny $3$}\raisebox{0.5mm}{\epsfig{file=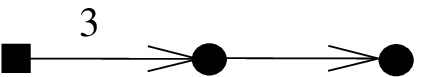,width=0.08\linewidth}}&
{\small $\left(\begin{smallmatrix}
0&3&0\\
-1&0&1\\
0&-1&0
\end{smallmatrix}\right)$}&
{\small $\left(\begin{smallmatrix}
\\
0&0&0&1&0\\
0&0&0&1&0\\
0&0&0&1&0\\
-1&-1&-1&0&1\\
0&0&0&-1&0\\
\end{smallmatrix}\right)$}&\raisebox{-3.5mm}[9mm][7mm]{\epsfig{file=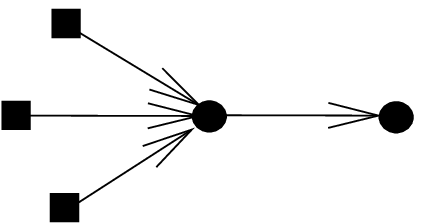,width=0.08\linewidth}}&
\begin{tabular}{c}{\small block-}\\ {\small decomposable}
\end{tabular}\\
%&&&&&\\
\hline
%&&&&&\\
$\t G_2$&\psfrag{3}{\tiny $3$}\raisebox{0.5mm}{\epsfig{file=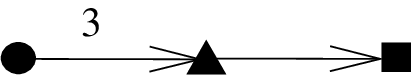,width=0.08\linewidth}}&
{\small $\left(\begin{smallmatrix}
\\
0&1&0\\
-3&0&1\\
0&-1&0
\end{smallmatrix}\right)$}&
{\small $\left(\begin{smallmatrix}
\\
0&1&1&1&0&0&0\\
-1&0&0&0&1&0&0\\
-1&0&0&0&0&1&0\\
-1&0&0&0&0&0&1\\
0&-1&0&0&0&0&0\\
0&0&-1&0&0&0&0\\
0&0&0&-1&0&0&0\\
\end{smallmatrix}\right)$}&\raisebox{-3.5mm}[0.056\linewidth][0.047\linewidth]{\epsfig{file=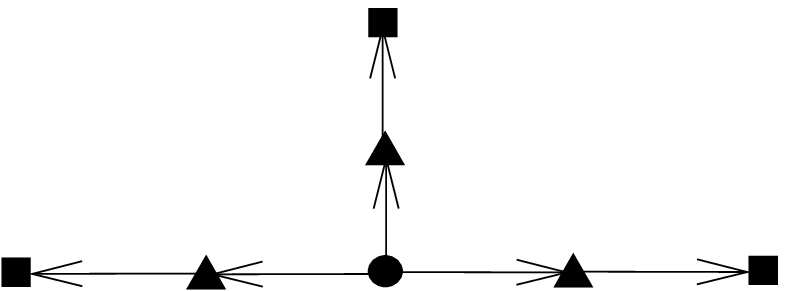,width=0.16\linewidth}}
&$\t E_6$\\
%&&&&&\\
\hline
%&&&&&\\
$F_4$&\psfrag{2}{\tiny $2$}\raisebox{0.5mm}{\epsfig{file=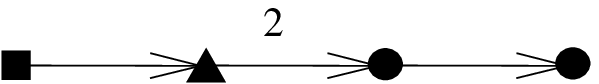,width=0.12\linewidth}}&
{\small $\left(\begin{smallmatrix}
\\
0&1&0&0\\
-1&0&2&0\\
0&-1&0&1\\
0&0&-1&0\\
\end{smallmatrix}\right)$}&
{\small $\left(\begin{smallmatrix}
\\
0&0&1&0&0&0\\
0&0&0&1&0&0\\
-1&0&0&0&1&0\\
0&-1&0&0&1&0\\
0&0&-1&-1&0&1\\
0&0&0&0&-1&0\\
\end{smallmatrix}\right)$}&\raisebox{-1.5mm}[0.051\linewidth][0.042\linewidth]{\epsfig{file=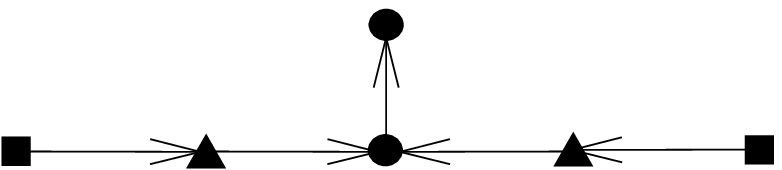,width=0.16\linewidth}}
&$E_6$\\
%&&&&&\\
\hline
%&&&&&\\
\begin{tabular}{c}$G_2^{(*,+)}$\\  \\ {\scriptsize \!\!($G_2^{(1,3)}$ or $G_2^{(3,1)}$)\!\!}    \end{tabular}
&\psfrag{3}{\tiny $3$}\psfrag{4}{\tiny $4$}\psfrag{2}{\tiny $2$}\raisebox{-4mm}{\epsfig{file=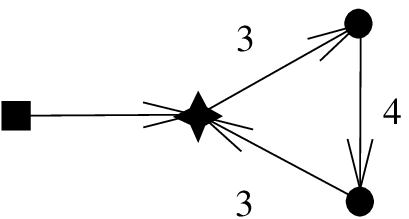,width=0.095\linewidth}}&
{\small $\left(\begin{smallmatrix}
\\
0&1&0&0\\
-1&0&3&-3\\
0&-1&0&2\\
0&1&-2&0\\
\end{smallmatrix}\right)$}&
{\small $\left(\begin{smallmatrix}
\\
0&0&0&1&0&0&0&0\\
0&0&0&0&1&0&0&0\\
0&0&0&0&0&1&0&0\\
-1&0&0&0&0&0&1&-1\\
0&-1&0&0&0&0&1&-1\\
0&0&-1&0&0&0&1&-1\\
0&0&0&-1&-1&-1&0&2\\
0&0&0&1&1&1&-2&0\\
\end{smallmatrix}\right)$}
&\psfrag{4}{\tiny $4$}\raisebox{-4.5mm}[0.063\linewidth][0.055\linewidth]{\epsfig{file=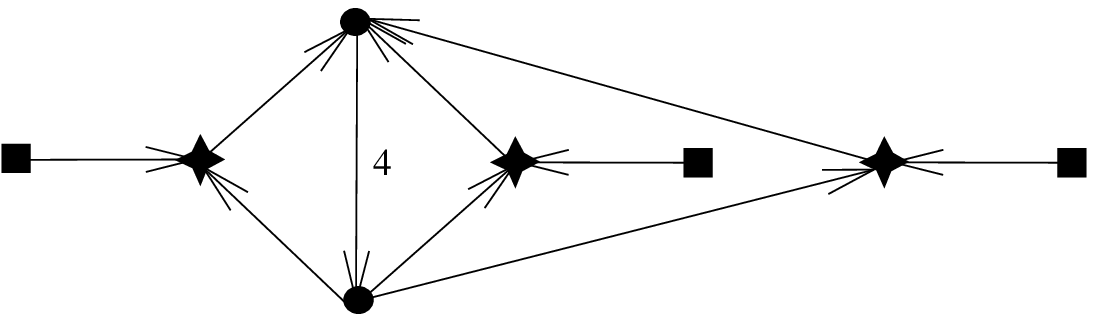,width=0.18\linewidth}}
&$E_6^{(1,1)}$\\
%&&&&&\\
\hline
%&&&&&\\
\begin{tabular}{c}$G_2^{(*,*)}$\\  \\ {\scriptsize \!\!($G_2^{(3,3)}$)\!\!}    \end{tabular}
&\psfrag{3}{\tiny $3$}\psfrag{4}{\tiny $4$}\raisebox{-5mm}{\epsfig{file=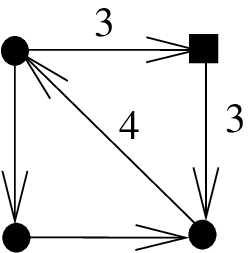,width=0.06\linewidth}}&
{\small $\left(\begin{smallmatrix}
\\
0&-1&2&-1\\
1&0&-1&0\\
-2&1&0&1\\
3&0&-3&0\\
\end{smallmatrix}\right)$}&
{\small $\left(\begin{smallmatrix}
\\
0&-1&2&-1&-1&-1\\
1&0&-1&0&0&0\\
-2&1&0&1&1&1\\
1&0&-1&0&0&0\\
1&0&-1&0&0&0\\
1&0&-1&0&0&0\\
\end{smallmatrix}\right)$}
&\psfrag{4}{\tiny $4$}\raisebox{-9mm}[0.063\linewidth][0.055\linewidth]{\epsfig{file=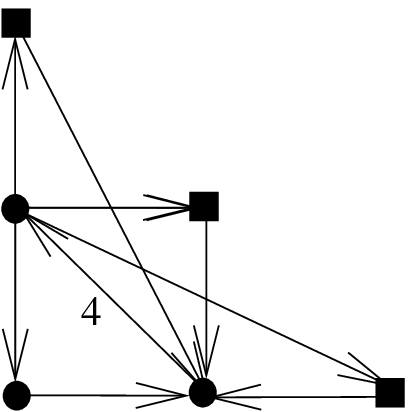,width=0.1\linewidth}}
&\begin{tabular}{c}
{\small block-}\\ {\small decomposable}
\end{tabular}\\
%&&&&&\\
\hline
\end{tabular}
\end{minipage}
\end{turn}
\hss}\vss}
\end{table}
\addtocounter{table}{-1}
\begin{table}
\vbox to\textheight{\vss
\hbox to\textwidth{\hss
\begin{turn}{90}
\begin{minipage}{\textheight}\centering
\caption{Cont.}
\begin{tabular}{|c|c|c|c|c|c|}
\hline
\multicolumn{2}{|c|}{\small Diagram}&{\small Matrix}&{\small Unfolding}&{\small Diagram unfolding}&
{\small \begin{tabular}{c}
Mutation \\ class of the \\ unfolding
\end{tabular}}
\\
\hline
%&&&&&\\
\begin{tabular}{c}$G_2^{(*,*)}$\\  \\ {\scriptsize \!\!($G_2^{(1,1)}$)\!\!}    \end{tabular}
&\psfrag{4}{\tiny $4$}\psfrag{3}{\tiny $3$}\raisebox{-5mm}{\epsfig{file=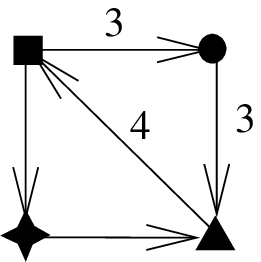,width=0.065\linewidth}}&
{\small $\left(\begin{smallmatrix}
\\
0&-1&2&-3\\
1&0&-1&0\\
-2&1&0&3\\
1&0&-1&0\\
\end{smallmatrix}\right)$}&
{\small $\left(\begin{smallmatrix}
\\
0&0&0&-1&0&0&1&0&1&-1\\
0&0&0&0&-1&0&1&1&0&-1\\
0&0&0&0&0&-1&0&1&1&-1\\
1&0&0&0&0&0&-1&0&0&0\\
0&1&0&0&0&0&0&-1&0&0\\
0&0&1&0&0&0&0&0&-1&0\\
-1&-1&0&1&0&0&0&0&0&1\\
0&-1&-1&0&1&0&0&0&0&1\\
-1&0&-1&0&0&1&0&0&0&1\\
1&1&1&0&0&0&-1&-1&-1&0\\
\end{smallmatrix}\right)$}&\raisebox{-11.5mm}[0.071\linewidth][0.066\linewidth]{\epsfig{file=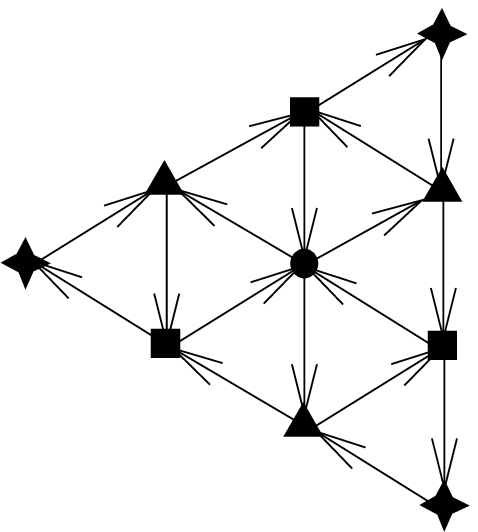,width=0.11\linewidth}}
&$E_8^{(1,1)}$\\
%&&&&&\\
\hline
%&&&&&\\
$\t F_4$&\psfrag{2-}{\tiny $2$}\raisebox{0mm}{\epsfig{file=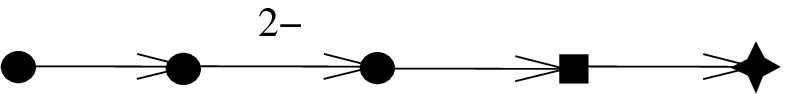,width=0.14\linewidth}}&
{\small $\left(\begin{smallmatrix}
\\
0&1&0&0&0\\
-1&0&2&0&0\\
0&-1&0&1&0\\
0&0&-1&0&1\\
0&0&0&-1&0\\
\end{smallmatrix}\right)$}&
{\small $\left(\begin{smallmatrix}
\\
0&0&1&0&0&0&0\\
0&0&0&1&0&0&0\\
-1&0&0&0&1&0&0\\
0&-1&0&0&1&0&0\\
0&0&-1&-1&0&1&0\\
0&0&0&0&-1&0&1\\
0&0&0&0&0&-1&0\\
\end{smallmatrix}\right)$}&\raisebox{-7mm}[0.055\linewidth][0.047\linewidth]{\epsfig{file=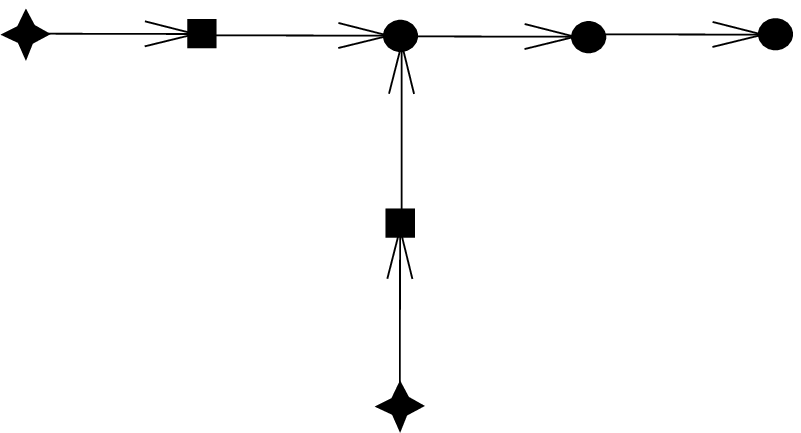,width=0.13\linewidth}}
&$\t E_6$\\
%&&&&&\\
\hline
%&&&&&\\
$\t F_4$&\psfrag{2-}{\tiny $2$}\raisebox{-0mm}{\epsfig{file=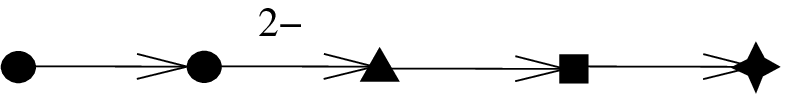,width=0.13\linewidth}}&
{\small $\left(\begin{smallmatrix}
\\
0&1&0&0&0\\
-1&0&1&0&0\\
0&-2&0&1&0\\
0&0&-1&0&1\\
0&0&0&-1&0\\
\end{smallmatrix}\right)$}&
{\small $\left(\begin{smallmatrix}
\\
0&1&0&0&0&0&0&0\\
-1&0&1&1&0&0&0&0\\
0&-1&0&0&1&0&0&0\\
0&-1&0&0&0&1&0&0\\
0&0&-1&0&0&0&1&0\\
0&0&0&-1&0&0&0&1\\
0&0&0&0&-1&0&0&0\\
0&0&0&0&0&-1&0&0\\
\end{smallmatrix}\right)$}&\raisebox{-3.5mm}[0.06\linewidth][0.052\linewidth]{\epsfig{file=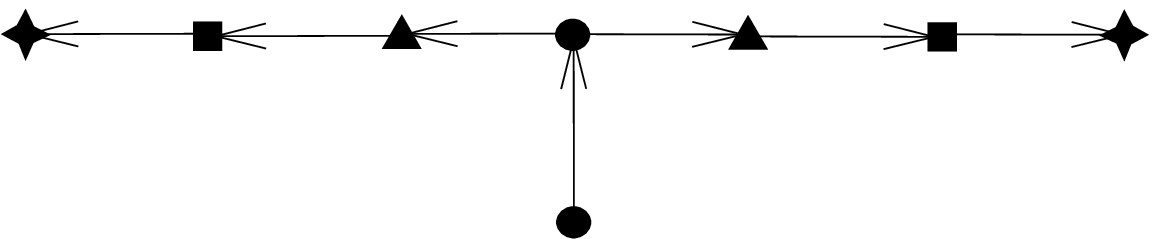,width=0.19\linewidth}}
&$\t E_7$\\
%&&&&&\\
\hline
%&&&&&\\
%$Y_5$&\psfrag{2-}{\tiny $2$}\raisebox{-5.5mm}{\epsfig{file=diagrams_pic/y5.eps,width=0.06\linewidth}}&
%{\small $\left(\begin{smallmatrix}
%\\
%0&1&0&0&-1\\
%-1&0&1&0&0\\
%0&-1&0&1&0\\
%0&0&-2&0&1\\
%2&0&0&-1&0\\
%\end{smallmatrix}\right)$}&
%{\small $\left(\begin{smallmatrix}
%\\
%0&1&0&0&0&-1&-1\\
%-1&0&1&0&0&0&0\\
%0&-1&0&1&1&0&0\\
%0&0&-1&0&0&1&0\\
%0&0&-1&0&0&0&1\\
%1&0&0&-1&0&0&0\\
%1&0&0&0&-1&0&0\\
%\end{smallmatrix}\right)$}&\raisebox{-6.5mm}[0.055\linewidth][0.047\linewidth]{\epsfig{file=diagrams_pic/e6t.eps,width=0.075\linewidth}}
%&$\t E_6$\\
%%&&&&&\\
%\hline
\end{tabular}
\end{minipage}
\end{turn}
\hss}\vss}
\end{table}
\addtocounter{table}{-1}
\begin{table}[p]
\vbox to\textheight{\vss
\hbox to\textwidth{\hss
\begin{turn}{90}
\begin{minipage}{\textheight}\centering
\caption{Cont.}
\begin{tabular}{|c|c|c|c|c|c|}
\hline
\multicolumn{2}{|c|}{\scriptsize Diagram}&{\scriptsize Matrix}&{\scriptsize Unfolding}&{\scriptsize Diagram unfolding}&
{\scriptsize \begin{tabular}{c}
Mutation \\ class of the\\ unfolding
\end{tabular}}\\
\hline
%&&&&&\\
%$Y_5$&\psfrag{2-}{\tiny $2$}\raisebox{-6.5mm}{\epsfig{file=diagrams_pic/y5_.eps,width=0.065\linewidth}}&
%{\small $\left(\begin{smallmatrix}
%\\
%0&1&0&0&-2\\
%-1&0&1&0&0\\
%0&-1&0&2&0\\
%0&0&-1&0&1\\
%1&0&0&-1&0\\
%\end{smallmatrix}\right)$}&
%{\small $\left(\begin{smallmatrix}
%\\
%0&0&1&0&0&0&0&-1\\
%0&0&0&1&0&0&0&-1\\
%-1&0&0&0&1&0&0&0\\
%0&-1&0&0&0&1&0&0\\
%0&0&-1&0&0&0&1&0\\
%0&0&0&-1&0&0&1&0\\
%0&0&0&0&-1&-1&0&1\\
%1&1&0&0&0&0&-1&0\\
%\end{smallmatrix}\right)$}&\raisebox{-6.5mm}[0.061\linewidth][0.053\linewidth]{\epsfig{file=diagrams_pic/e7t.eps,width=0.12\linewidth}}
%&$\t E_7$\\
%%&&&&&\\
%\hline
%&&&&&\\
{\scriptsize \begin{tabular}{c}{\small $F_4^{(*,+)}$}\\  \\ ($F_4^{(1,2)}$\\ or\\ $F_4^{(2,1)}$)    \end{tabular}}
&\psfrag{2}{\tiny $2$}\psfrag{4}{\tiny $4$}\raisebox{-3.8mm}{\epsfig{file=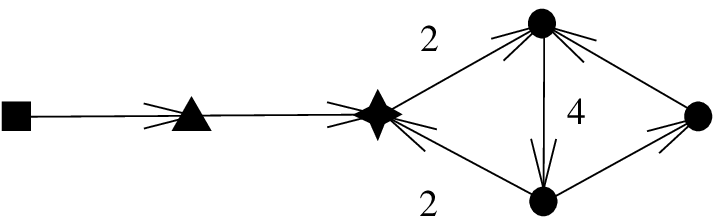,width=0.13\linewidth}}&
{\small $\left(\begin{smallmatrix}
\\
0&1&0&0&0&0\\
-1&0&1&0&0&0\\
0&-1&0&2&-2&0\\
0&0&-1&0&2&-1\\
0&0&1&-2&0&1\\
0&0&0&1&-1&0\\
\end{smallmatrix}\right)$}&
{\small $\left(\begin{smallmatrix}
\\
0&0&1&0&0&0&0&0&0\\
0&0&0&1&0&0&0&0&0\\
-1&0&0&0&1&0&0&0&0\\
0&-1&0&0&0&1&0&0&0\\
0&0&-1&0&0&0&1&-1&0\\
0&0&0&-1&0&0&1&-1&0\\
0&0&0&0&-1&-1&0&2&-1\\
0&0&0&0&1&1&-2&0&1\\
0&0&0&0&0&0&1&-1&0\\
\end{smallmatrix}\right)$}
&\psfrag{4}{\tiny $4$}\raisebox{-3.8mm}[0.067\linewidth][0.059\linewidth]{\epsfig{file=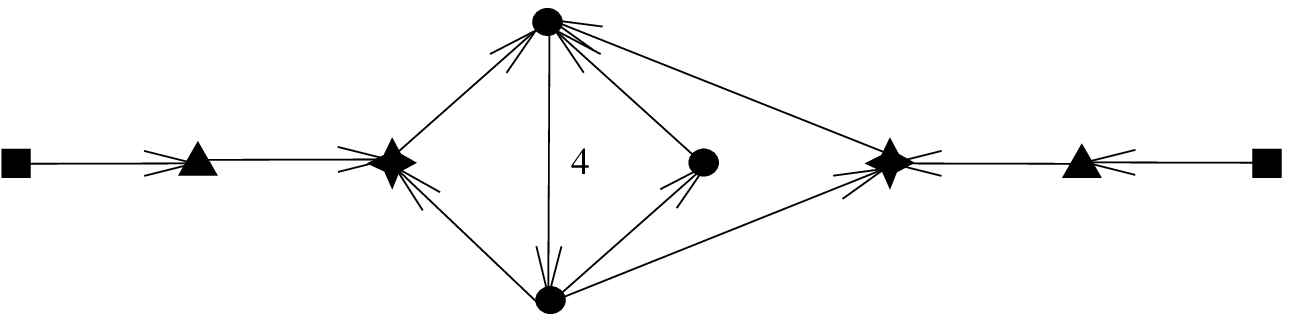,width=0.19\linewidth}}
&$E_7^{(1,1)}$\\
%&&&&&\\
\hline
%&&&&&\\
\begin{tabular}{c}$F_4^{(*,*)}$\\  \\ {\scriptsize \!\!($F_4^{(2,2)}$)\!\!}    \end{tabular}
&\psfrag{4}{\tiny $4$}\psfrag{2}{\tiny $2$}\raisebox{-3.1mm}{\epsfig{file=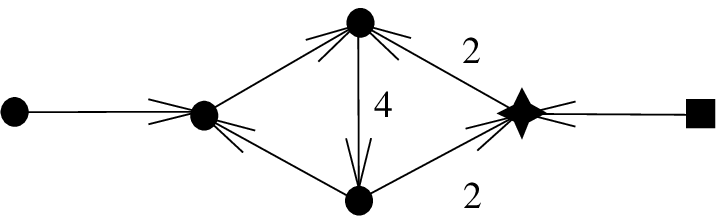,width=0.13\linewidth}}&
{\small $\left(\begin{smallmatrix}
\\
0&1&0&0&0&0\\
-1&0&1&-1&0&0\\
0&-1&0&2&-1&0\\
0&1&-2&0&1&0\\
0&0&2&-2&0&-1\\
0&0&0&0&1&0\\
\end{smallmatrix}\right)$}&
{\small $\left(\begin{smallmatrix}
\\
0&1&0&0&0&0&0&0\\
-1&0&1&-1&0&0&0&0\\
0&-1&0&2&-1&-1&0&0\\
0&1&-2&0&1&1&0&0\\
0&0&1&-1&0&0&-1&0\\
0&0&1&-1&0&0&0&-1\\
0&0&0&0&1&0&0&0\\
0&0&0&0&0&1&0&0\\
\end{smallmatrix}\right)$}
&\psfrag{4}{\tiny $4$}\raisebox{-4.3mm}[0.061\linewidth][0.053\linewidth]{\epsfig{file=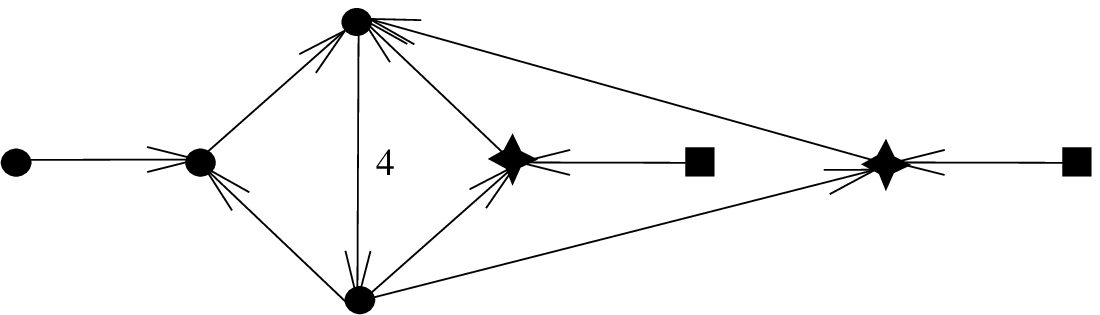,width=0.19\linewidth}}
&$E_6^{(1,1)}$\\
%&&&&&\\
\hline
%&&&&&\\
\begin{tabular}{c}$F_4^{(*,*)}$\\  \\ {\scriptsize \!\!($F_4^{(1,1)}$)\!\!}    \end{tabular}
&\psfrag{2}{\tiny $2$}\psfrag{4}{\tiny $4$}\raisebox{-3.5mm}{\epsfig{file=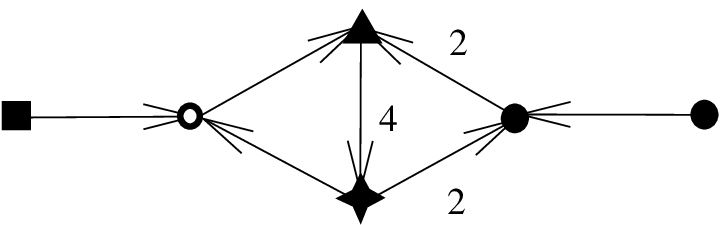,width=0.13\linewidth}}&
{\small $\left(\begin{smallmatrix}
\\
0&1&0&0&0&0\\
-1&0&1&-1&0&0\\
0&-1&0&2&-2&0\\
0&1&-2&0&2&0\\
0&0&1&-1&0&-1\\
0&0&0&0&1&0\\
\end{smallmatrix}\right)$}&
{\small $\left(\begin{smallmatrix}
\\
0&0&1&0&0&0&0&0&0&0\\
0&0&0&1&0&0&0&0&0&0\\
-1&0&0&0&1&0&-1&0&0&0\\
0&-1&0&0&0&1&0&-1&0&0\\
0&0&-1&0&0&0&1&1&-1&0\\
0&0&0&-1&0&0&1&1&-1&0\\
0&0&1&0&-1&-1&0&0&1&0\\
0&0&0&1&-1&-1&0&0&1&0\\
0&0&0&0&1&1&-1&-1&0&-1\\
0&0&0&0&0&0&0&0&1&0\\
\end{smallmatrix}\right)$}
&\psfrag{4}{\tiny $4$}\raisebox{-6.5mm}[0.071\linewidth][0.063\linewidth]{\epsfig{file=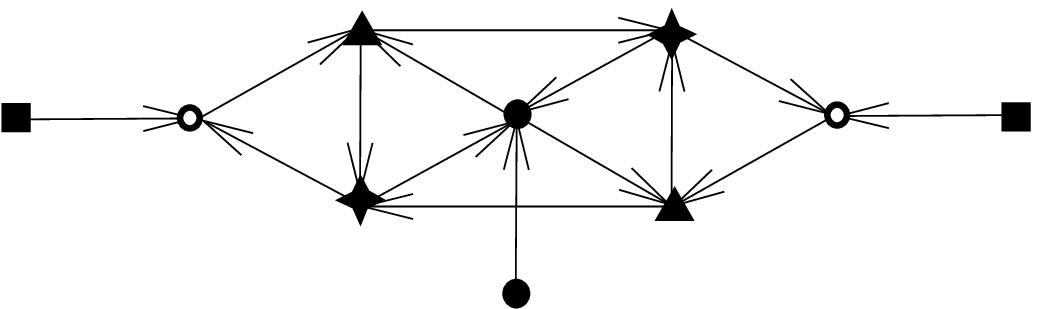,width=0.19\linewidth}}
&$E_8^{(1,1)}$\\
%&&&&&\\
\hline
\end{tabular}
\end{minipage}
\end{turn}
\hss}\vss}
\end{table}

\begin{remark}
\label{proof1}
As we can see from Table~\ref{non}, all the unfoldings constructed are mutation-finite. Together with Remark~\ref{proof2}, this completes the proof of Theorem~\ref{unf}.

\end{remark}

%%%%%%%%%%%%%%%%%%%%%%%%%%%%%%%%%%%%%%%%%%%

\section{Triangulations of bordered surfaces and s-decomposable diagrams}
\label{triangle}

In this section we discuss relations between s-decomposable diagrams and triangulations of bordered surfaces.
In Section~\ref{unfoldings}, we have shown that for any s-decomposable diagram $S$ there is a matrix admitting an unfolding with a block-decomposable diagram $\h S$. Abusing notation, we will call the original matrix $B$ (resp., diagram $S$) \emph{folding} of $C$ (resp, $\h S$). Every time we use notion of folding we keep in mind a fixed unfolding. Further, if a vertex $x$ of s-decomposable diagram $S$ corresponds to vertices $x_1,\dots, x_k$ of its unfolding $\h S$ we say that $x$ is a {\it folding} of $x_1,\dots,x_k$, and mutation of $S$ in the vertex $x$ is called the {\it folding} of the composite mutation $\h\mu_x$, which is a $k$-tuple of corresponding mutations of $\h S$ in vertices $x_1,\dots,x_k$. 

As we mentioned above block-decomposable diagrams are in one-to-one correspondence with adjacency matrices of arcs of ideal tagged triangulations of bordered two-dimensional surfaces with marked points. Below we identify diagram of unfolding (with fixed block decomposition) and the corresponding triangulation. We refer to~\cite{FST} for background on tagged triangulations.

New blocks of types $\t{\mr{III}}-\t{\mr{VI}}$ admit local unfoldings into block-decomposable diagrams shown in Table~\ref{reg}. 
These unfoldings are in one-to-one correspondence with the triangulations shown on Figure~\ref{fig:triangulations}. The last one is a tagged triangulation of a sphere (the exterior is also a triangle). The others are tagged triangulations of a disk. 

\begin{table}[!h]
\begin{center}
\caption{Triangulations of blocks corresponding to local unfoldings}
\label{fig:triangulations}
\begin{tabular}{cp{0.7cm}cp{0.7cm}c}
Diagram&&Unfolding&&Triangulation\\
\hline
&&&&\\
\psfrag{u}{\tiny $u$}
\psfrag{v}{\tiny $v$}
\psfrag{2-}{\tiny $2$}
\raisebox{8mm}{\epsfig{file=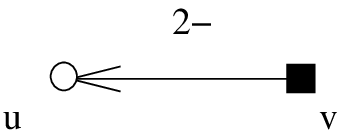,width=0.105\linewidth}}&&
\psfrag{u}{\tiny $u$}
\psfrag{v1}{\tiny $v_1$}
\psfrag{v2}{\tiny $v_2$}
\raisebox{4mm}{\epsfig{file=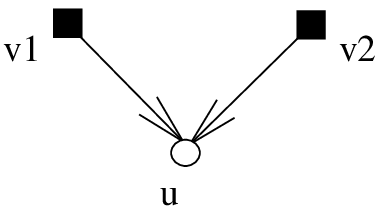,width=0.12\linewidth}}&&
\psfrag{u}{\tiny $u$}
\psfrag{v1}{\tiny $v_1$}
\psfrag{v2}{\tiny $v_2$}
\raisebox{0mm}{\epsfig{file=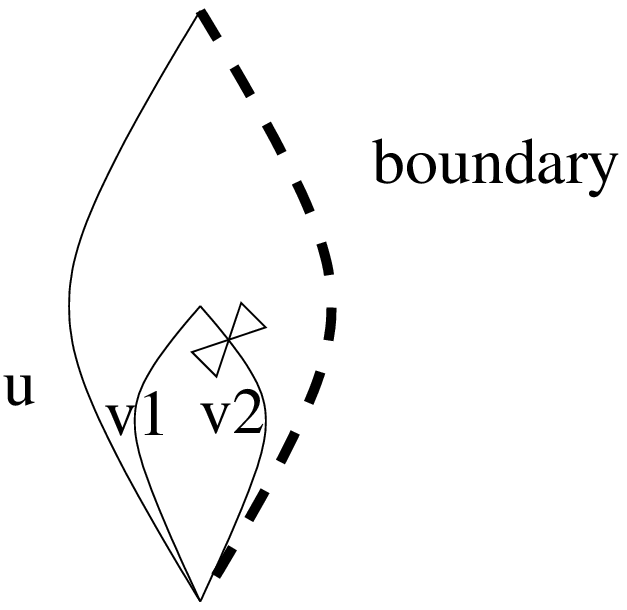,width=0.12\linewidth}}\\
%&&&&\\
%&&&&\\
\psfrag{u}{\tiny $u$}
\psfrag{v}{\tiny $v$}
\psfrag{2-}{\tiny $2$}
\raisebox{8mm}{\epsfig{file=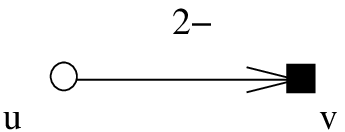,width=0.105\linewidth}}&&
\psfrag{u}{\tiny $u$}
\psfrag{v1}{\tiny $v_1$}
\psfrag{v2}{\tiny $v_2$}
\raisebox{4mm}{\epsfig{file=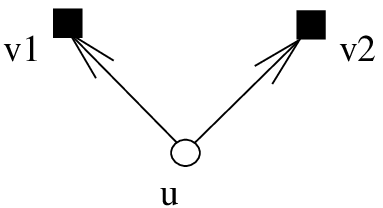,width=0.12\linewidth}}&&
\psfrag{u}{\tiny $u$}
\psfrag{v1}{\tiny $v_1$}
\psfrag{v2}{\tiny $v_2$}
\raisebox{0mm}{\epsfig{file=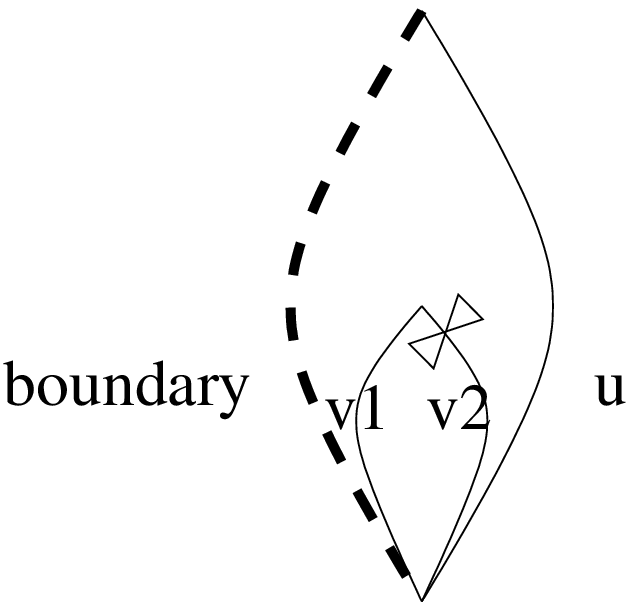,width=0.12\linewidth}}\\
%&&&&\\
%&&&&\\
\psfrag{u}{\tiny $u$}
\psfrag{v}{\tiny $v$}
\psfrag{w}{\tiny $w$}
\psfrag{2-}{\tiny $2$}
\raisebox{6mm}{\epsfig{file=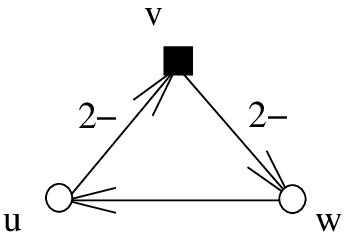,width=0.105\linewidth}}&&
\psfrag{u}{\tiny $u$}
\psfrag{w}{\tiny $w$}
\psfrag{v1}{\tiny $v_1$}
\psfrag{v2}{\tiny $v_2$}
\raisebox{0mm}{\epsfig{file=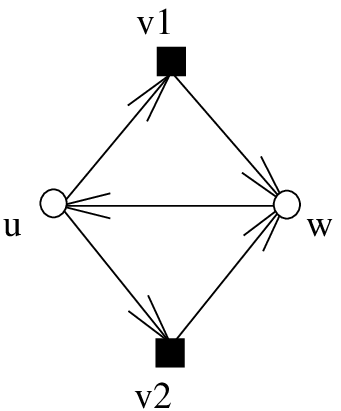,width=0.12\linewidth}}&&
\psfrag{u}{\tiny $u$}
\psfrag{w}{\tiny $w$}
\psfrag{v1}{\tiny $v_1$}
\psfrag{v2}{\tiny $v_2$}
\raisebox{0mm}{\epsfig{file=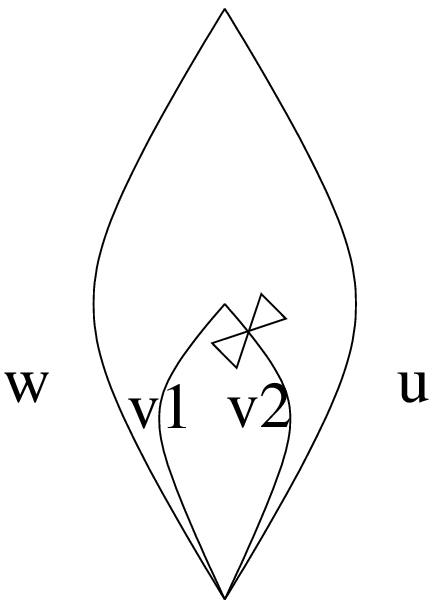,width=0.09\linewidth}}\\
%&&&&\\
&&&&\\
\psfrag{u}{\tiny $u$}
\psfrag{w}{\tiny $w$}
\psfrag{p}{\tiny $p$}
\psfrag{q}{\tiny $q$}
\psfrag{2-}{\tiny $2$}
\raisebox{0mm}{\epsfig{file=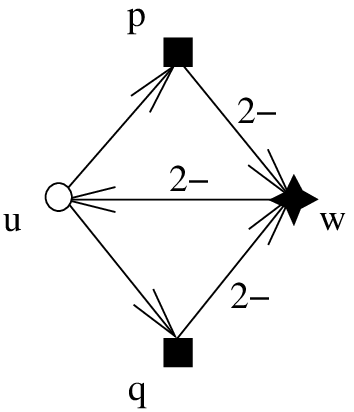,width=0.12\linewidth}}&&
\psfrag{u}{\tiny $u$}
\psfrag{r}{\tiny $q$}
\psfrag{w}{\tiny $w_2$}
\psfrag{p}{\tiny $p$}
\psfrag{q}{\tiny $w_1$}
\raisebox{4mm}{\epsfig{file=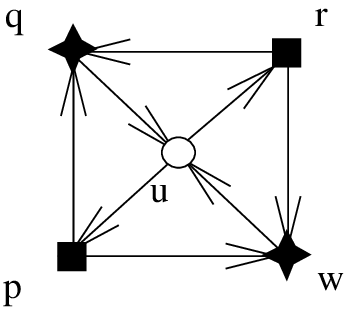,width=0.12\linewidth}}&&
\psfrag{u}{\tiny $u$}
\psfrag{w1}{\tiny $p$}
\psfrag{w2}{\tiny $q$}
\psfrag{p1}{\tiny $w_1$}
\psfrag{p2}{\tiny $w_2$}
\raisebox{0mm}{\epsfig{file=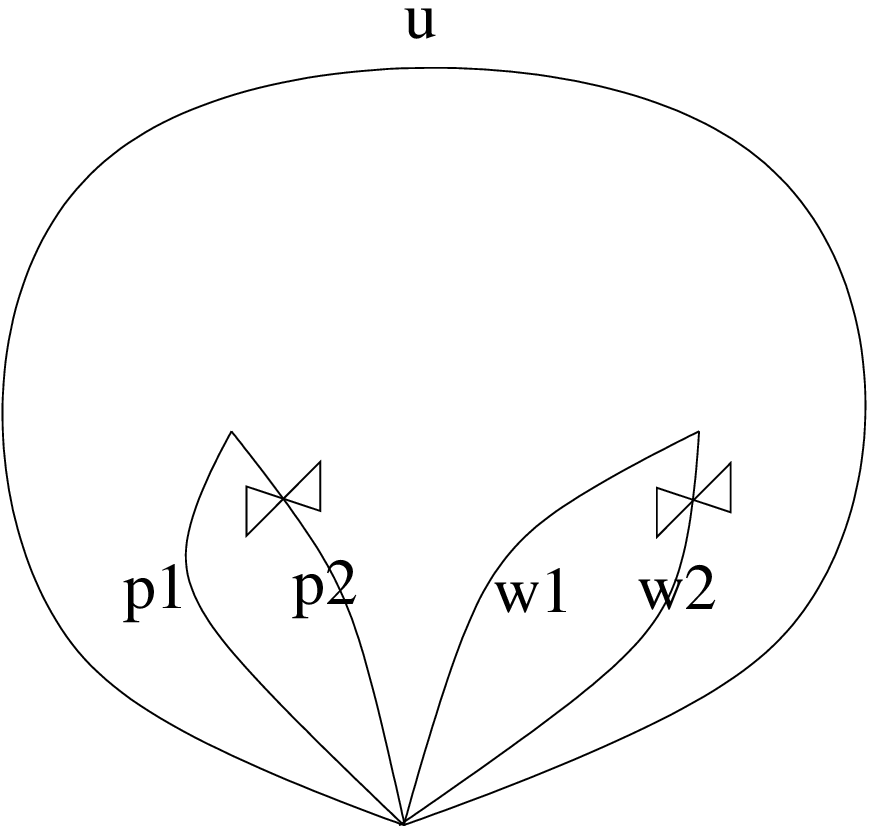,width=0.15\linewidth}}\\
%&&&&\\
&&&&\\
\psfrag{u}{\tiny $u$}
\psfrag{w}{\tiny $w$}
\psfrag{p}{\tiny $p$}
\psfrag{q}{\tiny $q$}
\psfrag{2-}{\tiny $2$}
\raisebox{0mm}{\epsfig{file=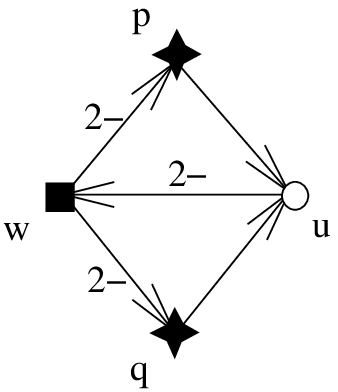,width=0.12\linewidth}}&&
\psfrag{u}{\tiny $u$}
\psfrag{r}{\tiny $w_2$}
\psfrag{w}{\tiny $q$}
\psfrag{p}{\tiny $w_1$}
\psfrag{q}{\tiny $p$}
\raisebox{4mm}{\epsfig{file=diagrams_pic/block5l.eps,width=0.12\linewidth}}&&
\psfrag{u}{\tiny $u$}
\psfrag{w1}{\tiny $w_1$}
\psfrag{w2}{\tiny $w_2$}
\psfrag{p1}{\tiny $p$}
\psfrag{p2}{\tiny $q$}
\raisebox{0mm}{\epsfig{file=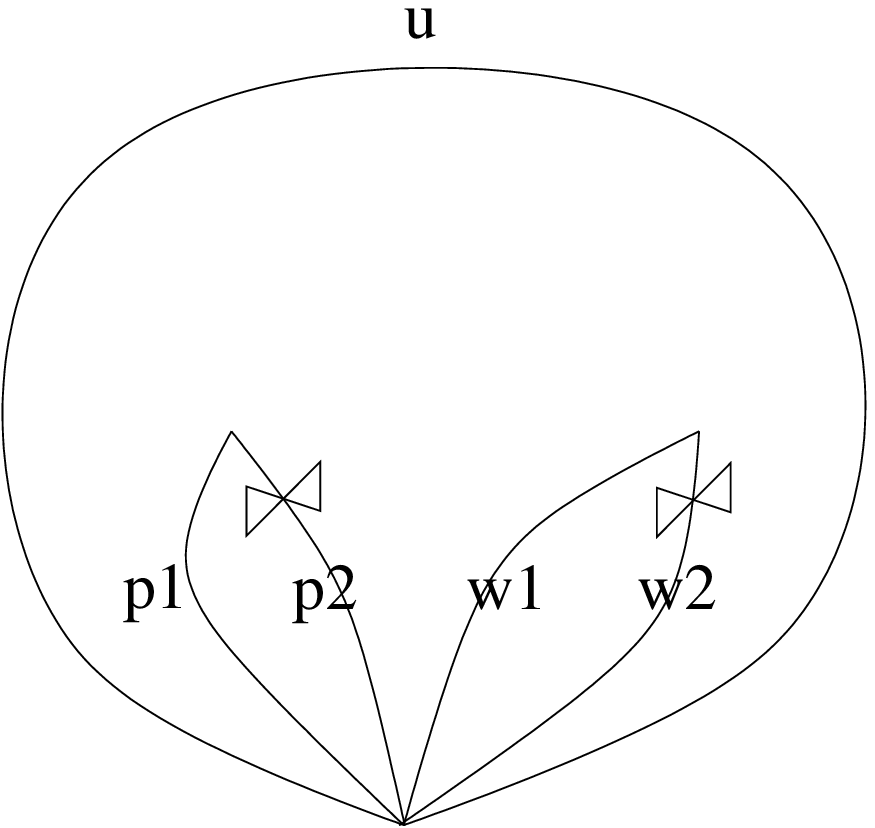,width=0.15\linewidth}}\\
%&&&&\\
&&&&\\
\psfrag{u}{\tiny $u$}
\psfrag{w}{\tiny $p$}
\psfrag{p}{\tiny $w$}
\psfrag{2-}{\tiny $2$}\psfrag{4}{\tiny $4$}
\raisebox{5mm}{\epsfig{file=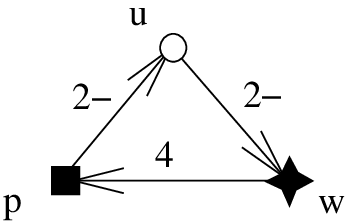,width=0.12\linewidth}}&&
\psfrag{u}{\tiny $u$}
\psfrag{r}{\tiny $w_2$}
\psfrag{w}{\tiny $p_2$}
\psfrag{p}{\tiny $w_1$}
\psfrag{q}{\tiny $p_1$}
\raisebox{3mm}{\epsfig{file=diagrams_pic/block5l.eps,width=0.12\linewidth}}&&
\psfrag{u}{\tiny $u$}
\psfrag{w1}{\tiny $w_1$}
\psfrag{w2}{\tiny $w_2$}
\psfrag{p1}{\tiny $p_1$}
\psfrag{p2}{\tiny $p_2$}
\raisebox{0mm}{\epsfig{file=diagrams_pic/s-block-VI.eps,width=0.15\linewidth}}\\
&&&&\\

%&&&&\\
\psfrag{u}{\tiny $u$}
\psfrag{r}{\tiny $r$}
\psfrag{w}{\tiny $w$}
\psfrag{p}{\tiny $p$}
\psfrag{q}{\tiny $q$}
\psfrag{2-}{\tiny $2$}
\raisebox{6mm}{\epsfig{file=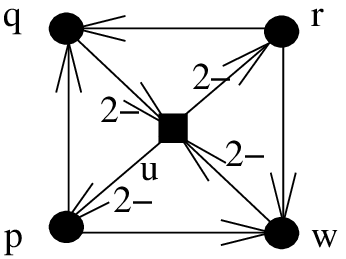,width=0.12\linewidth}}&&
\psfrag{u1}{\tiny $u_1$}
\psfrag{u2}{\tiny $u_2$}
\psfrag{r}{\tiny $r$}
\psfrag{w}{\tiny $w$}
\psfrag{p}{\tiny $p$}
\psfrag{q}{\tiny $q$}
\raisebox{0mm}{\epsfig{file=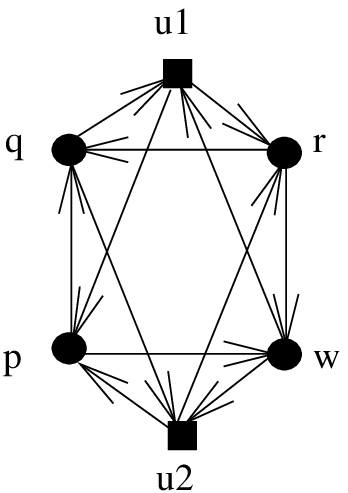,width=0.12\linewidth}}&&
\psfrag{w1}{\tiny $p$}
\psfrag{w2}{\tiny $r$}
\psfrag{u}{\tiny $u_1$}
\psfrag{v}{\tiny $u_2$}
\psfrag{p}{\tiny $w$}
\psfrag{q}{\tiny $q$}
\raisebox{0mm}{\epsfig{file=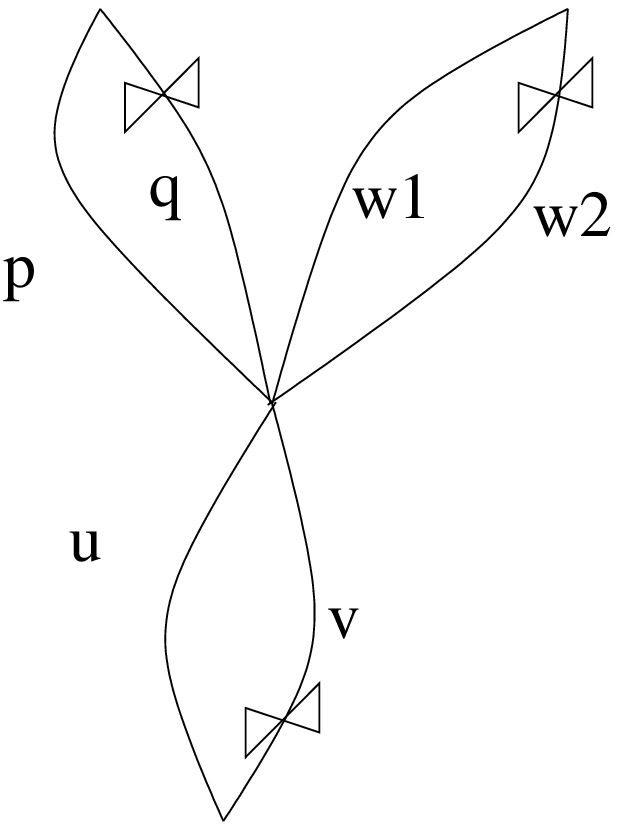,width=0.11\linewidth}}\\
%&&&&\\

\end{tabular}
\end{center}
\end{table}

\begin{remark}
\label{block6u}
The triangulation corresponding to the local unfolding of block $\t{\mr{VI}}$ has no decomposition into surfaces representing blocks of type ${\mr{I}}-{\mr{V}}$, and thus, does not correspond to any block decomposition of the unfolding diagram. Therefore, this triangulation occurs to be an exclusion from the theory derived in~\cite{FST}. 

In fact, similarly to block $\t{\mr{VI}}$, its local unfolding diagram has no outlets, so it cannot be used in any construction of further diagrams. This is the reason the authors of~\cite{FST} have made no use of that diagram as a block. For completeness of our theory, it is convenient to define the local unfolding of block $\t{\mr{VI}}$ (see Table~\ref{newblocks}~or~\ref{reg}) to be a skew-symmetric block of type ${\mr{VI}}$.  

\end{remark}

Note that any such local unfolding (except the last one) corresponds to the triangulation with two edges inside a digon (or monogon) representing the same isotopy class: one tagged plain and the other tagged notched. Let us call such pair of edges \emph{conjugate}. Conjugate pair of edges  represents two vertices of the unfolding diagram whose folding in s-decomposable diagram is exactly one vertex. Mutation of the folding vertex corresponds to the flips of the both edges from the conjugate pair. These flips do commute, and as a result we obtain again a triangulation where the corresponding edges form a conjugate pair. 

Similar to the notion of composite mutation for an unfolding diagram, we define a {\it composite flip} of a triangulation corresponding to an unfolding diagram as a collection of flips in all edges representing vertices whose folding is the same vertex. An example of a composite flip is a sequence of two flips in conjugate edges. Note that individual flips in a composite flip always mutually commute. 

Given an s-decomposable diagram $S_{\mr{dec}}$ with fixed block decomposition (different from block $\t{\mr{VI}}$), the considerations above allow us to construct a unique tagged triangulation of a marked bordered surface with chosen tuple of conjugate pairs. This surface (with triangulation) can be obtained by gluing of surfaces corresponding to local unfoldings of blocks of $S_{\mr{dec}}$, and we mark every conjugate pair that corresponds to one vertex in $S$. 
This construction is invariant under mutations of $S$: mutating $S$, the corresponding triangulation can be obtained from the initial one by corresponding composite flips.

Conversely, looking at tagged triangulations containing conjugate pairs (different from block ${\mr{VI}}$), one can easily see that every conjugate pair lies either inside a digon, or inside a monogon. Recalling the definition of block-decomposable diagram, this implies that the first case corresponds to blocks of types $\mr{{III}}$ and ${\mr{IV}}$, and the latter corresponds to blocks of type ${\mr{V}}$ (in this case there is another conjugate pair inside the same monogon). In other words, every such triangulation with arbitrary chosen tuple of conjugated pairs of edges can be obtained via local unfolding from some s-decomposable skew-symmetrizable diagram.  

Furthermore, every such triangulation with chosen conjugate pairs may come from a unique s-decomposable diagram (with fixed block decomposition) only. Indeed, given a triangulation, there is a unique way to distribute triangles, digons and monogons amongst blocks, which implies uniqueness of block decomposition of folding.

The case of block $\t{\mr{VI}}$ can be easily treated separately. Folding one of the three conjugate pairs of the triangulation corresponding to block ${\mr{VI}}$ leads to the diagram of block $\t{\mr{VI}}$. 

Summarizing the discussion above, we come to the following statement.

\begin{theorem}
\label{tr}
There is a one-to-one correspondence between s-decomposable skew-symmetrizable diagrams with fixed block decomposition and ideal tagged triangulations of marked bordered surfaces with fixed tuple of conjugate pairs of edges.

\end{theorem}  

The correspondence above is invariant under mutations: mutating a skew-symmetriz\-able diagram, the corresponding triangulation can be obtained from the initial one by corresponding composite flips.

%
%\begin{figure}[!h]
%\begin{center}
%\epsfig{file=diagrams_pic/sphere_4pts_triang.eps,width=7cm}
%\caption{Folded tagged triangulations of a sphere with four marked points up to relabelling of vertices and changing taggs}
%\label{fig:sphere_4pts_triang}
%\end{center}
%\end{figure}
%
%%%%%%%%%%%%%%%%%%%%%%%%%%%%%%%%%%%%%%%%%%%%%%%%%%%%%%%%%%%%%%%%%%%%%%%%%%%%%%%%%%%%%%%%%%%%%%%%%%%%%%%%%%%%%%%%%%%%%%%%%%

\section{Minimal non-decomposable diagrams}
\label{inf}

In this section we provide a polynomial-time criterion for a diagram to be mutation-finite by proving Theorem~\ref{crit}. The considerations are identical to ones used in~\cite[Section~7]{FST1}.

\begin{definition}
\label{definf}
A \emph{minimal mutation-infinite diagram} $S$ is a diagram that
\begin{itemize}
\item has infinite mutation class;
\item any proper subdiagram of $S$ is mutation-finite.
\end{itemize}
\end{definition}

Any minimal mutation-infinite diagram is connected. Notice that the property to be minimal mutation-infinite is not mutation invariant.
Note also that minimal mutation-infinite diagram of order at least $4$ does not contain edges of multiplicity greater than $4$.

We will deduce the criterion from the following lemma.

\begin{lemma}
\label{le10}

Any minimal mutation-infinite diagram contains at most $10$ vertices.

\end{lemma}

\begin{proof}
Let $S$ be a minimal mutation-infinite diagram.

First, we prove a weaker statement, i.e. we show that $|S|\le 11$. In fact, this bound follows immediately from Theorems~\ref{g8} and~\ref{all}.
Indeed, either all the proper subdiagrams of $S$ are block-decomposable, or $S$ contains a proper mutation-finite non-decomposable subdiagram of order $|S|-1$ (we can assume that this diagram is connected: if it is not connected but non-decomposable, it contains a non-decomposable connected component $S_0$, and any connected subdiagram of $S$ of order $|S|-1$ containing $S_0$ is non-decomposable). In the former case $|S|\le 7$ according to Theorem~\ref{g8} (again, we emphasize that we did not require $S$ to be mutation-finite in the assumptions of Theorem~\ref{g8}). In the latter case $|S|-1\le 10$ due to Theorem~\ref{all}, which proves inequality  $|S|\le 11$.

Now suppose that $|S|=11$. Then $S$ contains a proper finite mutational non-decomposable
subdiagram $S'$ of order $10$. According to Theorem~\ref{all}, $S'$ is mutation-equivalent
to $E_{10}^{(1,1)}$. The mutation class of $E_{10}^{(1,1)}$ consists of $5739$ diagrams,
which can be easily computed using Keller's \rm{{\tt Java}} applet~\cite{K}. In other words, we see that $S$
contains one of $5739$ diagrams of order $10$ as a proper subdiagram.

Hence, we can list all minimal mutation-infinite diagrams of order $11$ in the following
way. To each of $5739$ diagrams above we add one vertex in all possible ways (we can do that since
the weight of edge is bounded by $4$; the sources codes can be found in~\cite{progr}). For every obtained diagram we check whether all its proper subdiagrams of order $10$ (and, therefore, all the others) are mutation-finite. However, the resulting set of the procedure above is empty: every obtained diagram has at least one mutation-infinite subdiagram of order $10$, so it is not minimal.

\end{proof}

As a corollary of Lemma~\ref{le10}, we get the criterion for a diagram to be mutation-finite.

\begin{theorem}
\label{crit}
A diagram $S$ of order at least $10$ is mutation-finite if and only if all subdiagrams of $S$ of order $10$ are mutation-finite.

\end{theorem}

\begin{proof}%[Proof of Theorem~\ref{crit}]
According to Definition~\ref{definf}, every mutation-infinite diagram contains some minimal mutation-infinite diagram as a subdiagram. Thus, a diagram is mutation-finite if and only if it does not contain any minimal mutation-infinite subdiagram. By Lemma~\ref{le10}, this holds if and only if all subdiagrams of order at most $10$ are mutation-finite. Since a subdiagram of a mutation-finite diagram is also mutation-finite, the latter condition, in its turn, holds if and only if all subdiagrams of order $10$ are mutation-finite, which completes the proof.

\end{proof}

\begin{remark}
\label{sharp}
The bound in Lemma~\ref{le10} is sharp: as it was mentioned in~\cite{FST1}, there exist skew-symmetric minimal mutation-infinite diagrams of order $10$.
\end{remark}

Reformulating Theorem~\ref{crit} in terms of matrices, we obtain the following result.

\begin{theorem}
\label{min_inf_matr}
A skew-symmetrizable $n\times n$ matrix $B$, $n\ge 10$, has finite mutation class if and only if a mutation class of every principal $10\times 10$ submatrix of $B$ is finite.

\end{theorem}

%%%%%%%%%%%%%%%%%%%%%%%%%%%%%%%%%%%%%%%%%%%%%%%%%%%%%%%%%%%%%%%%%%%%%%%%%%%%%%%%%%%%%%%%%%%%%%%%%%%%%%%%%%%%%%%%%%%%%%%%%%
\appendix
\section{}
In this section we list the refinements to the proof of~\cite[Theorem~5.2]{FST1} which allow us to prove Theorem~\ref{g8}.
\subsection{Block decompositions: basic tools}
We reformulate statements from~\cite[Section~4]{FST1} in our settings.

First, we fix some notation we will use.

Let $S_1$ and $S_2$ be subdiagrams of $S$ having no common vertex. We say that $S_1$ and $S_2$ are {\it orthogonal} ($S_1\perp S_2$) if no edge joins vertices of $S_1$ and $S_2$.

For a vertex $v$ of $S$ by {\it valence} of $v$ in $S$ we mean the number of neighbors of $v$ in $S$ (i.e., unsigned valence: every edge is counted with a unit weight).

For two vertices $u_i,u_j$ of diagram $S$ we denote by $(u_i,u_j)$ a directed arc connecting $u_i$ and $u_j$ which may or may not belong to $S$. It may be directed  either way. By $(u_i,u_j,u_k)$ we denote oriented triangle with vertices $u_i,u_j,u_k$ which is oriented either way and whose edges also may or may not belong to $S$. We use standard notation $\l u_i,u_j\r$ for an edge of $S$.

We denote by $\B_{\rm{I}}$, $\B_{\rm{II}}$ etc. the isomorphism classes of blocks of types $\rm{I}$, $\rm{II}$, etc. respectively. For a block ${\mathsf{B}}$ we write ${\mathsf{B}}\in\B_{\rm{I}}$ if ${\mathsf{B}}$ is of type $\rm{I}$.

\begin{prop}
\label{razval1}
Let $S$ be a connected diagram with $n$ vertices, and let $b$ be a vertex of $S$ satisfying the following properties:

$(0)$ $S\setminus b$ is not connected;

$(1)$ for any $u\in S$ the diagram $S\setminus u$ is s-decomposable;

$(2)$ at least one connected component of $S\setminus b$ has at least $3$ vertices;

$(3)$ each connected component of $S\setminus b$ has at most $n-3$ vertices.

\smallskip
\noindent
Then $S$ is s-decomposable.

\end{prop}

\begin{proof}
The proof follows the proof of~\cite[Proposition~4.6]{FST1}. We divide $S\setminus b$ into two parts $S_1$ and $S_2$ in the following way: $S_1$ is any connected component of $S\setminus b$ with at least $2$ vertices (it exists by assumption $(2)$), and $S_2=S\setminus \l b,S_1\r$.

Now choose points $a_1\in S_1$ and $a_2\in S_2$ satisfying the following conditions: $S\setminus a_i$ is connected, and $S\setminus a_i$ does not contain leaves attached to $b$ and belonging to $S_1$. We always can take as $a_2$ a vertex of $S_2$ at the maximal distance from $b$. To choose $a_1\in S_1$, we look at the valence of $b$ in $\l S_1,b\r$ and structure of $S_1$. If either there is exactly one vertex of $S_1$ joined with $b$, or $S_1$ has no leaves (as a diagram), we choose $a_1$ as a vertex of $S_1$ being at the maximal distance from $b$. If there are at least two vertices of $S_1$ joined with $b$ and there is a leaf $x$ of $S_1$, then we take $x$ as $a_1$. Clearly, $a_1$ and $a_2$ chosen in this way satisfy required conditions.

We need to prove that each $\l S_i,b\r$ is s-decomposable with outlet $b$. For that, we consider the diagram $S\setminus a_2$, choose any its decomposition into blocks, and prove that for any block ${\mathsf{B}}$ either $\mathsf{B}\cap S_1=\emptyset$ or ${\mathsf{B}}\cap (S_2\setminus a_2)=\emptyset$. Let us consider all possible types of block ${\mathsf{B}}$.

\smallskip
\noindent
{\bf Case 1:} ${\mathsf{B}}\in\B_{\mr{III}},\B_{\mr{IV}},\B_{\mr{V}},\B_{\t{\mr{III}}},\B_{\t{\mr{V}}_1},\B_{\t{\rm{V}}_2},\B_{\t{\rm{V}}_{12}}$.\\
The proof repeats the proofs of Cases~1 and~2 of~\cite[Proposition~4.6]{FST1}.

\smallskip
\noindent
{\bf Case 2:} ${\mathsf{B}}\in\B_{\t{\mr{IV}}},\B_{\rm{II}},\B_{\rm{I}}$.\\
The proof repeats the proof of Cases~3 (for the first two types) and~4 (for the latter type) of~\cite[Proposition~4.6]{FST1}.

\end{proof}

We say that a leaf $x$ of a diagram $S$ is {\it simple} if a unique edge emanating from $x$ is simple (i.e., of unit weight).

\begin{prop}
\label{razval2}

Let $S$ be a connected diagram $S=\l S_1,b_1,b_2,S_2\r$, where $S_1\perp S_2$, and $S$ has at least $8$ vertices. Suppose that

$(0)$ $b_1$ and $b_2$ are not joined in $S$;

$(1)$ for any $u\in S$ the diagram $S\setminus u$ is s-decomposable;

$(2)$ there exist $a_1\in S_1,a_2\in S_2$ such that

${}$\phantom{w}  $(2a)$ $S\setminus a_i$ is connected;

${}$\phantom{w}    $(2b)$ either $\l S_i,b_1,b_2\r\setminus a_i$ or $\l S_j,b_1,b_2\r$ (for $i,j=1,2,\ j\ne i$)  contains no simple leaves attached to $b_1$;

\qquad\quad similarly, either $\l S_i,b_1,b_2\r\setminus a_i$ or $\l S_j,b_1,b_2\r$ (for $j\ne i$) contains no simple leaves attached to $b_2$;

%%%%%%%%%%%%%

${}$\phantom{w}  $(2c)$ if $a_i$ is joined with $b_j$ (for $i,j=1,2$), then there is another vertex $w_i\in S_i$ attached to $b_j$.

\smallskip

\noindent
Then $S$ is s-decomposable.

\end{prop}

\begin{proof}
The proof follows the proof of~\cite[Proposition~4.8]{FST1}.
First, we show that for any decomposition of  $S\setminus a_2$ any block ${\mathsf{B}}$ is contained entirely either in $\l S_1,b_1,b_2\r$ or in $\l S_2,b_1,b_2\r\setminus a_2$. For this, we consider any decomposition of $S\setminus a_2$, assuming that for a block ${\mathsf{B}}$ both intersections ${\mathsf{B}}\cap S_1$ and ${\mathsf{B}}\cap (S_2\setminus a_2)$ are not empty. We consider all possible types of block ${\mathsf{B}}$
%(see Table~\ref{razval2_})
and obtain contradiction for each type.

\medskip
\noindent
{\bf Case 1:} ${\mathsf{B}}\in\B_{\rm{V}},\B_{\rm{III}},\B_{\t{\mr{III}}},\B_{\t{\rm{V}}_1},\B_{\t{\rm{V}}_2},\B_{\t{\rm{V}}_{12}}$. \
The proof is the same as in Proposition~\ref{razval1}.

\medskip
\noindent
{\bf Case 2:} ${\mathsf{B}}\in\B_{\rm{IV}}$. \
The proof is very similar to the proof of Case~2 of~\cite[Proposition~4.8]{FST1}. In the case when $b_1,b_2$ are the dead ends of ${\mathsf{B}}$, the only difference is that, while considering complementary block ${\mathsf{B}}_1$, we need to allow it to be of type $\t{\rm{IV}}$, too. The consideration of that type itself does not differ from consideration of ${\mathsf{B}}_1$ of type $\t{\rm{III}}$.% (??????????).

In the case when $b_1,b_2$ are the outlets of ${\mathsf{B}}$, the only difference is in possibility of gluing a block ${\mathsf{B}}_1$ of type $\t{\rm{IV}}$ along the edge $(b_1,b_2)$. As a result we get a diagram with $5$ vertices without outlets in contradiction to connectedness of $S\setminus a_2$.

\medskip
\noindent
{\bf Case 3:} ${\mathsf{B}}\in\B_{\t{\rm{IV}}}$. This case is new.\
First, we note that a unique dead end of ${\mathsf{B}}$ must conside with one of $b_1$ and $b_2$, say $b_1$. Denote the outlets of ${\mathsf{B}}$ by $w_1\in S_1$ and $w_2\in S_2\setminus a_2$. Then, to avoid the edge $(w_1,w_2)$ in $S$, a block ${\mathsf{B}}_1$ should be glued along $(w_1,w_2)$. ${\mathsf{B}}_1$ may be of type ${\rm{I}},{\t{\rm{IV}}}$, or ${\rm{II}}$. In the first two cases we get a diagram with $3$ or $4$ vertices without outlets, which contradicts connectedness of $S\setminus a_2$.

In the latter case, we may assume that the third vertex of ${\mathsf{B}}_1$ is $b_2$. Then $b_2$ is the unique outlet of the union of ${\mathsf{B}}$ and ${\mathsf{B}}_1$. Since $|S_1|\ge 2$, $b_2$ is contained in some block ${\mathsf{B}}_2$, where all the vertices of ${\mathsf{B}}_2\setminus b_2$ belong to $S_1$. Furthermore, notice that no vertex of $S\setminus\l b_1,b_2\r$ is joined with $w_1$, and no vertex of $S\setminus\l b_1,b_2,a_2\r$ is joined with $w_2$. Since $|S\setminus\l b_1,b_2,w_1,w_2,a_2\r|\ge 3$, we conclude that $S$ is s-decomposable by Proposition~\ref{razval1} applied to $b_2$.

\medskip
\noindent
{\bf Case 4:} ${\mathsf{B}}\in\B_{\t{\rm{II}}}$. \
The proof follows the proof of Case~3 of~\cite[Proposition~4.8]{FST1}. Let vertices of ${\mathsf{B}}$ be $w_1\in S_1, w_2\in S_2\setminus a_2$, and $b_1$. We may assume that there is a block ${\mathsf{B}}_1$ of second type with vertices $w_1,w_2,b_2$. We may also assume that there is a vertex $t_1\in S_1$ distinct from $w_1$ attached to $b_1$, and there is a block ${\mathsf{B}}_2$ containing $b_2$. The proof splits into two cases: ${\mathsf{B}}_2$ is entirely contained either in $\l S_1,b_1,b_2\r$ or in $\l b_1,b_2,S_2\setminus a_2\r$.

\smallskip
\noindent
{\bf Case 4.1:} ${\mathsf{B}}_2$ is contained in $\l S_1,b_1,b_2\r$. \
The proof repeats the proof of Case~3.1 of~\cite[Proposition~4.8]{FST1}. We just need to substitute all occurrences of ``leaf'' by ``simple leaf''. Also in case~3.1.2 of~\cite[Proposition~4.8]{FST1} vertex $r_1$ may not exist, but then the edge $(t_1,b_1)$ is not simple. In this case we take as $S'$ the subdiagram $S'\!=\!\l t_1,b_1,w_1,w_2,a_2\r$ which is mutation-infinite.

\smallskip
\noindent
{\bf Case 4.2:} ${\mathsf{B}}_2$ is contained in $\l b_1,b_2,S_2\setminus a_2\r$. \
As in the proof of Case~3.2 of~\cite[Proposition~4.8]{FST1}, we can assume that $t_1$ is a leaf of $S$ (it may not be simple), and $a_2$ is attached to $w_2$ by non-double edge. We can also assume that $a_2$ is joined with some $t_2\in S\setminus\l t_1,b_1,w_1,w_2,b_2\r$. Now we take any decomposition of $S\setminus t_1$ and consider all possible types of blocks (with at least $3$ vertices) containing $w_2$ (taking into account that the only vertices joined with $w_2$ are $b_1,b_2$ and $a_2$).

\smallskip
\noindent
{\bf Case 4.2.1:} $w_2$ lies in block ${\mathsf{B}}_3$ of type ${\rm{V}}$. \
See the proof of Case~3.2.1 of~\cite[Proposition~4.8]{FST1}.

\smallskip
\noindent
{\bf Case 4.2.2:} $w_2$ lies in block ${\mathsf{B}}_3$ of type ${\t{\mr{V}}_1}$ or ${\t{\mr{V}}_2}$. \
Due to its valence and the fact that only one edge emanating from $w_2$ may not be simple, $w_2$ is an outlet of ${\mathsf{B}}_3$. Thus, orientations of edges $(w_2,b_1)$ and $(w_2,b_2)$ must coincide, which does not hold.

\smallskip
\noindent
{\bf Case 4.2.3:} $w_2$ lies in block ${\mathsf{B}}_3$ of type ${\t{\mr{V}}_{12}}$. \
Since $w_2$ is incident to three edges, $w_2$ is an outlet of ${\mathsf{B}}_3$. Therefore, $w_2$ is incident to at least $2$ non-simple edges, which is not true.

\smallskip
\noindent
{\bf Case 4.2.4:} $w_2$ is contained in block ${\mathsf{B}}_3$ of type ${\rm{IV}}$. \
Due to its valence, $w_2$ is an outlet of ${\mathsf{B}}_3$. Consider two cases.

\smallskip
\noindent
{\bf Case 4.2.4.1:} $w_2$ is contained in block ${\mathsf{B}}_3$ only. \
See the proof of Case~3.2.2.1 of~\cite[Proposition~4.8]{FST1}.

\smallskip
\noindent
{\bf Case 4.2.4.2:} $w_2$ is contained simultaneously in two blocks ${\mathsf{B}}_3$ and ${\mathsf{B}}_4$, ${\mathsf{B}}_4\ne {\mathsf{B}}_3$. \
Block ${\mathsf{B}}_4$ is of type ${\rm{II}}$ or ${\t{\rm{IV}}}$. In the latter case orientations of edges $\l w_2,b_1\r$ and $\l w_2,b_2\r$ coincide, so we get a contradiction. In the first case, by the same reason $a_2$ is a dead end of ${\mathsf{B}}_3$. This implies that valence of $a_2$ is $2$,
so only $t_2$ can be outlet of ${\mathsf{B}}_3$. The second dead end of ${\mathsf{B}}_3$ should be joined
with both $t_2$ and $w_2$. Since $b_1$ is not joined with $t_2$, $b_2$ is a dead end of ${\mathsf{B}}_3$.
But this contradicts existence of the edge joining $b_2$ and $w_1$.

\smallskip
\noindent
{\bf Case 4.2.5:} $w_2$ is contained in block ${\mathsf{B}}_3$ of type ${\t{\rm{IV}}}$. \
In this case $w_2$ is contained in block ${\mathsf{B}}_4$ of type ${\rm{I}}$, and $a_2$ is the dead end of ${\mathsf{B}}_3$, while $w_2$ and one of $b_1,b_2$ are the outlets of ${\mathsf{B}}_3$. But this contradicts the existence of the edge $\l a_2,t_2\r$.

\smallskip
\noindent
{\bf Case 4.2.6:} $w_2$ is contained in block ${\mathsf{B}}_3$ of type ${\rm{III}}$. \
Due to orientation of edges, $w_2$ is the outlet of ${\mathsf{B}}_3$, and at least one of $b_1$ and $b_2$ is a dead end of ${\mathsf{B}}_3$, hence a leaf of $S\setminus t_1$. But neither $b_1$ nor $b_2$ is a leaf since they are joined with $w_1$.

\smallskip
\noindent
{\bf Case 4.2.7:} $w_2$ is contained in block ${\mathsf{B}}_3$ of type ${\rm{II}}$. \
In this case $w_2$ is also contained in block ${\mathsf{B}}_4$ of type ${\rm{I}}$ or ${\t{\mr{III}}}$.
In the latter case $a_2$ must be a leaf of $S\setminus t_1$, which contradicts existence of edge $\l a_2,t_2\r$.
For the proof of the first case see Case~3.2.4 of~\cite[Proposition~4.8]{FST1}.

\medskip
\noindent
{\bf Case 5:} ${\mathsf{B}}\in\B_{\rm{I}}$. \ The proof is the same as in Proposition~\ref{razval1}.

\medskip

The rest of the proof repeats the proof of the~\cite[Proposition~4.8]{FST1}.
In few cases we need to consider blocks of type ${\t{\mr{IV}}}$ together with types ${\rm{II}}$ and ${\rm{IV}}$, but this requires only minor changes in the proof: while substituting block of fourth type, we lose a vertex, and while substituting block of second type, we substitute an outlet by a dead end.

\end{proof}

\begin{cor}
\label{after2}
Suppose that $S=\l S_1,b_1,b_2,S_2\r$ satisfies all the assumptions of Proposition~\ref{razval2} except $(2)$. Suppose also that $|S_1|\ge 2$, $|S_2|\ge 3$, and there exists $c_1\in S_1$ such that the following holds:

$(a)$ $S_1\setminus c_1$ is connected;

$(b)$ $S_1$ contains no leaves of $S$ attached to  $b_1$ or $b_2$,  and $S_1\setminus c_1$ contains no leaves of $S\setminus c_1$ attached to  $b_1$ or $b_2$;

$(c)$ $S_1\setminus c_1$ is attached to both $b_1$ and $b_2$.

Then $S$ is s-decomposable.

\end{cor}

The proof repeats the proof of~\cite[Corollary~4.9]{FST1}.

\subsection{Minimal non-decomposable diagrams}

In this section, we generalize results of~\cite[Section~5]{FST1}.
We recall the definition of minimal non-decomposable diagram.

A \emph{minimal non-decomposable diagram} $S$ is a diagram that
\begin{itemize}
\item is non-decomposable;
\item for any $u\in S$ the diagram $S\setminus u$ is s-decomposable.
\end{itemize}
As before, any minimal non-decomposable diagram is connected.

\medskip
\noindent
{\bf Theorem~5.1.}
{\it Any minimal non-decomposable diagram contains at most $7$ vertices.}

\medskip

The proof follows the proof of Theorem~5.2 from~\cite{FST1}. We assume that there exists a diagram $S$ of order at least $8$ satisfying the assumptions of Theorem~\ref{g8}, and show for each type of block that if an s-decomposable subdiagram $S\setminus u$ contains block of this type then $S$ is also s-decomposable.

Throughout this section we assume that $S$ satisfies the assumptions of Theorem~\ref{g8} (and $|S|\ge 8$). We do not assume the mutation class of $S$ to be finite.

\begin{lemma}
\label{no5}
For any $x\in S$ any block decomposition of $S\setminus x$ does not contain blocks of type ${\rm{V}}$, ${\t{\mr{V}}_1}$, ${\t{\mr{V}}_2}$ and ${\t{\mr{V}}_{12}}$.

\end{lemma}

To prove the lemma we use the following proposition.

\begin{prop}
\label{5is5}
Suppose that $S\setminus x$ contains a subdiagram $S_1$ consisting of a block ${\mathsf{B}}$ of type ${\rm{V}}$ (or ${\t{\mr{V}}_1}$, ${\t{\mr{V}}_2}$, ${\t{\mr{V}}_{12}}$) with outlet $v$ and dead ends $v_1,\dots,v_k$, and a vertex $t$ joined with $v$ (and probably with some of $v_i$). Then for any $u\in S\setminus S_1$
and any block decomposition of $S\setminus u$ a subdiagram $\l v,v_1,\dots,v_k\r$ is contained in one block of type ${\rm{V}}$ (or ${\t{\mr{V}}_1}$, ${\t{\mr{V}}_2}$, and ${\t{\mr{V}}_{12}}$ respectively). In particular, $t$ does not attach to any of $v_i$, $i=1,\dots,k$.

\end{prop}

\begin{proof}
Take any $u\in S\setminus S_1$ and consider any block decomposition of $S_2=S\setminus u$. Since valence of $v$ in $S_2$ is at least $k+1$ (and $v$ is contained in at least $4-k$ edges of weight $2$), $v$ is contained in exactly two blocks ${\mathsf{B}}_1$ and ${\mathsf{B}}_2$, at least one of which is of the type ${\rm{V}}$ or ${\rm{IV}}$ (or ${\t{\mr{V}}_1}$, ${\t{\mr{V}}_2}$, ${\t{\mr{V}}_{12}}$, ${\t{\mr{IV}}}$). Suppose that none of ${\mathsf{B}}_1$ and ${\mathsf{B}}_2$ is of the type ${\rm{V}}$ (or ${\t{\mr{V}}_1}$, ${\t{\mr{V}}_2}$, ${\t{\mr{V}}_{12}}$ respectively; notice that because of valence and orientation of edges, the types can not mix), and let ${\mathsf{B}}_1$ be of the type ${\rm{IV}}$ (or ${\t{\mr{IV}}}$). Then for any choice of ${\mathsf{B}}_2$ we have the following:\\

$-$ the number of vertices of $S_2$ which are neighbors of $v$ and  have valence at least three in $S_2$ does not exceed $3$, which means ${\mathsf{B}}$ is not of type ${\rm{V}}$;

$-$ no neighbor of $v$ in $S_2$ is incident to three edges of weight $2$ only, which means ${\mathsf{B}}$ is not of type ${\t{\mr{V}}_1}$ and ${\t{\mr{V}}_2}$;

$-$ no neighbor of $v$ in $S_2$ is incident to one edge of weight $2$ and one double edge only, which means ${\mathsf{B}}$ is not of type ${\t{\mr{V}}_{12}}$.

The contradiction implies that we may assume ${\mathsf{B}}_1$ to be of the type ${\rm{V}}$ (or ${\t{\mr{V}}_1}$, ${\t{\mr{V}}_2}$, and ${\t{\mr{V}}_{12}}$ respectively) with outlet $v$. Now consider four types of blocks separately.

If block ${\mathsf{B}}$ of $S\setminus x$ is of type ${\rm{V}}$ (and so is ${\mathsf{B}}_1$), then the proof repeats the proof of Proposition~5.4 from~\cite{FST1}.

If block ${\mathsf{B}}$ of $S\setminus x$ is of type ${\t{\mr{V}}_{1}}$ or ${\t{\mr{V}}_{2}}$, then the subdiagram $\l v_1,v_2,v_3\r\subset S$ consists of two edges of weight $2$.
At the same time, the link $L_{S_2}(v)$ is a disjoint union of a diagram composed of two edges of weight $2$ only having a vertex in common (composed by dead ends of ${\mathsf{B}}_1$) and another diagram with at most $4$ vertices (composed by vertices of ${\mathsf{B}}_2\setminus v$). If we assume that $v_1,v_2,v_3$ are not contained in one block (${\mathsf{B}}_1$ or ${\mathsf{B}}_2$) in $S_2$, then we come to a contradiction. Clearly, the only block with a subdiagram consisting of two edges of weight $2$ (and nothing else) is of type ${\t{\mr{V}}_{1}}$ or ${\t{\mr{V}}_{2}}$.

Finally, suppose that block ${\mathsf{B}}$ of $S\setminus x$ is of type ${\t{\mr{V}}_{12}}$. Again, consider the subdiagram $\l v_1,v_2\r\subset S$, it consists of a double edge. Since ${\mathsf{B}}_1$ is also of type ${\t{\mr{V}}_{12}}$, the link $L_{S_2}(v)$ is a disjoint union of a double edge and another diagram with at most $4$ vertices. No link of outlet contains a double edge except block of type ${\t{\mr{V}}_{12}}$, so we complete the proof.

\end{proof}

Now the proof of the lemma repeats the proof of Lemma~5.3 from~\cite{FST1}.

\begin{lemma}
\label{no4}
For any $x\in S$ no block decomposition of $S\setminus x$ contains blocks of type ${\rm{IV}}$.

\end{lemma}

The proof repeats the proof of Lemma~5.5 from~\cite{FST1}.

\begin{cor}
\label{valle4}
Valence of any vertex $v$ of a minimal non-decomposable diagram  $S$ does not exceed $4$.

\end{cor}

\smallskip

Consider now a block of type $\t{\rm{IV}}$. We will prove its absence in decompositions of subdiagrams of $S$ in two steps.

\begin{prop}
\label{no3in4t}
Suppose that some block decomposition of $S\setminus x$ contains blocks of type $\t{\rm{IV}}$ with outlet $v$ and dead ends $v_1, v_2$. Then

$(a)$ $x\perp v$;

$(b)$ for any $u\in S\setminus\l v,v_1,v_2,x\r$ and any decomposition of $S\setminus u$ vertices $\l v,v_1,v_2\r$ form a block of type $\t{\rm{IV}}$;

$(c)$ for any $u\in S\setminus\l v,v_1,v_2\r$ either $u\perp v_1$ or $u\perp v_2$.

\end{prop}

\begin{proof}
$(a)$ \ Suppose that $x$ is joined with $v$. Take any vertex $w\in S\setminus\l v,v_1,v_2,x\r$ and consider $S_1=S\setminus w$ with some decomposition. Since $v$ is contained in exactly three edges (and at least two of which are of weight $2$), we see that either $v$ is contained in block of type ${\t{\mr{V}}_1}$, ${\t{\mr{V}}_2}$ or ${\t{\mr{V}}_{12}}$ (which is impossible by Lemma~\ref{no5}), or $v$ is contained in one block of type ${\t{\mr{IV}}}$ and one block of type ${\t{\mr{III}}}$.

Let $\l v_1,v,x\r$ compose a block  of type ${\t{\mr{IV}}}$ in the decomposition of $S_1$, and $\l v_2,v\r$ compose a block  of type ${\t{\mr{III}}}$. Then $v_1$ and $x$ are joined by an edge of weight $2$, and $v_1$ and $v_2$ are not joined in $S_1$ (so, in $S$). This means that some block ${\mathsf{B}}$ of type ${\t{\mr{IV}}}$, ${\rm{II}}$ or ${\rm{I}}$ is glued in $S\setminus x$ along the edge $\l v_1,v_2\r$. Since $v_1$ and $v_2$ are dead ends of their blocks in $S_1=S\setminus w$, the only vertex which can attach to $v_1$ or $v_2$ is $w$.

If ${\mathsf{B}}$ is of type ${\rm{I}}$, then $S$ is s-decomposable by Proposition~\ref{razval1} applied to $x$. If ${\mathsf{B}}$ is of type $\t{\rm{IV}}$, then $v_1$ is contained in at least three edges of weight $2$, so in any decomposition of $S\setminus v_2$ it should be contained in block of of type ${\t{\mr{V}}_1}$, ${\t{\mr{V}}_2}$ or ${\t{\mr{V}}_{12}}$, which contradicts Lemma~\ref{no5}. Therefore, ${\mathsf{B}}$ is of type ${\rm{II}}$.
%(see Fig.~\ref{no4tf}).

%\begin{figure}[!h]
%\begin{center}
%\psfrag{v1}{\tiny $v_1$}
%\psfrag{v2}{\tiny $v_2$}
%\psfrag{w}{\tiny $w$}
%\psfrag{w2}{\tiny $w_2$}
%\psfrag{x}{\tiny $x$}
%\psfrag{or}{\small or}
%\epsfig{file=diagrams_pic/bl4_.eps,width=0.65\linewidth}
%\caption{To the proof of Prop.~\ref{no3in4t}(a).}
%\label{no4tf}
%\end{center}
%\end{figure}

Now consider the diagram $S\setminus v_2$ with some decomposition. Vertex $v_1$ is contained in exactly three edges, two of them are of weight $2$. Due to orientations,
%(see Fig.~\ref{no4tf}),
$\l w,v,v_1\r$ must compose a block of type $\t{\rm{IV}}$, which is impossible since $w$ is not joined with $v$.

\smallskip
\noindent
$(b)$ \ According to $(a)$, $v$ is incident to exactly two edges, each of them is of weight $2$. Suppose that $\l v,v_1,v_2\r$ do not compose a block of type $\t{\rm{IV}}$. Then $v$ is contained either in two blocks of type $\t{\rm{III}}$, or in two blocks of type $\t{\rm{IV}}$ glued along simple edge. In the both cases the union of these two blocks has no outlets, so $S$ is s-decomposable by Proposition~\ref{razval1} applied to $u$.

\smallskip
\noindent
$(c)$ \ Suppose that some $u\in S\setminus\l v,v_1,v_2\r$ is joined with both $v_1$ and $v_2$. Since $|S|\ge 8$, valence of any vertex does not exceed $4$, and both $v_1$ and $v_2$ are joined with $u$, there exist at least two vertices ($y$ and $z$) which are not joined with any of $v_1$ and $v_2$. Consider $S\setminus y$ with some decomposition. Due to $(b)$, vertices $\l v,v_1,v_2\r$ compose a block ${\mathsf{B}}$ of type $\t{\rm{IV}}$ with outlet $v$. Notice also that $y$ is not attached to $\l v,v_1,v_2\r$.

Consider all cases to join $u$ with $v_1$ and $v_2$ by attaching different blocks to $v_1$ and $v_2$. First, suppose $\l u,v_1,v_2\r$ belong to one block ${\mathsf{B}}_1$. Then
the diagram $\l u,v_1,v_2,v\r$ either has no outlets (if ${\mathsf{B}}_1$ is of type $\t{\rm{IV}}$) or has a unique outlet $u$ (if ${\mathsf{B}}_1$ is of type ${\rm{II}}$). This implies that $S$ is s-decomposable by Proposition~\ref{razval1} applied to $y$ or $u$ respectively.

Now suppose that edges $\l u,v_1\r$ and $\l u,v_2\r$ belong to distinct blocks ${\mathsf{B}}_1$ and ${\mathsf{B}}_2$. These blocks can be of types $\t{\rm{IV}}$, ${\rm{I}}$ or ${\rm{II}}$.
It is easy to see that for two different pairs of blocks $({\mathsf{B}}_1,{\mathsf{B}}_2)$ and $({\mathsf{B}}_1',{\mathsf{B}}_2')$ the pairs of links $(L_S(v_1),L_S(v_2))$ of $v_1$ and $v_2$ will be different, and the union of these two links will contain all the vertices of ${\mathsf{B}}_1$ and ${\mathsf{B}}_2$ distinct from $v_1$ and $v_2$. This means that decompositions of the union of ${\mathsf{B}}$, ${\mathsf{B}}_1$ and ${\mathsf{B}}_2$ into blocks in $S\setminus z$ will be the same as in $S\setminus y$. In particular, if we have proved that no vertex except $y$ is not joined with some vertex of the union of blocks, then $y$ is not joined with that vertex either.

If one of ${\mathsf{B}}_1$ and ${\mathsf{B}}_2$ (say, ${\mathsf{B}}_1$) is of the first type, then the union of blocks ${\mathsf{B}}$, ${\mathsf{B}}_1$ and ${\mathsf{B}}_2$ contains at most $5$ vertices and has at most one outlet, so $S$ is either s-decomposable by Proposition~\ref{razval1} applied to the outlet of ${\mathsf{B}}_2$ (if any) or disconnected (otherwise). If one of ${\mathsf{B}}_1$ and ${\mathsf{B}}_2$ (say, ${\mathsf{B}}_1$) is of type $\t{\rm{IV}}$, then the union of blocks ${\mathsf{B}}$, ${\mathsf{B}}_1$ and ${\mathsf{B}}_2$ contains at most $6$ vertices and has at most one outlet, so, again, $S$ is s-decomposable by Proposition~\ref{razval1} applied to the outlet of ${\mathsf{B}}_2$ (if any) or disconnected (otherwise). Therefore, we can assume that both ${\mathsf{B}}_1$ and ${\mathsf{B}}_2$ are of second type. Moreover, we can assume that they have the only common vertex $u$, otherwise $S$ is disconnected as above.

If $|S|=8$ (i.e., there are exactly two vertices not contained in the union of ${\mathsf{B}}$, ${\mathsf{B}}_1$ and ${\mathsf{B}}_2$), then a short direct check shows that $S$ is either s-decomposable or contains a mutation-infinite subdiagram. So, assume that $|S|\ge 9$, and denote by $w_1$ and $w_2$ the remaining vertices of ${\mathsf{B}}_1$ and ${\mathsf{B}}_2$.
If $w_1$ and $w_2$ are not joined in $S$, then $S$ is s-decomposable by Proposition~\ref{razval2} applied to $S=\left\l S_1=\l v,v_1,v_2,u\r,b_1=w_1,b_2=w_2\right.$, $\left.S_2=S\setminus\l S_1,w_1,w_2\r\right\r$. Thus, there exists some block ${\mathsf{B}}_3$ of type $\t{\rm{IV}}$, ${\rm{I}}$ or ${\rm{II}}$ containing vertices $w_1$ and $w_2$. In the first two cases the union of four blocks has no outlets, so $S$ is disconnected. Therefore, ${\mathsf{B}}_3$ is of second type. Denote by $w$ its remaining vertex.
% (see Fig.~\ref{no4tf2}).

%\begin{figure}[!h]
%\begin{center}
%\psfrag{v1}{\tiny $v_1$}
%\psfrag{v2}{\tiny $v_2$}
%\psfrag{w}{\tiny $w$}
%\psfrag{w2}{\tiny $w_2$}
%\psfrag{x}{\tiny $x$}
%\psfrag{or}{\small or}
%\epsfig{file=diagrams_pic/bl4_.eps,width=0.65\linewidth}
%\caption{To the proof of Prop.~\ref{no3in4t}(c).}
%\label{no4tf2}
%\end{center}
%\end{figure}

Since valence of any vertex in $S$ does not exceed four, $y$ is not joined with any of $w_1, w_2$, so $S$ is s-decomposable by Proposition~\ref{razval1} applied to $w$.

\end{proof}

\begin{lemma}
\label{no4t}
For any $x\in S$ no block decomposition of $S\setminus x$ contains blocks of type $\t{\rm{IV}}$.

\end{lemma}

\begin{proof}
Let ${\mathsf{B}}$ be a block of type $\t{\rm{IV}}$ in the decomposition of $S\setminus x$, denote by $v$ its dead end, and by $v_1,v_2$ its outlets. We can also assume that $x$ is not joined with any of $v_1$ and $v_2$. By Proposition~\ref{no3in4t}(c), $v_1$ and $v_2$ are joined by a simple edge in $S$. Indeed, if some block ${\mathsf{B}}'$ is glued to ${\mathsf{B}}$ along the edge $(v_1,v_2)$, then either $S$ is disconnected (if ${\mathsf{B}}'$ is of the first type), or there is a vertex joined with both $v_1$ and $v_2$.

Consider the diagram $S\setminus v$ with any decomposition. Our aim is to prove that the edge $\l v_1,v_2\r$ forms a block of first type. Then, substituting this block by block ${\mathsf{B}}$, we get an s-decomposition of $S$ (due to Proposition~\ref{no3in4t}, $v$ is joined in $S$ with $v_1$ and $v_2$ only).

Suppose that $\l v_1,v_2\r$ belongs to some block ${\mathsf{B}}_1$ containing more than two vertices, i.e. there exists a third vertex $u\in {\mathsf{B}}_1$. According to Proposition~\ref{no3in4t}(c), ${\mathsf{B}}_1$ is either of second or third type.

If ${\mathsf{B}}_1$ is of type ${\rm{III}}$, then $u$ and one of $v_1,v_2$ (say, $v_1$) are dead ends of ${\mathsf{B}}_1$, and the remaining vertex ($v_2$) is outlet. This implies that $S$ is s-decomposable by Proposition~\ref{razval1} applied to $v_2$.

Therefore, ${\mathsf{B}}_1$ is of type ${\rm{II}}$. According to Proposition~\ref{no3in4t}(c), $u$ is not joined with one of $v_1$ and $v_2$ (say, $v_1$), so there is a block ${\mathsf{B}}_2$ glued to ${\mathsf{B}}_1$ along the edge $(u,v_1)$. Clearly, ${\mathsf{B}}_2$ is of type $\t{\rm{IV}}$, ${\rm{I}}$ or ${\rm{II}}$. In the first two cases $S$ is s-decomposable by Proposition~\ref{razval1} applied to $v_2$ (which a unique outlet of the union of ${\mathsf{B}}$, ${\mathsf{B}}_1$ and ${\mathsf{B}}_2$). In the latter case $S$ is s-decomposable by Proposition~\ref{razval2} applied to $S=\left\l S_1=\l v,v_1,u\r,b_1=w,b_2=v_2,S_2=S\setminus\l S_1,w,v_2\r\right\r$, where $w$ is the remaining vertex of ${\mathsf{B}}_2$.

\end{proof}

Finally, we reduce the proof to the skew-symmetric case by proving the following lemma.

\begin{lemma}
\label{no3t}
$S$ contains no edges of weight $2$.

\end{lemma}

\begin{proof}
Suppose that $S$ contains an edge $\l v_1,v_2\r$ of weight $2$. For any vertex $x\in S\setminus\l v_1,v_2\r$, the edge $\l v_1,v_2\r$ form a block of type $\t{\rm{III}}$ in any decomposition of $S\setminus x$. In particular, one of $v_1$ and $v_2$ (say, $v_1$) should be a leaf of $S$, and valence of $v_2$ in $S$ does not exceed $3$. If valence of $v_2$ in $S$ equals $2$, then $S$ is s-decomposable by Proposition~\ref{razval1} applied to the vertex attached to $v_2$ distinct from $v_1$, so we assume that valence of $v_2$ in $S$ equals $3$. Take as $x$ any vertex not joined with $v_2$, and consider decomposition of $S\setminus x$.

Denote by ${\mathsf{B}}$ be a block with $3$ vertices containing $v_2$. Clearly, ${\mathsf{B}}$ is of third or second type. In the first case union of $v_1$ and ${\mathsf{B}}$ has no outlets, so $S$ is s-decomposable by Proposition~\ref{razval1} applied to $x$, thus, we assume that ${\mathsf{B}}$ is of type ${\rm{II}}$. Denote the two remaining vertices of ${\mathsf{B}}$ by $w_1$ and $w_2$. We can assume that $w_1$ joined with $w_2$ in $S$, otherwise $S$ is s-decomposable by Proposition~\ref{razval2} applied to $S=\left\l S_1=\l v_1,v_2\r,b_1=w_1,b_2=w_2,S_2=\right.$ $\left.S\setminus\l S_1,w_1,w_2\r\right\r$.

Consider any decomposition of $S\setminus v_1$. If $\l v_2,w_1,w_2\r$ form one block, then $v_2$ is an outlet of the decomposition, and we get a decomposition of $S$ by gluing a block $\l v_1,v_2\r$ of type $\t{\rm{III}}$. Therefore, two blocks meet at $v_2$, one of them (${\mathsf{B}}_1$) contains $w_1$, and the other (${\mathsf{B}}_2$) $w_2$. Moreover, there is a block ${\mathsf{B}}_3$ containing an edge $\l w_1,w_2\r$. Blocks ${\mathsf{B}}_1$ and ${\mathsf{B}}_2$ are simultaneously of first or second type, and ${\mathsf{B}}_3$ is also of one of these two types. Notice that the only outlet of the union of ${\mathsf{B}}_1$, ${\mathsf{B}}_2$ and ${\mathsf{B}}_3$ is the third vertex of ${\mathsf{B}}_3$ (if any). Therefore, either $S$ is disconnected (if ${\mathsf{B}}_3$ is of type ${\rm{I}}$), or $S$ is s-decomposable by Proposition~\ref{razval1} applied to the third vertex of ${\mathsf{B}}_3$ (otherwise).

\end{proof}

Now, reasoning as in the skew-symmetric case, we complete the proof of Theorem~\ref{g8}.

%%%%%%%%%%%%%%%%%%%%%%%%%%%%%%%%%%%%%%%%%%%%%%%%%%%%%%%%%%%%%%%%%%%%%%%%%%%%%%%%%%%%%%%%%%%%%%%%%%%%%%%%%%%%%%%%%%%%%%%%%%

\end{document}